\documentclass[10pt]{amsart}
\usepackage{amsmath}
\usepackage{amsxtra}
\usepackage{amscd}
\usepackage{amsthm}
\usepackage{amsfonts}
\usepackage{amssymb}
\usepackage{eucal}
\usepackage[all]{xy}
\usepackage{graphicx}

\newtheorem{lem}[subsubsection]{Lemma}
\newtheorem{prop}[subsubsection]{Proposition}

\newtheorem{conj}[subsubsection]{Conjecture}
\newtheorem{thm}[subsubsection]{Theorem}
\newtheorem{qthm}[subsubsection]{Quasi-Theorem}

\theoremstyle{definition}

\theoremstyle{remark}
\newtheorem{rem}[subsubsection]{Remark}

\newcommand{\thmref}[1]{Theorem~\ref{#1}}

\newcommand{\secref}[1]{Sect.~\ref{#1}}
\newcommand{\lemref}[1]{Lemma~\ref{#1}}
\newcommand{\propref}[1]{Proposition~\ref{#1}}

\newcommand{\conjref}[1]{Conjecture~\ref{#1}}

\numberwithin{equation}{section}

\newcommand{\nc}{\newcommand}
\nc{\renc}{\renewcommand}
\nc{\ssec}{\subsection}
\nc{\sssec}{\subsubsection}
\nc{\on}{\operatorname}

\nc\ol{\overline}
\nc\wt{\widetilde}
\nc\tboxtimes{\wt{\boxtimes}}
\nc\tstar{\wt{\star}}
\nc{\alp}{\alpha}

\nc{\ZZ}{{\mathbb Z}}
\nc{\NN}{{\mathbb N}}
\nc{\OO}{{\mathbb O}}
\renc{\SS}{{\mathbb S}}
\nc{\DD}{{\mathbb D}}
\nc{\GG}{{\mathbb G}}

\nc{\Fq}{{\mathbb F}_q}
\nc{\Fqb}{\ol{{\mathbb F}_q}}
\nc{\Ql}{\ol{{\mathbb Q}_\ell}}
\nc{\id}{\text{id}}
\nc\X{\mathcal X}

\nc{\Hom}{\on{Hom}}
\nc{\Lie}{\on{Lie}}
\nc{\Loc}{\on{Loc}}
\nc{\Pic}{\on{Pic}}
\nc{\Bun}{\on{Bun}}
\nc{\IC}{\on{IC}}
\nc{\Aut}{\on{Aut}}
\nc{\rk}{\on{rk}}
\nc{\Sh}{\on{Shv}}
\nc{\Perv}{\on{Perv}}
\nc{\pos}{{\on{pos}}}
\nc{\Conv}{\on{Conv}}
\nc{\Sph}{\on{Sph}}
\nc{\Sym}{\on{Sym}}
\nc{\BunBb}{\overline{\Bun}_B}
\nc{\BunNb}{\overline{\Bun}_N}
\nc{\BunTb}{\overline{\Bun}_T}
\nc{\BunBbm}{\overline{\Bun}_{B^-}}
\nc{\BunBbel}{\overline{\Bun}_{B,el}}
\nc{\BunBbmel}{\overline{\Bun}_{B^-,el}}
\nc{\Buno}{\overset{o}{\Bun}}
\nc{\BunPb}{{\overline{\Bun}_P}}
\nc{\BunBM}{\Bun_{B(M)}}
\nc{\BunBMb}{\overline{\Bun}_{B(M)}}
\nc{\BunPbw}{{\widetilde{\Bun}_P}}
\nc{\BunBP}{\widetilde{\Bun}_{B,P}}
\nc{\GUb}{\overline{G/U}}
\nc{\GUPb}{\overline{G/U(P)}}

\nc{\Hhom}{\underline{\on{Hom}}}
\nc\syminfty{\on{Sym}^{\infty}}
\nc\lal{\ol{\lambda}}
\nc\xl{\ol{x}}
\nc\thl{\ol{\theta}}
\nc\nul{\ol{\nu}}
\nc\mul{\ol{\mu}}
\nc\Sum\Sigma
\nc{\oX}{\overset{o}{X}{}}
\nc{\hl}{\overset{\leftarrow}h{}}
\nc{\hr}{\overset{\rightarrow}h{}}
\nc{\M}{{\mathcal M}}
\nc{\N}{{\mathcal N}}
\nc{\F}{{\mathcal F}}
\nc{\D}{{\mathcal D}}
\nc{\Q}{{\mathcal Q}}
\nc{\Y}{{\mathcal Y}}
\nc{\G}{{\mathcal G}}
\nc{\E}{{\mathcal E}}
\nc{\CalC}{{\mathcal C}}
\nc\Dh{\widehat{\D}}

\nc{\C}{{\mathcal C}}
\nc{\K}{{\mathcal K}}
\renewcommand{\H}{{\mathcal H}}

\nc{\T}{{\mathcal T}}
\nc{\V}{{\mathcal V}}
\renc{\P}{{\mathcal P}}
\nc{\A}{{\mathcal A}}
\nc{\B}{{\mathcal B}}
\nc{\U}{{\mathcal U}}

\nc{\Gr}{{\on{Gr}}}

\nc{\frn}{{\check{\mathfrak u}(P)}}

\nc{\fC}{\mathfrak C}
\nc{\p}{\mathfrak p}
\nc{\q}{\mathfrak q}
\nc\f{{\mathfrak f}}

\nc{\qo}{{\mathfrak q}}
\nc{\po}{{\mathfrak p}}
\nc{\s}{{\mathfrak s}}
\nc\w{\text{w}}

\renewcommand{\mod}{{\on{-mod}}}

\nc\mathi\iota
\nc\Spec{\on{Spec}}
\nc\Mod{\on{Mod}}
\nc{\tw}{\widetilde{\mathfrak t}}
\nc{\pw}{\widetilde{\mathfrak p}}
\nc{\qw}{\widetilde{\mathfrak q}}
\nc{\jw}{\widetilde j}

\nc{\grb}{\overline{\Gr}}
\nc{\I}{\mathcal I}

\nc{\lambdach}{{\check\lambda}}
\nc{\Lambdach}{{\check\Lambda}{}}
\nc{\much}{{\check\mu}}
\nc{\omegach}{{\check\omega}}
\nc{\nuch}{{\check\nu}}
\nc{\etach}{{\check\eta}}
\nc{\alphach}{{\check\alpha}}
\nc{\oblvtach}{{\check\oblvta}}
\nc{\rhoch}{{\check\rho}}
\nc{\ch}{{\check h}}

\nc{\Hb}{\overline{\H}}

\nc{\BA}{{\mathbb{A}}}
\nc{\BC}{{\mathbb{C}}}
\nc{\BE}{{\mathbb{E}}}
\nc{\BG}{{\mathbb{G}}}
\nc{\BM}{{\mathbb{M}}}
\nc{\BO}{{\mathbb{O}}}
\nc{\BD}{{\mathbb{D}}}
\nc{\BL}{{\mathbb{L}}}
\nc{\Bl}{{\mathbb{l}}}
\nc{\BN}{{\mathbb{N}}}
\nc{\BP}{{\mathbb{P}}}
\nc{\BQ}{{\mathbb{Q}}}
\nc{\BR}{{\mathbb{R}}}
\nc{\BZ}{{\mathbb{Z}}}
\nc{\BS}{{\mathbb{S}}}

\nc{\CA}{{\mathcal{A}}}
\nc{\CB}{{\mathcal{B}}}

\nc{\CE}{{\mathcal{E}}}
\nc{\CF}{{\mathcal{F}}}
\nc{\CH}{{\mathcal{H}}}

\nc{\CL}{{\mathcal{L}}}
\nc{\CC}{{\mathcal{C}}}
\nc{\CG}{{\mathcal{G}}}
\nc{\CM}{{\mathcal{M}}}
\nc{\CN}{{\mathcal{N}}}
\nc{\CK}{{\mathcal{K}}}
\nc{\CO}{{\mathcal{O}}}
\nc{\CP}{{\mathcal{P}}}
\nc{\CQ}{{\mathcal{Q}}}
\nc{\CR}{{\mathcal{R}}}
\nc{\CS}{{\mathcal{S}}}
\nc{\CT}{{\mathcal{T}}}
\nc{\CU}{{\mathcal{U}}}
\nc{\CV}{{\mathcal{V}}}
\nc{\CW}{{\mathcal{W}}}
\nc{\CX}{{\mathcal{X}}}
\nc{\CY}{{\mathcal{Y}}}
\nc{\CZ}{{\mathcal{Z}}}
\nc{\CI}{{\mathcal{I}}}
\nc{\cD}{{\mathcal{D}}}

\nc{\csM}{{\check{\mathcal A}}{}}
\nc{\oM}{{\overset{\circ}{\mathcal M}}{}}
\nc{\obM}{{\overset{\circ}{\mathbf M}}{}}
\nc{\oCA}{{\overset{\circ}{\mathcal A}}{}}
\nc{\obA}{{\overset{\circ}{\mathbf A}}{}}
\nc{\ooM}{{\overset{\circ}{M}}{}}
\nc{\osM}{{\overset{\circ}{\mathsf M}}{}}
\nc{\vM}{{\overset{\bullet}{\mathcal M}}{}}
\nc{\nM}{{\underset{\bullet}{\mathcal M}}{}}
\nc{\oD}{{\overset{\circ}{\mathcal D}}{}}
\nc{\obD}{{\overset{\circ}{\mathbf D}}{}}
\nc{\oA}{{\overset{\circ}{\mathbb A}}{}}
\nc{\op}{{\overset{\bullet}{\mathbf p}}{}}
\nc{\cp}{{\overset{\circ}{\mathbf p}}{}}
\nc{\oU}{{\overset{\bullet}{\mathcal U}}{}}
\nc{\oZ}{{\overset{\circ}{\mathcal Z}}{}}
\nc{\ofZ}{{\overset{\circ}{\mathfrak Z}}{}}
\nc{\oF}{{\overset{\circ}{\fF}}}

\nc{\fa}{{\mathfrak{a}}}
\nc{\fb}{{\mathfrak{b}}}
\nc{\fc}{{\mathfrak{c}}}
\nc{\fch}{{\mathfrak{ch}}}
\nc{\fd}{{\mathfrak{d}}}
\nc{\ff}{{\mathfrak{f}}}
\nc{\fg}{{\mathfrak{g}}}
\nc{\fgl}{{\mathfrak{gl}}}
\nc{\fh}{{\mathfrak{h}}}
\nc{\fj}{{\mathfrak{j}}}
\nc{\fl}{{\mathfrak{l}}}
\nc{\fm}{{\mathfrak{m}}}
\nc{\fn}{{\mathfrak{n}}}
\nc{\fu}{{\mathfrak{u}}}
\nc{\fp}{{\mathfrak{p}}}
\nc{\fr}{{\mathfrak{r}}}
\nc{\fs}{{\mathfrak{s}}}
\nc{\ft}{{\mathfrak{t}}}
\nc{\fT}{{\mathfrak{T}}}
\nc{\fz}{{\mathfrak{z}}}
\nc{\fsl}{{\mathfrak{sl}}}
\nc{\hsl}{{\widehat{\mathfrak{sl}}}}
\nc{\hgl}{{\widehat{\mathfrak{gl}}}}
\nc{\hg}{{\widehat{\mathfrak{g}}}}
\nc{\htt}{{\widehat{\mathfrak{t}}}}
\nc{\chg}{{\widehat{\mathfrak{g}}}{}^\vee}
\nc{\hn}{{\widehat{\mathfrak{n}}}}
\nc{\chn}{{\widehat{\mathfrak{n}}}{}^\vee}

\nc{\fA}{{\mathfrak{A}}}
\nc{\fB}{{\mathfrak{B}}}
\nc{\fD}{{\mathfrak{D}}}
\nc{\fE}{{\mathfrak{E}}}
\nc{\fF}{{\mathfrak{F}}}
\nc{\fG}{{\mathfrak{G}}}
\nc{\fK}{{\mathfrak{K}}}
\nc{\fL}{{\mathfrak{L}}}
\nc{\fM}{{\mathfrak{M}}}
\nc{\fN}{{\mathfrak{N}}}
\nc{\fP}{{\mathfrak{P}}}
\nc{\fU}{{\mathfrak{U}}}
\nc{\fV}{{\mathfrak{V}}}
\nc{\fZ}{{\mathfrak{Z}}}

\nc{\bb}{{\mathbf{b}}}
\nc{\bc}{{\mathbf{c}}}
\nc{\bd}{{\mathbf{d}}}
\nc{\bbf}{{\mathbf{f}}}
\nc{\be}{{\mathbf{e}}}
\nc{\bg}{{\mathbf{g}}}
\nc{\bi}{{\mathbf{i}}}
\nc{\bj}{{\mathbf{j}}}
\nc{\bn}{{\mathbf{n}}}
\nc{\bp}{{\mathbf{p}}}
\nc{\bq}{{\mathbf{q}}}
\nc{\bu}{{\mathbf{u}}}
\nc{\bv}{{\mathbf{v}}}
\nc{\bx}{{\mathbf{x}}}
\nc{\bs}{{\mathbf{s}}}
\nc{\by}{{\mathbf{y}}}
\nc{\bw}{{\mathbf{w}}}
\nc{\bA}{{\mathbf{A}}}
\nc{\bK}{{\mathbf{K}}}
\nc{\bB}{{\mathbf{B}}}
\nc{\bC}{{\mathbf{C}}}
\nc{\bG}{{\mathbf{G}}}
\nc{\bD}{{\mathbf{D}}}
\nc{\bH}{{\mathbf{He}}}
\nc{\bM}{{\mathbf{M}}}
\nc{\bN}{{\mathbf{N}}}
\nc{\bO}{{\mathbf{O}}}
\nc{\bV}{{\mathbf{V}}}
\nc{\bW}{{\mathbf{Wh}}}
\nc{\bX}{{\mathbf{X}}}
\nc{\bZ}{{\mathbf{Z}}}
\nc{\bS}{{\mathbf{S}}}
\nc{\bT}{{\mathbf{T}}}

\nc{\sA}{{\mathsf{A}}}
\nc{\sB}{{\mathsf{B}}}
\nc{\sC}{{\mathsf{C}}}
\nc{\sD}{{\mathsf{D}}}
\nc{\sF}{{\mathsf{F}}}
\nc{\sG}{{\mathsf{G}}}
\nc{\sK}{{\mathsf{K}}}
\nc{\sM}{{\mathsf{M}}}
\nc{\sO}{{\mathsf{O}}}
\nc{\sU}{{\mathsf{U}}}
\nc{\sW}{{\mathsf{W}}}
\nc{\sQ}{{\mathsf{Q}}}
\nc{\sP}{{\mathsf{P}}}
\nc{\sZ}{{\mathsf{Z}}}
\nc{\sfp}{{\mathsf{p}}}
\nc{\sfq}{{\mathsf{q}}}
\nc{\sr}{{\mathsf{r}}}
\nc{\sk}{{\mathsf{k}}}
\nc{\su}{{\mathsf{u}}}
\nc{\sv}{{\mathsf{v}}}
\nc{\sg}{{\mathsf{g}}}
\nc{\sff}{{\mathsf{f}}}
\nc{\sfb}{{\mathsf{b}}}
\nc{\sfc}{{\mathsf{c}}}
\nc{\sd}{{\mathsf{d}}}

\nc{\BK}{{\bar{K}}}

\nc{\tA}{{\widetilde{\mathbf{A}}}}
\nc{\tB}{{\widetilde{\mathcal{B}}}}
\nc{\tg}{{\widetilde{\mathfrak{g}}}}
\nc{\tG}{{\widetilde{G}}}
\nc{\TM}{{\widetilde{\mathbb{M}}}{}}
\nc{\tO}{{\widetilde{\mathsf{O}}}{}}
\nc{\tU}{{\widetilde{\mathfrak{U}}}{}}
\nc{\TZ}{{\tilde{Z}}}
\nc{\tx}{{\tilde{x}}}
\nc{\tbv}{{\tilde{\bv}}}
\nc{\tfP}{{\widetilde{\mathfrak{P}}}{}}
\nc{\tz}{{\tilde{\zeta}}}
\nc{\tmu}{{\tilde{\mu}}}

\nc{\urho}{\underline{\rho}}
\nc{\uB}{\underline{B}}
\nc{\uC}{{\underline{\mathbb{C}}}}
\nc{\ui}{\underline{i}}
\nc{\uj}{\underline{j}}
\nc{\ofP}{{\overline{\mathfrak{P}}}}
\nc{\oB}{{\overline{\mathcal{B}}}}
\nc{\og}{{\overline{\mathfrak{g}}}}
\nc{\oI}{{\overline{I}}}

\nc{\eps}{\varepsilon}
\nc{\hrho}{{\hat{\rho}}}

\nc{\one}{{\mathbf{1}}}
\nc{\two}{{\mathbf{t}}}

\nc{\Rep}{{\mathop{\operatorname{\rm Rep}}}}
\nc{\Tot}{{\mathop{\operatorname{\rm Tot}}}}
\nc{\Ker}{{\mathop{\operatorname{\rm Ker}}}}
\nc{\Hilb}{{\mathop{\operatorname{\rm Hilb}}}}
\nc{\End}{{\mathop{\operatorname{\rm End}}}}
\nc{\Ext}{{\mathop{\operatorname{\rm Ext}}}}
\nc{\CHom}{{\mathop{\operatorname{{\mathcal{H}}\it om}}}}
\nc{\GL}{{\mathop{\operatorname{\rm GL}}}}
\nc{\gr}{{\mathop{\operatorname{\rm gr}}}}
\nc{\Id}{{\mathop{\operatorname{\rm Id}}}}
\nc{\de}{{\mathop{\operatorname{\rm def}}}}
\nc{\length}{{\mathop{\operatorname{\rm length}}}}
\nc{\supp}{{\mathop{\operatorname{\rm supp}}}}

\nc{\Cliff}{{\mathsf{Cliff}}}
\nc{\Fl}{\on{Fl}}
\nc{\Fib}{{\mathsf{Fib}}}
\nc{\Coh}{{\on{Coh}}}
\nc{\QCoh}{{\on{QCoh}}}
\nc{\IndCoh}{{\on{IndCoh}}}
\nc{\FCoh}{{\mathsf{FCoh}}}

\nc{\reg}{{\text{\rm reg}}}

\nc{\cplus}{{\mathbf{C}_+}}
\nc{\cminus}{{\mathbf{C}_-}}
\nc{\cthree}{{\mathbf{C}_*}}
\nc{\Qbar}{{\bar{Q}}}
\nc\Eis{{\on{Eis}}}
\nc\Eisb{\ol\Eis{}}
\nc\Eisr{\on{Eis}^{rat}{}}
\nc\wh{\widehat}
\nc{\Def}{\on{Def_{\check{\fb}}(E)}}
\nc{\barZ}{\overline{Z}{}}
\nc{\barbarZ}{\overline{\barZ}{}}
\nc{\barpi}{\overline\pi}
\nc{\barbarpi}{\overline\barpi}
\nc{\barpip}{\overline\pi{}^+}
\nc{\barpim}{\overline\pi{}^-}

\nc{\fq}{\mathfrak q}

\nc{\fqb}{\ol{\fq}{}}
\nc{\fpb}{\ol{\fp}{}}
\nc{\fpr}{{\fp^{rat}}{}}
\nc{\fqr}{{\fq^{rat}}{}}

\nc{\hattimes}{\wh\otimes}

\nc{\bh}{{{\mathbf h}}}
\nc{\bk}{{{\mathbf k}}}
\nc{\bOmega}{{\overline{\Omega(\check \fn)}}}

\nc{\seq}[1]{\stackrel{#1}{\sim}}

\nc{\cT}{{\check{T}}}
\nc{\cG}{{\check{G}}}
\nc{\cM}{{\check{M}}}
\nc{\cB}{{\check{B}}}
\nc{\cP}{{\check{P}}}

\nc{\ct}{{\check{\mathfrak t}}}
\nc{\cg}{{\check{\fg}}}
\nc{\cb}{{\check{\fb}}}
\nc{\cn}{{\check{\fn}}}

\nc{\cLambda}{{\check\Lambda}}

\nc{\cla}{{\check\lambda}}
\nc{\cmu}{{\check\mu}}
\nc{\cnu}{{\check\nu}}
\nc{\ceta}{{\check\eta}}

\nc{\DefbE}{{\on{Def}_{\cB}(E_\cT)}}

\nc{\imathb}{{\ol{\imath}}}
\nc{\rlr}{\overset{\longrightarrow}{\underset{\longrightarrow}\longleftarrow}}

\nc{\oBun}{\overset{\circ}\Bun}
\nc{\LocSys}{\on{LocSys}}
\nc{\BunBbb}{\ol{\ol{Bun}}_B}
\nc{\BunBr}{\Bun_B^{rat}}
\nc{\BunBrp}{\Bun_B^{rat,polar}}
\nc{\BunTrp}{\Bun_T^{rat,polar}}
\nc{\BunNr}{\Bun_N^{rat}}
\nc{\BunNre}{\Bun_N^{enh,rat}}
\nc{\BunTr}{\Bun_T^{rat}}
\nc{\Vect}{\on{Vect}}
\nc{\Whit}{\on{Whit}}
\nc{\CTb}{\ol{\on{CT}}}
\nc{\Ran}{\on{Ran}}
\nc{\CTr}{\on{CT}^{rat}{}}
\nc\jmathr{\jmath^{rat}{}}
\nc{\ux}{\underline{x}}
\nc{\clambda}{{\check\lambda}}
\nc{\calpha}{{\check\alpha}}
\nc{\ind}{{\mathbf{ind}}}
\nc{\oblv}{{\mathbf{oblv}}}
\nc{\coeff}{\on{W-coeff}}
\nc{\Poinc}{\on{Poinc}}
\nc{\Dmod}{\on{D-mod}}
\nc{\dr}{\on{dR}}
\nc{\oCZ}{\overset{\circ}\CZ}
\nc{\KL}{\on{KL}}
\nc{\BFS}{\on{BFS}}
\nc{\triv}{{\mathbf{triv}}}

\nc{\dgSch}{\on{DGSch}}
\nc{\Sch}{\on{Sch}}
\nc{\affdgSch}{\on{DGSch}^{\on{aff}}}
\nc{\affSch}{\on{Sch}^{\on{aff}}}
\nc{\Sing}{\on{Sing}}
\nc{\inftygroup}{\infty\on{-Grpd}}
\renc{\dr}{{\on{dr}}}
\nc\Maps{\on{Maps}}
\nc\Res{\on{Res}}
\nc\bMaps{\mathbf{Maps}}
\nc{\ul}{\underline}
\nc{\bNP}{\mathbf{N(P)}}
\nc{\ofc}{\overset{\circ}\fch}
\nc{\ppart}{(\!(t)\!)}
\nc{\qqart}{[\![t]\!]}
\nc{\crit}{\on{crit}}
\nc{\bDelta}{\mathbf{\Delta}}
\nc{\genB}{{\overset{\on{gen}}\to B}}
\nc{\genP}{{\underset{\on{gen}}\longrightarrow P}}
\nc{\genN}{{\underset{\on{gen}}\longrightarrow N}}
\nc{\semiinf}{{\frac{\infty}{2}}}
\nc{\semiinfi}{{\frac{\infty}{2}+\bullet}}
\nc{\sotimes}{\overset{!}\otimes}

\begin{document}

%\begin{quote}
%{\small ``In jedem Minus steckt ein Plus. Vielleicht habe ich so etwas gesagt, aber man braucht das doch nicht allzu w\"ortlich zu nehmen."}
%\hskip1.3cm {\tiny R.~Musil. Der Mann ohne Eigenschaften.}
%\end{quote}

\vskip1cm

\title[Eisenstein series and quantum groups]{Eisenstein series and quantum groups}

\author{D.~Gaitsgory} 

%\address{Department of Mathematics, Harvard University \\
%1 Oxford street, Cambridge, MA USA}

\dedicatory{To V.~Schechtman, with admiration}

\date{\today}

\begin{abstract}
We sketch a proof of a conjecture of \cite{FFKM} that relates the geometric Eisenstein series sheaf
with semi-infinite cohomology of the small quantum group with coefficients in the tilting module for
the big quantum group.
 \end{abstract}

\maketitle

\tableofcontents

\section*{Introduction}

\ssec{The conjecture}

A mysterious conjecture was suggested in the paper \cite{FFKM}. It tied two objects of very different
origins associated with a reductive group $G$. 

\sssec{}

On the one hand, we consider the \emph{geometric Eisenstein series} sheaf $\Eis_{!*}$, which is an object 
of the derived category of constructible sheaves on $\Bun_G$ for the curve $X=\BP^1$. 
(Here and elsewhere $\Bun_G$ denotes the moduli stack of $G$-bundles on $X$.) 
See \secref{ss:Eis}, where the construction of geometric Eisenstein series
is recalled.  By the Decomposition Theorem, $\Eis_{!*}$ splits as a direct sum of 
(cohomologically shifted) irreducible perverse sheaves. 

\medskip

Now, for a curve $X=\BP^1$, the stack $\Bun_G$ has discretely many isomorphism classes
of points, which are parameterized by dominant coweights of $G$. Therefore, irreducible 
perverse sheaves on $\Bun_G$ are in bijection with dominant coweights of $G$: to each
$\lambda\in \Lambda^+$ we attach the intersection cohomology sheaf $\on{IC}^\lambda$ of the closure
of the corresponding stratum. 

\medskip

In the left-hand side of the conjecture of \cite{FFKM} we consider the (cohomologically graded)
vector space equal to the space of multiplicities of $\on{IC}^\lambda$ in $\Eis_{!*}$. 

\sssec{}

On the other hand, we consider the big and small quantum groups, $\fU_q(G)$ and $\fu_q(G)$, 
attached to $G$, where $q$ is a root of unity of sufficiently high order. To the quantum parameter $q$ one associates
the action of the extended affine Weyl group $W\ltimes \Lambda$ on the weight lattice $\cLambda$,
and using this action, to a dominant coweight $\lambda$ one attaches a particular dominant weight,
denoted $\on{min}_\lambda(0)$; see \secref{ss:q} for the construction. 

\medskip

Consider the \emph{indecomposable titling module} over $\fU_q(G)$ 
of highest weight $\on{min}_\lambda(0)$; denote it $\fT^\lambda_q$. The right-hand side of the conjecture of \cite{FFKM}  
is the \emph{semi-infinite} cohomology of the small quantum group $\fu_q(G)$ with coefficients
in $\fT^\lambda_q|_{\fu_q(G)}$.  

\medskip

The conjecture of \cite{FFKM} says that the above two (cohomologically graded)
vector spaces are canonically isomorphic. Because of the appearing of titling modules, 
the above conjecture acquired a name of the ``Tilting Conjecture".

\sssec{}

In this paper we will sketch a proof of the Tilting Conjecture. The word ``sketch" should be understood
in the following sense. We indicate\footnote{Indicate=explain the main ideas, but far from supplying 
full details.} how to reduce it to two statements that we call ``quasi-theorems",
Quasi-Theorem \ref{t:char KL !*} and Quasi-Theorem \ref{t:BRST and global}. 
These are plausible statements of more general nature, which we hope will turn into actual theorems
soon.  We will explain the content of these quasi-theorems below, see \secref{sss:q-thm 1} and 
\secref{sss:q-thm 2}, respectively.

\sssec{}

This approach to the proof of the Tilting Conjecture is quite involved. It is very possible that if one does not aim for
the more general \conjref{c:main} (described in \secref{ss:approach}), a much shorter (and elementary) argument proving the Tilting Conjecture exists.

\medskip

In particular, in a subsequent publication we will show that the Tilting Conjecture can be obtained as a formal consequence of 
the classical\footnote{Classical=non-quantum.} geometric Langlands conjecture for curves of genus $0$. 

\ssec{Our approach}  \label{ss:approach}

We approach the Tilting Conjecture from the following perspective. Rather than trying to prove the required
isomorphism directly, we first rewrite both sides so that they become amenable to generalization,
and then proceed to proving the resulting general statement, \conjref{c:main}.

\sssec{}

This generalized version of the Tilting Conjecture, i.e., \conjref{c:main}, takes the following form. 
First, our geometric input is a (smooth and complete) curve $X$ of arbitrary genus, equipped with a finite
collection of marked points $x_1,...,x_n$. 
Our representation-theoretic input is a collection 
$\CM_1,...,\CM_n$ of representations of $\fU_q(G)$, so that we think of $\CM_i$ as sitting at $x_i$.

\medskip

Starting with this data, we produce two (cohomologically graded) vector spaces.

\sssec{}  \label{sss:1st procedure}

The first vector space is obtained by combining the following steps. 

\smallskip

\noindent(i) We apply the Kazhdan-Lusztig equivalence 
$$\KL_G:\fU_q(G)\mod\simeq \hg_{\kappa'}\mod^{G(\CO)}$$
to $\CM_1,...,\CM_n$ and convert them
to representations $M_1,...,M_n$ of the Kac-Moody Lie algebra $\hg_{\kappa'}$, where $\kappa'$ is
a negative integral level corresponding to $q$. 

\medskip

(We recall that $\hg_{\kappa'}$ is the central extension of $\fg(\CK)$ equipped with a splitting over $\fg(\CO)$,
with the bracket specified by $\kappa'$. Here and elsewhere $\CO=\BC\qqart$ and $\CK=\BC\ppart$.) 

\medskip

\noindent(ii) Starting with $M_1,...,M_n$, we apply the \emph{localization functor} and obtain 
a $\kappa'$-twisted D-module\footnote{For the duration of the introduction we will ignore the difference between 
the two versions of the derived category of (twisted) D-modules on $\Bun_G$ that occurs because the latter stack is
non quasi-compact.} $\on{Loc}_{G,\kappa',x_1,...,x_n}(M_1,...,M_n)$ on $\Bun_G$. 
Using the fact that $\kappa'$ was integral, we convert $\on{Loc}_{G,\kappa',x_1,...,x_n}(M_1,...,M_n)$ 
to a non-twisted D-module (by a slight abuse of notation we denote it by the same character). 

\smallskip

\noindent(iii) We tensor $\on{Loc}_{G,\kappa',x_1,...,x_n}(M_1,...,M_n)$ with $\Eis_{!*}$
and take its de Rham cohomology on $\Bun_G$. 

\medskip

In \secref{s:localization} we explain that the space of multiplicities appearing in the Tilting Conjecture,
is a particular case of this procedure, when we take $X$ to be of genus $0$, $n=1$ with the module $\CM$ being 
$\fT^\Lambda_q$.   

\medskip

This derivation is a rather straightforward application of the Kashiwara-Tanisaki equivalence
between the (regular block of the) affine category $\CO$ and the category of D-modules
on (the parabolic version) of $\Bun_G$, combined with manipulation of various dualities. 

\sssec{}  \label{sss:2nd procedure}

The second vector space is obtained by combining the following steps. 

\smallskip

\noindent(i) We use the theory of \emph{factorizable sheaves} of \cite{BFS}, thought of as a functor
$$\BFS^{\on{top}}_{\fu_q}:\fu_q(G)\mod\otimes...\otimes \fu_q(G)\mod\to \on{Shv}_{\CG_{q,\on{loc}}}(\Ran(X,\cLambda))$$ 
(here $\Ran(X,\cLambda)$ is the configuration space of $\cLambda$-colored divisors), and attach to 
$\CM_1|_{\fu_q(G)},...,\CM_n|_{\fu_q(G)}$ a (twisted) constructible sheaf\footnote{The 
twisting is given by a canonically defined gerbe over $\Ran(X,\cLambda)$, denoted $\CG_{q,\on{loc}}$.} on $\Ran(X,\cLambda)$, denoted,
$$\BFS^{\on{top}}_{\fu_q}(\CM_1|_{\fu_q(G)},...,\CM_n|_{\fu_q(G)}).$$

\smallskip

\noindent(ii) We apply the direct image functor with respect to the Abel-Jacobi map
$$\on{AJ}:\Ran(X,\cLambda)\to \Pic(X)\underset{\BZ}\otimes \cLambda$$
and obtain a (twisted) sheaf 
\begin{equation}  \label{e:AJ}
\on{AJ}_!(\BFS^{\on{top}}_{\fu_q}(\CM_1|_{\fu_q(G)},...,\CM_n|_{\fu_q(G)}))
\end{equation}
on $\Pic(X)\underset{\BZ}\otimes \cLambda$.

\smallskip

\noindent(iii) We tensor \eqref{e:AJ} with a canonically defined (twisted\footnote{By means of the inverse gerbe, so that 
the tensor product is a usual sheaf, for which it make sense to take cohomology.}) local system
$\CE_{q^{-1}}$ on $\Pic(X)\underset{\BZ}\otimes \cLambda$, 
and take cohomology along  $\Pic(X)\underset{\BZ}\otimes \cLambda$. 

\medskip

In \secref{s:BFS} we explain why the above procedure, applied in the case when $X$ has genus $0$, $n=1$
and $\CM=\fT^\Lambda_q$, recovers the right-hand side of the Tilting Conjecture. 

\medskip

In fact, this derivation
is immediate from one of the main results of the book \cite{BFS} that gives the expression for
the semi-infinite cohomology of $\fu_q(G)$ in terms of the procedure indicated above when $X$ has genus $0$. 

\sssec{}

Thus, \conjref{c:main} states that the two procedures, indicated in Sects. \ref{sss:1st procedure} and \ref{sss:2nd procedure} 
above, are canonically isomorphic as functors
$$\fU_q(G)\mod\times...\times \fU_q(G)\mod\to \Vect.$$

\medskip

The second half of this paper is devoted to the outline of the proof of \conjref{c:main}. As was already mentioned, 
we do not try to give a complete proof, but rather show how to deduce 
\conjref{c:main} from Quasi-Theorems \ref{t:char KL !*} and \ref{t:BRST and global}.

\ssec{KL vs. BFS via BRST}

The two most essential ingredients in the functors in Sects. \ref{sss:1st procedure} and \ref{sss:2nd procedure} are the
Kazhdan-Lusztig equivalence
\begin{equation} \label{e:KL prel}
\KL_G:\fU_q(G)\mod\simeq \hg_{\kappa'}\mod^{G(\CO)}
\end{equation}
(in the case of the former\footnote{Here and elsewhere $\hg_{\kappa'}\mod^{G(\CO)}$ denotes the category of Harish-Chandra modules for
the pair $(\hg_{\kappa'},G(\CO))$. This is the category studied by Kazhdan and Lusztig in the series of papers \cite{KL}.}) 
and the \cite{BFS} construction
\begin{equation} \label{e:BFS prel}
\BFS^{\on{top}}_{\fu_q}:\fu_q(G)\mod\times...\times \fu_q(G)\mod\to \on{Shv}_{\CG_{q,\on{loc}}}(\Ran(X,\cLambda)),
\end{equation} 
(in the case of the latter). 

\medskip

In order to approach \conjref{c:main} we need to understand how these two constructions are related. The precise relationship
is given by Quasi-Theorem \ref{t:char KL !*}, and it goes through \emph{a particular version} of the functor of BRST reduction 
of $\hg_{\kappa'}$-modules with respect to the Lie subalgebra $\fn(\CK)\subset \hg_{\kappa'}$:
$$\on{BRST}_{\fn,!*}:\hg_{\kappa'}\mod^{G(\CO)}\to \htt_{\kappa'}\mod^{T(\CO)},$$ 
introduced\footnote{One actually needs to replace $\htt_{\kappa'}$ by its version that takes into account the critical twist
and the $\rho$-shift, but we will ignore this for the duration of the introduction.}  
in \secref{ss:BRST vs KL}, using the theory of D-modules on \emph{the semi-infinite flag space}.

\sssec{}    \label{sss:q-thm 1} 

Quasi-Theorem \ref{t:char KL !*} is a local assertion, which may be thought of as a characterization of the Kazhdan-Lusztig 
equivalence. It says that the following diagram of functors commutes
$$
\CD 
\fU_q(G)\mod   @>{\KL_G}>>  \hg_{\kappa'}\mod^{G(\CO)} \\
@V{\on{Inv}_{\fu_q(N^+)} \circ \Res^{\on{big}\to \on{small}}}VV    @VV{\on{BRST}_{\fn,!*}}V   \\
\fU_q(T)\mod   @>{\KL_T}>>  \htt_{\kappa'}\mod^{T(\CO)}.
\endCD
$$

In this diagram, $\fU_q(T)\mod$ is the category of representations of the quantum torus, denoted in the main body
of the paper $\Rep_q(T)$.  The functor $\KL_T$ is Kazhdan-Lusztig equivalence for $T$, which is more or less
tautological.   The functor 
$$\on{Inv}_{\fu_q(N^+)} \circ \Res^{\on{big}\to \on{small}}: \fU_q(G)\mod \to \fU_q(T)\mod$$
is the following: we restrict a $\fU_q(G)$-module to $\fu_q(G)$, and then take (derived) invariants with respect to the subalgebra
$\fu_q(N^+)$.

\medskip

Thus, the upshot of Quasi-Theorem \ref{t:char KL !*} is that the Kazhdan-Lusztig equivalences for $G$ and $T$, respectively,
intertwine the functor $\on{BRST}_{\fn,!*}$ and the functor of taking invariants with respect to $\fu_q(N^+)$.

\sssec{}

Let us now explain how Quasi-Theorem \ref{t:char KL !*} allows to relate the functors $\KL_G$ and $\BFS^{\on{top}}_{\fu_q}$.
This crucially relies in the notions of \emph{factorization category}, and of the category over the Ran space, attached to
a given factorization category. We refer the reader to \cite{Ras2} for background on these notions. 

\medskip

First, the equivalence $\KL_T$ (combined with Riemann-Hilbert correspondence) can be viewed as a functor
$$\on{Shv}_{\CG_{q,\on{loc}}}(\Ran(X,\cLambda)) \overset{(\KL_T)_{\Ran(X)}}\longrightarrow (\htt_{\kappa'}\mod^{T(\CO)})_{\Ran(X)},$$
where $(\htt_{\kappa'}\mod^{T(\CO)})_{\Ran(X)}$ is the category over the Ran space attached to $\htt_{\kappa'}\mod^{T(\CO)}$, 
when the latter is viewed as a factorization category. 

\medskip

Second, the functor $\on{BRST}_{\fn,!*}$, viewed as a factorization functor gives rise to a functor
$$(\on{BRST}_{\fn,!*})_{\Ran(X)}:(\hg_{\kappa'}\mod^{G(\CO)})_{\Ran(X)} \to (\htt_{\kappa'}\mod^{T(\CO)})_{\Ran(X)}.$$

\medskip

Now, it is a formal consequence of Quasi-Theorem \ref{t:char KL !*} (and the interpretation of the functor $\BFS^{\on{top}}_{\fu_q}$ via
$\on{Inv}_{\fu_q(N^+)}$ that we explain in \secref{ss:BFS via Koszul} ) that we have the following commutative diagram of categories:
\begin{equation} \label{e:com diag 1}
\CD
\fU_q(G)\mod\times...\times \fU_q(G)\mod     @>{\KL_G}>>   
\hg_{\kappa',x_1}\mod^{G(\CO_{x_1})}\times...\times \hg_{\kappa',x_n}\mod^{G(\CO_{x_n})}  \\
@VV{\Res^{\on{big}\to \on{small}}\times...\times \Res^{\on{big}\to \on{small}}}V @VVV   \\
\fu_q(G)\mod\times...\times \fu_q(G)\mod     & &  (\hg_{\kappa'}\mod^{G(\CO)})_{\Ran(X)}  \\
@VV{\BFS^{\on{top}}_{\fu_q}}V    @VV(\on{BRST}_{\fn,!*})_{\Ran(X)}V    \\
\on{Shv}_{\CG_{q,\on{loc}}}(\Ran(X,\cLambda))  @>{(\KL_T)_{\Ran(X)}}>>   (\htt_{\kappa'}\mod^{T(\CO)})_{\Ran(X)}.
\endCD
\end{equation}

\ssec{Disposing of quantum groups}

We shall now show how to use the commutative diagram \eqref{e:com diag 1} to rewrite \conjref{c:main} as a statement
that is purely algebraic, i.e., one that only deals with D-modules as opposed to constructible sheaves, 
and in particular one that does not involve quantum groups, but only Kac-Moody representations. 

\sssec{}

The commutative diagram \eqref{e:com diag 1} gets us one step closer to the proof of \conjref{c:main}. Namely, it 
gives an interpretation of Step (i) in the procedure of \secref{sss:2nd procedure}
in terms of Kac-Moody algebras. In order to make it possible to compare the entire procedure of  \secref{sss:2nd procedure}
with that of  \secref{sss:1st procedure} we need to give a similar interpretation of Steps (ii) and (iii).

\medskip

This is done by means of combining Riemann-Hilbert correspondence with Fourier-Mukai transform. 
Namely, we claim that we have the following two commutative diagrams. 

\medskip

One diagram is: 
$$
\CD
\on{Shv}_{\CG_{q,\on{loc}}}(\Ran(X,\cLambda))    @>{(\KL_T)_{\Ran(X)}}>>  (\htt_{\kappa'}\mod^{T(\CO)})_{\Ran(X)}     \\
@V{\on{AJ}_!}VV       @VV{\on{Loc}_{T,\kappa',\Ran(X)}}V   \\
\on{Shv}_{\CG_{q,\on{glob}}}(\Pic(X)\underset{\BZ}\otimes \cLambda) @>{\on{FM}\circ \on{RH}}>>  \Dmod_{\kappa'}(\Bun_T).
\endCD
$$
Here $\on{RH}$ stands for the Riemann-Hilbert functor, and the subscript $\CG_{q,\on{glob}}$ stands for an appropriate gerbe
on $\Pic(X)\underset{\BZ}\otimes \cLambda$. The commutativity of this diagram follows from the standard properties of
the Fourier-Mukai transform. 

\medskip

The other diagram is:
$$
\CD
\on{Shv}_{\CG_{q,\on{glob}}}(\Pic(X)\underset{\BZ}\otimes \cLambda) @>{\on{FM}\circ \on{RH}}>>  \Dmod_{\kappa'}(\Bun_T)  \\
@V{-\otimes \CE_{q^{-1}}}VV    @VV{\sim}V   \\
\on{Shv}(\Pic(X)\underset{\BZ}\otimes \cLambda) & & \Dmod(\Bun_T)  \\
@V{\Gamma(\Pic(X)\underset{\BZ}\otimes \cLambda,-)}VV   @VV{\Gamma_\dr(\Bun_T,-)}V   \\
\Vect  @>{\on{Id}}>> \Vect.
\endCD
$$
In this diagram\footnote{In the lower right vertical arrow, as well as elsewhere in the paper,  
the notation $\Gamma_\dr(-,-)$ stands for the functor of de Rham cohomology.} 
$\CE_{q^{-1}}$ is the (twisted) local system from Step (iii) in \secref{sss:2nd procedure}. 

\medskip

The equivalence
$\Dmod_{\kappa'}(\Bun_T)\to \Dmod(\Bun_T)$, appearing in the above diagram, comes from the fact that the twisting $\kappa'$ is \emph{integral}.
The commutativity of the diagram follows from the definition of the (twisted) local system $\CE_{q^{-1}}$. 

\sssec{}

Putting the above two diagrams together with \eqref{e:com diag 1}, we obtain a commutative diagram
\begin{equation} \label{e:com diag 2}
\CD
\fU_q(G)\mod\times...\times \fU_q(G)\mod     @>{\KL_G}>>    
\hg_{\kappa',x_1}\mod^{G(\CO_{x_1})}\times...\times \hg_{\kappa',x_n}\mod^{G(\CO_{x_n})}  \\
@VV{\Res^{\on{big}\to \on{small}}\times...\times \Res^{\on{big}\to \on{small}}}V   @VVV     \\
\fu_q(G)\mod\times...\times \fu_q(G)\mod   & &    (\hg_{\kappa'}\mod^{G(\CO)})_{\Ran(X)}  \\
@VV{\BFS^{\on{top}}_{\fu_q}}V        @VV(\on{BRST}_{\fn,!*})_{\Ran(X)}V    \\
\on{Shv}_{\CG_{q,\on{loc}}}(\Ran(X,\cLambda))  & &   (\htt_{\kappa'}\mod^{T(\CO)})_{\Ran(X)}  \\
@V{\on{AJ}_!}VV     @VV{\on{Loc}_{T,\kappa',\Ran(X)}}V   \\  
\on{Shv}_{\CG_{q,\on{glob}}}(\Pic(X)\underset{\BZ}\otimes \cLambda) & & \Dmod_{\kappa'}(\Bun_T)  \\
@V{-\otimes \CE_{q^{-1}}}VV    @VV{\sim}V   \\
\on{Shv}(\Pic(X)\underset{\BZ}\otimes \cLambda) & & \Dmod(\Bun_T)  \\
@V{\Gamma(\Pic(X)\underset{\BZ}\otimes \cLambda,-)}VV   @VV{\Gamma_\dr(\Bun_T,-)}V   \\
\Vect  @>{\on{Id}}>> \Vect,
\endCD
\end{equation} 
in which the left composed vertical arrow is the procedure of \secref{sss:2nd procedure}. 

\medskip 

Thus, in order to prove \conjref{c:main}, it remains to show the right composed vertical arrow in \eqref{e:com diag 2}
is canonically isomorphic to the composition of Steps (ii) and (iii) in the procedure of \secref{sss:1st procedure}. Recall,
however, that the latter functor involves $\Eis_{!*}$ and thus contains the information about the intersection cohomology 
(a.k.a. IC) sheaf on Drinfeld's compactification $\BunBb$. 

\medskip

Note that, as promised, the latter assertion only involves algebraic objects.

\ssec{Bringing the semi-infinite flag space into the game}

We now outline the remaining steps in the derivation of \conjref{c:main}. 

\sssec{}

In order to compare the right vertical composition in \eqref{e:com diag 2} with the functor 
$$\hg_{\kappa',x_1}\mod^{G(\CO_{x_1})}\times...\times \hg_{\kappa',x_n}\mod^{G(\CO_{x_n})} \to \Vect,$$
given by composing Steps (ii) and (iii) in \secref{sss:1st procedure}, it is convenient to rewrite both sides 
using the notion of dual functor, see \secref{sss:DG category conventions}.

\medskip

Let 
$$\on{CT}_{!*}:\Dmod(\Bun_G)\to \Dmod(\Bun_T)$$
denote the functor \emph{dual} to $\on{Eis}_{!*}$. Let 
$$\on{CT}_{\kappa',!*}:\Dmod_{\kappa'}(\Bun_G)\to \Dmod_{\kappa'}(\Bun_T)$$
denote its $\kappa'$-twisted counterpart (we remind that because the level $\kappa'$ was assumed 
integral, the twisted categories are canonically equivalent to the non-twisted ones). 

\medskip

It then follows formally that the required isomorphism of functors is equivalent to the commutativity of the
next diagram:
\begin{equation} \label{e:com diag 3}
\CD
(\hg_{\kappa'}\mod^{G(\CO)})_{\Ran(X)}  @>{(\on{BRST}_{\fn,!*})_{\Ran(X)}}>>    (\htt_{\kappa'}\mod^{T(\CO)})_{\Ran(X)}  \\ 
@V{\on{Loc}_{G,\kappa',\Ran(X)}}VV    @VV{\on{Loc}_{T,\kappa',\Ran(X)}}V     \\ 
\Dmod_{\kappa'}(\Bun_G)     @>{\on{CT}_{\kappa',!*}}>>   \Dmod_{\kappa'}(\Bun_T)  
\endCD
\end{equation} 

\sssec{}   \label{sss:q-thm 2} 

Now, it turns out that the commutation of the diagram \eqref{e:com diag 3} is a particular case of a more general statement. 

\medskip

In \secref{s:semi-inf} we introduce the category, denoted $\CC{}^{T(\CO)}_{\kappa'}$, to be thought of as the category of
twisted D-modules on the double quotient 
$$N(\CK)\backslash G(\CK)/G(\CO).$$

This is also a factorization category, and we denote by $(\CC{}^{T(\CO)}_{\kappa'})_{\Ran(X)}$ the corresponding 
category over the Ran space. 

\medskip

In \secref{s:semi-inf and loc} we show that to any object 
$$\bc\in (\CC{}^{T(\CO)}_{\kappa'})_{\Ran(X)}$$ 
we can attach a functor
$$\on{CT}_{\kappa',\bc}:\Dmod_{\kappa'}(\Bun_G)\to \Dmod_{\kappa'}(\Bun_T),$$
and also a functor
$$\on{BRST}_{\bc}:(\hg_{\kappa'}\mod^{G(\CO)})_{\Ran(X)}\to (\htt_{\kappa'}\mod^{T(\CO)})_{\Ran(X)}.$$

\medskip

We now have the following statement, Quasi-Theorem \ref{t:BRST and global}, that says that the following diagram is 
commutative for any $\bc$ as above:
\begin{equation} \label{e:com diag 4}
\CD
(\hg_{\kappa'}\mod^{G(\CO)})_{\Ran(X)}  @>{(\on{BRST}_{\fn,\bc})_{\Ran(X)}}>>    (\htt_{\kappa'}\mod^{T(\CO)})_{\Ran(X)}  \\ 
@V{\on{Loc}_{G,\kappa',\Ran(X)}}VV    @VV{\on{Loc}_{T,\kappa',\Ran(X)}}V     \\ 
\Dmod_{\kappa'}(\Bun_G)     @>{\on{CT}_{\kappa',\bc}}>>   \Dmod_{\kappa'}(\Bun_T). 
\endCD
\end{equation} 

\sssec{The IC object on the semi-infinite flag space}

Let us explain how the commutation of the diagram \eqref{e:com diag 4} implies 
the desired commutation of the diagram \eqref{e:com diag 3}. 

\medskip

It turns out that the category $(\CC^{T(\CO)}_{\kappa'})_{\Ran(X)}$ contains a particular object
\footnote{This object actually belongs to a certain completion of $\CC^{T(\CO)}_{\kappa'}$; we ignore this issue in the introduction.},
denoted $\jmath_{\kappa',0,!*}\in (\CC^{T(\CO)}_{\kappa'})_{\Ran(X)}$. It should be thought of as the IC sheaf on the semi-infinite flag space.

\medskip

Now, on the one hand, the $(\on{BRST}_{\fn,!*})_{\Ran(X)}$ is the functor $(\on{BRST}_{\fn,\bc})_{\Ran(X)}$ for 
$\bc=\jmath_{\kappa',0,!*}$ (this is in fact the definition of the functor $(\on{BRST}_{\fn,!*})_{\Ran(X)}$).

\medskip

On the other hand, the object $\jmath_{\kappa',0,!*}$ is closely related to the IC sheaf on $\BunBb$; this relationship
is expressed via the isomorphism 
$$\on{CT}_{\kappa',!*}=\on{CT}_{\kappa',\bc} \text{ for } \bc=\jmath_{\kappa',0,!*}.$$

Thus, taking $\bc=\jmath_{\kappa',0,!*}$ in the diagram \eqref{e:com diag 4}, we obtain the diagram \eqref{e:com diag 3}. 

\ssec{Structure of the paper}

The proof of the Tilting Conjecture, sketched in the main body of the paper, follows the same steps
as those described above, but not necessarily in the same order. We shall now review the contents of this paper, 
section-by-section.  

\sssec{}

In \secref{s:statement} we recall the definition of the geometric Eisenstein series functor, the set-up for quantum groups,
and state the Tilting Conjecture. 

\medskip

At the end of that section we rewrite the space of multiplicities, appearing in the left-hand
side of the Tilting Conjecture as a Hom space from a certain canonically defined object 
$\wt\CP{}^\lambda\in \Dmod(\Bun_G)$ to our Eisenstein series object $\Eis_{!*}$. 

\medskip

In \secref{s:localization} we show that the object $\wt\CP{}^\lambda$, or rather its $\kappa'$-twisted counterpart, 
can be obtained as the localization of a projective object $\CP^\lambda_\kappa$ in the category $\hg_{\kappa}\mod^{G(\CO)}$
(here $\kappa$ is the \emph{positive} level, related to $\kappa'$ by the formula $\kappa=-\kappa_{\on{Kil}}-\kappa'$). 

\medskip

We then perform a duality manipulation and replace $\CHom(\wt\CP{}^\lambda,\Eis_{!*})$ by the cohomology over $\Bun_G$
of the D-module obtained by tensoring the Eisenstein sheaf $\Eis_{!*}$ with the localization (at the negative level $\kappa'$) 
of the tilting object $T^\lambda_{\kappa'}\in \hg_{\kappa'}\mod^{G(\CO)}$. 

\medskip

In \secref{s:dual and Eis} we further rewrite $\CHom(\wt\CP{}^\lambda,\Eis_{!*})$ as the cohomology over $\Bun_T$
of the D-module obtained by applying the \emph{constant term functor} to the localization of $T^\lambda_{\kappa'}$. 
The reason for making this (completely formal) step is that we will eventually generalize \conjref{c:tilting conj} to a 
statement that certain two functors from the category of Kac-Moody representations to $\Dmod(\Bun_T)$ are isomorphic. 

\sssec{}

In \secref{s:BFS} we review the Bezrukavnikov-Finkelberg-Schechtman realization of representations
of the small quantum group as factorizable sheaves. 

\medskip

The \cite{BFS} theory enables us to replace the semi-infinite cohomology appearing in the statement of the Tilting Conjecture
by a certain geometric expression: sheaf cohomology on the space of colored divisors. 

\medskip

In \secref{s:other versions} we give a reinterpretation of the main construction of \cite{BFS}, i.e., the functor
\eqref{e:BFS prel}, as an instance of
Koszul duality. This is needed in order to eventually compare it with the Kazhdan-Lusztig equivalence, i.e., the
functor \eqref{e:KL prel}.

\sssec{}

In \secref{s:KL} we reformulate and generalize the Tilting Conjecture as \conjref{c:main}, which is 
a statement that two particular functors from the category of Kac-Moody representations to $\Vect$
are canonically isomorphic. This reformulation uses the Kazhdan-Lusztig equivalence between 
quantum groups and Kac-Moody representations.

\medskip

We then further reformulate \conjref{c:main} and \conjref{c:deduction} in a way that gets rid of quantum groups altogether,
and compares two functors from the category of Kac-Moody representations to that of twisted D-modules on
$\Bun_T$. 

\medskip

The goal of the remaining sections it to sketch the proof of \conjref{c:deduction}.

\sssec{}

In \secref{s:semi-inf} we introduce the category of D-modules on the semi-infinite flag space and explain how objects
of this category give rise to (the various versions of) the functor of BRST reduction from modules over $\hg_{\kappa'}$
to modules over $\htt_{\kappa'}$. 

\medskip

We then formulate a crucial result, Quasi-Theorem \ref{t:char KL !*} that relates one specific such functor, denoted
$\on{BRST}_{\fn,!*}$, to the functor of $\fu_q(N^+)$-invariants for quantum groups. 

\medskip

In \secref{s:IC on semi-inf} we describe the particular object in the category of D-modules on the semi-infinite flag space
that gives rise to the functor $\on{BRST}_{\fn,!*}$. This is the ``IC sheaf" on the semi-infinite flag space. 

\sssec{}

Finally, in \secref{s:semi-inf and loc}, we show how the functor $\on{BRST}_{\fn,!*}$ interacts with the localization functors
for $G$ and $T$, respectively. In turns out that this interaction is described by the functor of \emph{constant term}
$\on{CT}_{\kappa',!*}$. We show how this leads to the proof of \conjref{c:deduction}.

\ssec{Conventions}

\sssec{}  

Throughout the paper we will be working over the ground field $\BC$. 

\medskip

We let $X$ be an arbitrary smooth projective curve; at some (specified) places in the paper we will 
take $X$ to be $\BP^1$.

\medskip

Given an algebraic group $H$, we denote by $\Bun_H$ the moduli \emph{stack} of 
principal $G$-bundles on $X$. 

\sssec{}  \label{sss:group conventions}

We let $G$ be a reductive group (over $\BC$). We shall assume that the derived group of $G$ is
simply-connected (so that the half-sum of positive roots $\check\rho$ is a weight of $G$). 

\medskip

We let $\Lambda$ denote the \emph{coweight} lattice of $G$, and $\cLambda$ the dual lattice, i.e., the
weight lattice. Let $\Lambda^+\subset \Lambda$ denote the monoid of dominant coweights.
This should not be confused with $\Lambda^{\on{pos}}$, the latter being the monoid
generated by simple coroots. 

\medskip

We denote by $B$ a (fixed) Borel subgroup of $G$ and by $T$ the Cartan quotient of $B$. We let
$N$ denote the unipotent radical of $B$.

\medskip

We let $W$ denote the Weyl group of $G$. 

\sssec{}  \label{sss:DG category conventions} 

This paper does not use derived algebraic geometry, but it does use higher category
theory in an essential way: whenever we say ``category"  we mean a DG category.
We refer the reader to \cite[Sect. 1]{DrGa2}, where the theory of DG categories is reviewed. 

\medskip

In particular, we need the reader to be familiar with the notions of: (i) compactly generated
DG category (see \cite[Sect. 1.2]{DrGa2}); (ii) ind-completion of a given (small) DG category (see \cite[Sect. 1.3]{DrGa2});
(iii) \emph{dual} category and \emph{dual} functor (see \cite[Sect. 1.5]{DrGa2});
(iv) the limit of a diagram of DG categories (see \cite[Sect. 1.6]{DrGa2}).

\medskip

We let $\Vect$ denote the DG category of chain complexes of $\BC$-vector spaces. 

\medskip

Given a DG category $\bC$ and a pair of objects $\bc_1,\bc_2\in \bC$ we let $\CHom(\bc_1,\bc_2)\in \Vect$
denote their Hom complex (this structure embodies the enrichment of every DG category over $\Vect$). 

\medskip

If a DG category $\bC$ is endowed with a t-structure, we denote by $\bC^\heartsuit$ its heart, 
and by $\bC^{\leq 0}$ (resp., $\bC^{\geq 0}$) the connective (resp., coconnective) parts. 

\sssec{}

Some of the geometric objects that we consider transcend the traditional realm of algebraic geometry:
in addition to schemes and algebraic stacks, we will consider arbitrary prestacks.

\medskip

By definition, a prestack is an arbitrary functor 
$$(\affSch_{\on{ft}})^{\on{op}}\to \infty\on{-Groupoids}.$$

A prime example of a prestack that appears in this paper is the Ran space of $X$, denoted $\Ran(X)$.

\sssec{}  \label{sss:D-module conventions}

For a given prestack $\CY$, we will consider the DG category of D-modules on $\CY$, denoted
$\Dmod(\CY)$; whenever we say ``D-module on $\CY$" we mean an object of $\Dmod(\CY)$. 

\medskip

This category is defined as the \emph{limit} of the categories $\Dmod(S)$ over the category of schemes
(of finite type) $S$ over $\CY$. We refer the reader to \cite{GRo} where a comprehensive review of the theory is given.

\medskip

The category of D-modules is contravariantly functorial with respect to the !-pullback: for a morphsim
of prestacks $f:\CY_1\to \CY_2$ we have the functor
$$f^!:\Dmod(\CY_2)\to \Dmod(\CY_1).$$

For a given $\CY$ we let $\omega_\CY$ denote the canonical object of $\Dmod(\CY)$ equal to the !-pillback
of $\BC\in \Vect=\Dmod(\on{pt})$. 

\sssec{}  \label{sss:twistings}

In addition, we will need the notion of \emph{twisting} on a prestack and, given a twisting, 
of the category of twisted D-modules. We refer the reader to \cite[Sects. 6 and 7]{GRo},
where these notions are developed. 

\sssec{}  \label{sss:sheaves}

Given a prestack $\CY$, we will also consider the DG category of \emph{constructible sheaves} on it,
denoted $\on{Shv}(\CY)$; whenever we say ``sheaf on $\CY$" we mean an object of $\on{Shv}(\CY)$. 

\medskip

When $\CY=S$ is a scheme of finite type, we let $\on{Shv}(S)$ be the ind-completion of the 
(standard DG model of the) constructible derived category of sheaves 
in the analytic topology on $S(\BC)$ with $\BC$-coefficients. 

\medskip

For an arbitrary prestack the definition is obtained by passing to the limit over the category
of schemes mapping to it, as in the case of D-modules.  We refer the reader to \cite[Sect. 1]{Tam}
for further details.

\medskip

The usual Riemann-Hilbert correspondence (for schemes) gives rise to the fully faithful embedding
$$\on{Shv}(\CY)\overset{\on{RH}}\longrightarrow \Dmod(\CY).$$

\medskip

The notions of $\BC^*$-gerbe over a prestack and of the DG category of sheaves twisted
by a given gerbe are obtained by mimicking the D-module context of \cite{GRo}. 

\sssec{}

In several places in this paper we mention algebro-geometric objects of infinite type,
such as the loop group $G(\CK)$, where $\CK=\BC\ppart$.  

\medskip

We do \emph{not} consider $\Dmod(-)$ or $\on{Shv}(-)$ on such objects directly. Rather we approximate
them by objects of finite type in a specified way. 

\ssec{Acknowledgements}

It is an honor to dedicate this paper to Vadim Schechtman.  Our central theme--geometric incarnations
of quantum groups--originated in his works \cite{SV1,SV2} and \cite{BFS}. His other ideas, such as
factorization of sheaves and anomalies of actions of infinite-dimensional Lie algebras, 
are also all-pervasive here. 

\medskip

The author learned about the main characters in this paper (such factorizable sheaves and their relation to quantum groups, 
the semi-infinite flag space and its relation to Drinfeld's compactifications, and the Tilting Conjecture) from M.~Finkelberg. 
I would like to thank him for his patient explanations throughout many years.

\medskip

The author would like to express his gratitude to A.~Beilinson for teaching him some of the key notions figuring in this paper
(the localization functors, BRST reduction and Tate extensions associated to it, the Ran space and factorization algebras). 

\medskip

The author would also like to thank S.~Arkhipov, R.~Bezrukavnikov, A.~Braverman, V.~Drinfeld, E.~Frenkel, D.~Kazhdan, J.~Lurie, 
I.~Mirkovic and S.~Raskin for conversations about various objects discussed in this paper. 

\medskip

The author was supported by NSF grant DMS-1063470. 

\section{Statement of the conjecture}  \label{s:statement}

\ssec{Eisenstein series functors}  \label{ss:Eis}

\sssec{}

Let $X$ be a smooth projective curve, and $G$ a reductive group.  We will be concerned with the moduli stack
$\Bun_G$ classifying principal $G$-bundles on $X$, and specifically with the DG category $\Dmod(\Bun_G)$ of
D-modules on $\Bun_G$. 

\medskip

There are (at least) three functors $\Dmod(\Bun_T)\to \Dmod(\Bun_G)$, denoted $\Eis_!$, $\Eis_*$ and $\Eis_{!*}$,
respectively. Let us recall their respective definitions. 

\sssec{}  \label{sss:Eis}

Consider the diagram
\begin{equation} \label{e:basic diagram again}
\xy
(15,0)*+{\Bun_G}="X";
(-15,0)*+{\Bun_T.}="Y";
(0,15)*+{\Bun_B}="Z";
{\ar@{->}^{\sfp} "Z";"X"};
{\ar@{->}_{\sfq} "Z";"Y"};
\endxy
\end{equation}

The functors $\Eis_*$ and $\Eis_!$ are defined to be 
$$\Eis_*(-):=\sfp_*(\sfq^!(\CF)\sotimes \IC_{\Bun_B})[\dim(\Bun_T)]\simeq \sfp_*\circ \sfq^!(\CF)[-\on{dim.rel.}(\Bun_B/\Bun_T)]$$
and
$$\Eis_!(\CF):=\sfp_!(\sfq^!(\CF)\sotimes \IC_{\Bun_B})[\dim(\Bun_T)]\simeq \sfp_!\circ \sfq^*(\CF)[\on{dim.rel.}(\Bun_B/\Bun_T)].$$
respectively. 
%Here $\CL_T$ is the local system on $\Bun_T$, which is \emph{constant} on each connected 
%component, uniquely defined so that the functor $\sfq^!(-\otimes \CL_T)$ commute with Verdier duality. 

\begin{rem}
The isomorphisms inserted into the above formulas are due to the fact that 
the stack $\Bun_B$ is smooth, so $\IC_{\Bun_B}$ is the constant D-module 
$\omega_{\Bun_B}[-\dim(\Bun_B)]$, and the morphism $\sfq$
is smooth as well. The above definition of $\Eis_*$ (resp., $\Eis_!$) differs from that in \cite{DrGa3} by a cohomological shift 
(that depends on the connected component of $\Bun_B$). 
\end{rem} 

\sssec{}

To define the compactified Eisenstein series functor $\Eis_{!*}$, we consider the diagram
\begin{gather}  \label{comp diag}
\xy
(20,0)*+{\Bun_G,}="X";
(-20,0)*+{\Bun_T}="Y";
(0,20)*+{\BunBb}="Z";
(-20,20)*+{\Bun_B}="W";
{\ar@{->}^{\ol\sfp} "Z";"X"};
{\ar@{->}_{\ol\sfq} "Z";"Y"};
{\ar@{^{(}->}^{\jmath} "W";"Z"};  
\endxy
\end{gather}
where $\BunBb$ is stack classifying $G$-bundles, equipped with a \emph{generalized reduction} to $B$;
see \cite[Sect. 1.2]{BG1} for the definition.

\medskip

In the above diagram $\ol\sfp\circ \jmath=\sfp$ and $\ol\sfq\circ \jmath=\sfq$, while the morphism $\ol\sfp$ is proper. We set
$$\Eis_{!*}(\CF)=\ol\sfp_*\left(\ol\sfq^!(\CF)\sotimes \IC_{\BunBb}\right)[\dim(\Bun_T)].$$

\medskip

Note that we can rewrite
$$\Eis_*(\CF)=\ol\sfp_*\left(\ol\sfq^!(\CF)\sotimes \jmath_*(\IC_{\Bun_B})\right)[\dim(\Bun_T)].$$

\begin{rem}
According to \cite[Theorem 5.1.5]{BG1}, the object $\jmath_!(\IC_{\Bun_B})\in \Dmod(\BunBb)$ 
is \emph{universally locally acyclic} with respect to the morphism $\ol\sfq$. This implies that we also have
$$\Eis_!(\CF)=\ol\sfp_*\left(\ol\sfq^!(\CF)\sotimes \jmath_!(\IC_{\Bun_B})\right)[\dim(\Bun_T)].$$

\medskip

The maps
$$\jmath_!(\IC_{\Bun_B})\to \IC_{\BunBb}\to \jmath_*(\IC_{\Bun_B})$$
induce the natural transformations
$$\Eis_!\to \Eis_{!*}\to \Eis_*.$$
\end{rem} 

\ssec{What do we want to study?}

In this subsection we specialize to the case when $X$ is of genus $0$. 

\sssec{}

Recall that in the case of a curve of genus $0$, Grothendieck's classification of
$G$-bundles implies that the stack $\Bun_G$ is stratified
by locally closed substacks $\Bun_G^\lambda$ where $\lambda$ ranges over $\Lambda^+$,
the semi-group of dominant weights. 

\medskip

For $\lambda\in \Lambda^+$, let $\IC^\lambda\in \Dmod(\Bun_G)^\heartsuit$
denote the corresponding irreducible object.

\sssec{}  

Since the morphism $\ol\sfp$ is proper, the Decomposition Theorem implies that the object
$$\Eis_{!*}(\IC_{\Bun_T})=
\Eis_{!*}(\omega_{\Bun_T})[-\dim(\Bun_T)]\simeq \ol\sfp_*(\IC_{\BunBb})\in  \Dmod(\Bun_G)$$
can be written as
\begin{equation} \label{e:V lambda}
\underset{\lambda}\bigoplus\, V^\lambda\otimes \IC^\lambda,\quad V^\lambda\in \Vect.
\end{equation} 

\medskip

The goal is to understand the vector spaces $V^\lambda$, i.e., the multiplicity of each $\IC^\lambda$
in $\Eis_{!*}(\omega_{\Bun_T})$. 

\sssec{}

Below we state a conjecture from \cite[Sect. 7.8]{FFKM} that describes this (cohomologically graded) vector space in terms of the semi-infinite
cohomology of the small quantum group. 

\medskip

As was mentioned in the introduction, the goal of this paper is to sketch a proof of this conjecture. 

\ssec{The ``q"-parameter}  \label{ss:q}

\sssec{}  \label{sss:q}

Let $\on{Quad}(\check\Lambda,\BZ)^W$ be the lattice of integer-valued $W$-invariant quadratic forms of the weight
lattice $\check\Lambda$. 

\medskip

We fix an element 
$$q\in \on{Quad}(\check\Lambda,\BZ)^W\underset{\BZ}\otimes \BC^*.$$

Let $b_q:\check\Lambda\otimes \check\Lambda\to \BC^*$ be the corresponding symmetric bilinear form. 
One should think of $b_q$ as the square of the braiding on the category of representations of the
quantum torus, whose lattice of characters is $\check\Lambda$; in what follows we denote this category by
$\Rep_q(T)$. 

\sssec{}

We will assume that $q$ is torsion. Let $G^\sharp$ be the recipient of Lusztig's quantum Frobenius. I.e.,
this is a reductive group, whose weight lattice is the kernel of $b_q$. 

\medskip

Let $\Lambda^\sharp$ denote the coweight 
lattice of $G^\sharp$, so that at the level of lattices, the quantum Frobenius defines a map
$$\on{Frob}_{\Lambda,q}:\Lambda\to \Lambda^\sharp.$$

\sssec{}

In what follows we will assume that $q$ is such that $G^\sharp$ equals the Langlands dual $\cG$ of $G$. 
In particular, $\Lambda^\sharp\simeq \check\Lambda$, and we can think of the quantum Frobenius as a map 
$$\on{Frob}_{\Lambda,q}:\Lambda\to \check\Lambda.$$

\medskip

Thus, we obtain that the extended affine Weyl group 
$$W_{q,\on{aff}}:=W\ltimes \Lambda$$
acts on $\check\Lambda$, with $\Lambda$ acting via $\on{Frob}_{\Lambda,q}$. (We are considering the ``dotted" action,
so that the fixed point of the action of the finite Weyl group $W$ is $-\check\rho$.) 

\sssec{}  

For $\lambda\in \Lambda$ let $\on{min}_\lambda\in W_{q,\on{aff}}$ be the \emph{shortest} representative in the double
coset of 
$$\lambda\in \Lambda\subset W_{q,\on{aff}}$$
with respect to $W\subset W_{q,\on{aff}}$.

\medskip

Consider the corresponding weight $\on{min}_\lambda(0)\in \check\Lambda$. 

\ssec{Quantum groups}

\sssec{}

Let $$\fU_q(G)\mod \text{ and }\fu_q(G)\mod$$ be the categories of representations of the big (Lusztig's)
and small\footnote{We are considering the graded version of the small quantum group, i.e., we have a forgetful functor
$\fu_q(G)\mod\to \Rep_q(T)$.} quantum groups, respectively, attached to $q$.

\medskip

Consider the \emph{indecomposable tilting module}
$$\fT^\lambda_q\in \fU_q(G)\mod^\heartsuit$$
with highest weight $\on{min}_\lambda(0)$.

\sssec{}

We have the tautological forgetful functor
$$\Res^{\on{big}\to \on{small}}:\fU_q(G)\mod\to \fu_q(G)\mod.$$

Recall now that there is a canonically defined functor
$$\on{C}^\semiinf: \fu_q(G)\mod\to \Vect,$$
see \cite{Arkh}. We have 
$$H^\bullet(\on{C}^\semiinf(\CM))=H^\semiinfi(\CM),\quad \CM\in \fu_q(G)\mod.$$

\begin{rem}  \label{r:sharp}
The functor $\on{C}^\semiinf$ is the functor of semi-infinite cochains with respect to the non-graded version of $\fu_q(G)$.
In particular, its natural target is the category $\Rep(T^\sharp)$ if representations of the Cartan group
$T^\sharp$ of $G^\sharp$. 
\end{rem}

\sssec{}

The following is the statement of the tilting conjecture from \cite{FFKM}:

\begin{conj}  \label{c:tilting conj} For $\lambda\in \Lambda^+$ we have a canonical
isomorphism
\begin{equation} \label{e:tilting eq}
V^\lambda\simeq \on{C}^\semiinf\left(\fu_q(G), \Res^{\on{big}\to \on{small}}(\fT^\lambda_q)\right),
\end{equation}
where $V^\lambda$ is as in \eqref{e:V lambda}. 
\end{conj}

\begin{rem}
According to Remark \ref{r:sharp}, the right-hand side in \eqref{e:tilting eq} is naturally an object of
$\Rep(T^\sharp)$, and since due to our choice of $q$ we have $T^\sharp=\check{T}$, we can view
it as a $\Lambda$-graded vector space. This grading corresponds to the grading on the left hand
side, given by the decomposition of $\Eis_{!*}(\IC_{\Bun_T})$ according to connected components
of $\Bun_T$.
\end{rem}

\begin{rem}
One can strengthen the previous remark as follows: both sides in \eqref{e:tilting eq} carry an action
of the Langlands dual Lie algebra $\check\fg$: on the right-hand side this action comes from the
quantum Frobenius, and on the left-hand side from the action of $\check\fg$ on $\Eis_{!*}(\IC_{\Bun_T})$
from \cite[Sect. 7.4]{FFKM}. One can strengthen the statement of \conjref{c:tilting conj} by requiring that these
two actions be compatible. Although our methods allow to deduce this stronger statement, we will
not pursue it in this paper.
\end{rem}

\begin{rem}
Note that the LHS in \eqref{e:tilting eq} is, by construction, independent of the choice of $q$, whereas
the definition of the RHS explicitly depends on $q$. 

\medskip

However,
one can show (by identifying the regular blocks of the categories $\fU_q(G)\mod$ for different $q$'s), that
the vector space $\on{C}^\semiinf\left(\fu_q(G), \Res^{\on{big}\to \on{small}}(\fT^\lambda_q)\right)$ is also
independent of $q$. 
\end{rem} 

\begin{rem}
For our derivation of the isomorphism \eqref{e:tilting eq} we have to take our ground field to be $\BC$, since it relies
on Riemann-Hilbert correspondence. It is an interesting question to understand whether the resulting isomorphism 
can be defined over $\BQ$ (or some small extension of $\BQ$). 
\end{rem} 

\ssec{Multiplicity space as a Hom}

The definition of the left-hand side in \conjref{c:tilting conj} as a space of multiplicities is not very convenient 
to work with. In this subsection we will rewrite it a certain $\Hom$ space, the latter being more 
amenable to categorical manipulations.  

\sssec{}

Fix a point $x_0\in X$. Let $\Bun_G^{N,x_0}$ (resp., $\Bun_G^{B,x_0}$)
be the moduli of $G$-bundles on $X$, equipped with
a reduction of the fiber at $x_0\in X$ to $N$ (resp., $B$).  Note that $\Bun_G^{N,x_0}$ is equipped with an action of
$T$.  We let
$$\Dmod(\Bun_G^{N,x_0})^{T\on{-mon}}\subset \Dmod(\Bun_G^{N,x_0})$$
denote the full subcategory consisting of $T$-monodromic objects, i.e., the full subcategory generated
by the image of the pullback functor
$$\Dmod(\Bun_G^{B,x_0})=\Dmod(T\backslash (\Bun_G^{N,x_0}))\to \Dmod(\Bun_G^{N,x_0}).$$

\medskip 

Let 
$$\pi:\Bun_G^{N,x_0}\to \Bun_G$$
denote the tautological projection. We consider the resulting pair of adjoint functors
$$\pi_![\dim(G/N)]:\Dmod(\Bun_G^{N,x_0})^{T\on{-mon}}\rightleftarrows \Dmod(\Bun_G):\pi^![-\dim(G/N)].$$

\sssec{}

It is known that for $(X,x_0)=(\BP^1,0)$, the category 
$\Dmod(\Bun_G^{N,x_0})^{T\on{-mon}}$ identifies with the derived DG category of the heart of the
natural t-structure\footnote{This is because the inclusions of the strata are affine morphisms.}
(see \cite[Corollary 3.3.2]{BGS}). 

\medskip

Let $$\CP^\lambda\in \Dmod(\Bun_G^{N,x_0})^\heartsuit$$ denote the projective cover of the irreducible
$\pi^!(\IC^\lambda)[-\dim(G/N)]$. Set
$$\wt\CP^\lambda:=\pi_!(\CP^\lambda)[\dim(G/N)]\in \Dmod(\Bun_G).$$

\medskip

It is clear that if $\CF\in \Dmod(\Bun_G)$ is a semi-simple object equal to
$$\underset{\lambda}\bigoplus \, V^\lambda_\CF\otimes \IC^\lambda,$$
then
$$\CHom(\wt\CP^\lambda,\CF)\simeq \CHom(\CP^\lambda,\pi^!(\CF)[-\dim(G/N)])\simeq V^\lambda_\CF.$$

\sssec{}

Thus, we can restate \conjref{c:tilting conj} as one about the existence of a canonical isomorphism
\begin{equation} \label{e:reform 1}
\CHom(\wt\CP^\lambda,\Eis_{!*}(\IC_{\Bun_T}))\simeq \on{C}^\semiinf\left(\fu_q(G), \Res^{\on{big}\to \on{small}}(\fT^\lambda_q)\right).
\end{equation} 

\section{Kac-Moody representations, localization functors and duality}  \label{s:localization}

\conjref{c:tilting conj} compares an algebraic object (semi-infinite cohomology of the quantum group)
with a geometric one (multiplicity spaces in geometric Eisenstein series). The link between the two will
be provided by the category of representations of the Kac-Moody algebra. 

\medskip

On the one hand, Kac-Moody representations will be related to modules over the quantum group via
the Kazhdan-Lusztig equivalence. On the other hand, they will be related to D-modules on $\Bun_G$
via \emph{localization functors}.

\medskip

In this section we will introduce the latter part of the story: Kac-Moody representations and the localization
functors to $\Dmod(\Bun_G)$. 

\ssec{Passing to twisted D-modules}

In this subsection we will introduce a \emph{twisting} on D-modules into our game. Ultimately, this twisting will 
account for the $q$ parameter in the quantum group via the Kazhdan-Lusztig equivalence. 

\sssec{}  \label{sss:level twisting}

Let $\kappa$ be a \emph{level} for $G$, i.e., a $G$-invariant symmetric bilinear form
$$\fg\otimes \fg\to \BC.$$

To the datum of $\kappa$ one canonically attaches a \emph{twisting} on $\Bun_G$ (resp. $\Bun_G^{N,x_0}$),
see \cite[Proposition-Construction 1.3.6]{KM lecture}.  Let
$$\Dmod_\kappa(\Bun_G^{N,x_0})^{T\on{-mon}} \text{ and } \Dmod_\kappa(\Bun_G)$$
denote the corresponding DG categories of twisted D-modules.

\sssec{}

Suppose now that $\kappa$ is an integral multiple of the Killing form,
$$\kappa=c\cdot \kappa_{\on{Kil}},\quad c\in \BZ.$$

In this case we have canonical equivalences
$$\Dmod_\kappa(\Bun_G)\simeq \Dmod(\Bun_G) \text{ and } 
\Dmod_\kappa(\Bun_G^{N,x_0})^{T\on{-mon}} \simeq \Dmod(\Bun_G^{N,x_0})^{T\on{-mon}},$$
given by tensoring by the $c$-th power of the determinant line bundle on $\Bun_G$, denoted $\CL_{G,\kappa}$, 
and its pullback to $\Bun_G^{N,x_0}$, respectively. 

\medskip

Let $\CP^\lambda_\kappa\in \Dmod_\kappa(\Bun_G^{N,x_0})^{T\on{-mon}}$, 
$$\wt\CP^\lambda_\kappa,  \,\, \Eis_{\kappa,!}(\IC_{\Bun_T}), \,\, 
 \Eis_{\kappa,*}(\IC_{\Bun_T}) \text{ and } \Eis_{\kappa,!*}(\IC_{\Bun_T}) \in  \Dmod_\kappa(\Bun_G),$$
denote the objects that correspond to 
$\CP^\lambda\in \Dmod(\Bun_G^{N,x_0})^{T\on{-mon}}$,
$$\wt\CP^\lambda,  \,\, \Eis_{!}(\IC_{\Bun_T}), \,\, 
\Eis_{*}(\IC_{\Bun_T}) \text{ and } \Eis_{!*}(\IC_{\Bun_T})\in \Dmod(\Bun_G),$$
respectively, under the above equivalences. 

\sssec{}

Hence, we can further reformulate \conjref{c:tilting conj} as one about the existence of a canonical isomorphism
\begin{equation} \label{e:reform 2}
\CHom(\wt\CP^\lambda_\kappa,\Eis_{\kappa,!*}(\IC_{\Bun_T}))\simeq 
\on{C}^\semiinf\left(\fu_q(G), \Res^{\on{big}\to \on{small}}(\fT^\lambda_q)\right).
\end{equation} 

\ssec{Localization functors}   \label{ss:positive level}

In this subsection we will assume that the level $\kappa$ is positive. Here and below by ``positive"
we mean that an each simple factor, $\kappa=c\cdot \kappa_{\on{Kil}}$, where $(c+\frac{1}{2})\notin \BQ^{\leq 0}$,
while the restriction of $\kappa$ to the center of $\fg$ is \emph{non-degenerate}.

\medskip

We are going to introduce a crucial piece of structure, namely, the localization functors from Kac-Moody representations
to (twisted) D-modules on $\Bun_G$. 

\sssec{}

We now choose a point $x_\infty\in X$, different from the point $x_0\in X$ (the latter is one at which we are taking the reduction to $N$). 
We consider the Kac-Moody Lie algebra $\hg_{\kappa,x_\infty}$ \emph{at} $x_\infty$
at level $\kappa$, i.e., the central extension
$$0\to \BC\to \hg_{\kappa,x_\infty}\to \fg(\CK_{x_\infty})\to 0,$$
which is split over $\fg(\CO_{x_\infty})\subset \fg(\CK_{x_\infty})$, and where the bracket is defined using $\kappa$. 

\medskip

Let $\hg_{\kappa,x_\infty}\mod^{G(\CO_{x_\infty})}$ denote the DG category of $G(\CO_{x_\infty})$-integrable 
$\hg_{\kappa,x_\infty}$-modules,
see \cite[Sect. 2.3]{KM lecture}.  

\sssec{}

In what follows we will also consider the Kac-Moody algebra denoted $\hg_{\kappa}$ that we think of as being attached to
the standard formal disc $\CO\subset \CK=\BC\qqart\subset \BC\ppart$. 

\sssec{}

Consider the corresponding localization functors
$$\Loc_{G,\kappa,x_\infty}: \hg_{\kappa,x_\infty}\mod^{G(\CO_{x_\infty})}\to \Dmod_\kappa(\Bun_G)$$ and 
$$\Loc^{N,x_0}_{G,\kappa,x_\infty}: \hg_{\kappa,x_\infty}\mod^{G(\CO_{x_\infty})}\to \Dmod_\kappa(\Bun_G^{N,x_0})^{T\on{-mon}},$$
see \cite[Sect. 2.4]{KM lecture}.

\sssec{}

Assume now that $(X,x_0,x_\infty)=(\BP^1,0,\infty)$. 

\medskip

Since $\kappa$ was assumed positive, the theorem of Kashiwara-Tanisaki (see \cite{KT}) implies 
that the functor $\Loc^{N,x_0}_{G,\kappa,x_\infty}$ defines a t-exact equivalence from the \emph{regular block} of 
$\hg_{\kappa,x_\infty}\mod^{G(\CO_{x_\infty})}$ to $\Dmod_\kappa(\Bun_G^{N,x_0})^{T\on{-mon}}$. 

\medskip 

Let
$$P^\lambda_\kappa\in (\hg_{\kappa,x_\infty}\mod^{G(\CO_{x_\infty})})^\heartsuit$$
denote the object such that 
$$\Loc^{N,x_0}_{\kappa,x_\infty}(P^\lambda_\kappa)\simeq \CP^\lambda_\kappa.$$ 

\sssec{}

Let
$$W_{\kappa,\on{aff}}:=W\ltimes \Lambda.$$

The datum of $\kappa$ defines an action of $W_{\kappa,\on{aff}}$ on the weight lattice $\check\Lambda$,
where $\Lambda$ acts on $\cLambda$ by translations via the map
$$\on{Frob}_{\Lambda,\kappa}:\Lambda\to \cLambda,\quad \lambda\mapsto (\kappa-\kappa_{\on{crit}})(\lambda,-),$$
where $\kappa_{\on{crit}}=-\frac{\kappa_{\on{Kil}}}{2}$.
(Again, we are considering the ``dotted" action,
so that the fixed point of the action of the finite Weyl group $W$ is $-\check\rho$.) 

\medskip

Let $\on{max}_\lambda\in W_{\kappa,\on{aff}}$ be the \emph{longest} representative in the double
coset of 
$$\lambda\in \Lambda\subset W_{\kappa,\on{aff}}$$
with respect to $W\in W_{\kappa,\on{aff}}$.

\medskip

Then the object $P^\lambda_\kappa$ is the projective cover of the irreducible with highest weight $\on{max}_\lambda(0)$. 

\sssec{}

We claim:

\begin{prop} \label{p:loc}
There exists a canonical isomorphism
$$\Loc_{G,\kappa,x_\infty}(P^\lambda_\kappa)\simeq \wt\CP^\lambda_\kappa$$
of objects in $\Dmod_\kappa(\Bun_G)$.
\end{prop}

\begin{proof}

Let 
$$\Gamma_{\kappa,x_\infty} \text{ and  } \Gamma_{\kappa,x_\infty}^{N,x_0}$$
be the functors right adjoint to 
$$\Loc_{G,\kappa,x_\infty} 
\text{ and  } \Loc_{G,\kappa,x_\infty}^{N,x_0},$$
respectively. 

\medskip

Interpreting the above functors as global sections on an approriate scheme (it is the scheme classifying $G$-bundles with a full
level structure at $x_\infty$), one shows that 
$$\Gamma_{\kappa,x_\infty} \simeq \Gamma_{\kappa,x_\infty}^{N,x_0}\circ \pi^![-\dim(G/N)].$$

Passing to the left adjoints, we obtain
$$\Loc_{G,\kappa,x_\infty} \simeq \pi_!\circ \Loc_{G,\kappa,x_\infty}^{N,x_0}[\dim(G/N)],$$
whence the assertion of the proposition.

\end{proof}

\sssec{}

Thus, by \propref{p:loc}, we can reformulate \conjref{c:tilting conj} as the existence of a canonical isomorphism
\begin{equation} \label{e:reform 3}
\CHom(\Loc_{G,\kappa,x_\infty}(P^\lambda_\kappa),\Eis_{\kappa,!*}(\IC_{\Bun_T}))\simeq \\
\simeq \on{C}^\semiinf\left(\fu_q(G), \Res^{\on{big}\to \on{small}}(\fT^\lambda_q)\right),
\end{equation} 
where $\kappa$ is some positive integral level and $(X,x_\infty)=(\BP^1,\infty)$. 

\begin{rem}
From now on we can ``forget" about the point $x_0$ and the stack $\Bun_G^{N,x_0}$. It was only needed 
to reduce \conjref{c:tilting conj} to Equation \eqref{e:reform 3}.
\end{rem} 

\ssec{Duality on Kac-Moody representations}

The goal of this and the next subsection is to replace $\CHom$ in the left-hand side in \eqref{e:reform 3}
by a \emph{pairing}. I.e., we will rewrite the left-hand side in \eqref{e:reform 3} as the value of a certain
\emph{covariant} functor. This interpretation will be important for our next series of manipulations. 

\medskip

We refer the reader to \cite[Sect. 1.5]{DrGa2} for a review of the general theory of duality 
in DG categories. 

\sssec{}

Recall (see \cite[Sects. 2.2 and 2.3 or 4.3]{KM lecture}) that the category $\hg_{\kappa,x_\infty}\mod^{G(\CO_{x_\infty})}$ 
is defined so that it is compactly generated by Weyl modules. 

\medskip

Let $\kappa'$ be the \emph{reflected} level, i.e., $\kappa':=-\kappa-\kappa_{\on{Kil}}$. Note that 
we have $(\kappa_{\on{crit}})'=\kappa_{\on{crit}}$, where we remind that $\kappa_{\on{crit}}=-\frac{\kappa_{\on{Kil}}}{2}$.

\medskip

We recall (see \cite[Sect. 4.6]{KM lecture}) that there exists a canonical equivalence
$$(\hg_{\kappa,x_\infty}\mod^{G(\CO_{x_\infty})})^\vee\simeq \hg_{\kappa',x_\infty}\mod^{G(\CO_{x_\infty})}.$$

\medskip

This equivalence is uniquely characterized by the property that the corresponding pairing
$$\langle-,-\rangle_{\on{KM}}:\hg_{\kappa,x_\infty}\mod^{G(\CO_{x_\infty})}\otimes \hg_{\kappa',x_\infty}\mod^{G(\CO_{x_\infty})}\to \Vect$$
is given by
$$\hg_{\kappa,x_\infty}\mod^{G(\CO_{x_\infty})}\otimes \hg_{\kappa',x_\infty}\mod^{G(\CO_{x_\infty})}\to
\hg_{\kappa,x_\infty}\mod\otimes \hg_{\kappa',x_\infty}\mod \overset{\otimes}\to \hg_{-\kappa_{\on{Kil}},x_\infty}\to\Vect,$$
where 
$$\hg_{-\kappa_{\on{Kil}},x_\infty}\to\Vect$$
is the functor of \emph{semi-infinite cochains} with respect to $\fg(\CK_{x_\infty})$, see \cite[Sect. 4.5 and 4.6]{KM lecture} or \cite[Sect. 2.2]{AG2}. 

\sssec{}

We denote the resulting contravariant equivalence
$$(\hg_{\kappa,x_\infty}\mod^{G(\CO_{x_\infty})})^c\simeq (\hg_{\kappa',x_\infty}\mod^{G(\CO_{x_\infty})})^c$$
by $\BD^{\on{KM}}$, see \cite[Sect. 1.5.3]{DrGa2}. 

\medskip

It has the property that for an object $M\in \hg_{\kappa,x_\infty}\mod^{G(\CO_{x_\infty})}$, induced from a compact
(i.e., finite-dimensional) representation $M_0$ of $G(\CO_{x_\infty})$, the corresponding object 
$$\BD^{\on{KM}}(M)\in \hg_{\kappa',x_\infty}\mod^{G(\CO_{x_\infty})}$$ is one induced from the dual representation
$M_0^\vee$. 

\medskip

In particular, by taking $M_0$ to be an irreducible representation of $G$, so that $M$ is the Weyl module, we obtain that
the functor $\BD^{\on{KM}}$ sends Weyl modules to Weyl modules.

\sssec{}

Assume that $\kappa$ (and hence $\kappa'$ is integral). Let $T^\lambda_{\kappa'}\in \hg_{\kappa',x_\infty}\mod^{G(\CO_{x_\infty})}$ 
be the indecomposable tilting module with highest weight $\on{min}_\lambda(0)$.

\medskip 

We claim:

\begin{prop}  \label{p:tilting and proj}
Let $\kappa$ be positive. Then $\BD^{\on{KM}}(P^\lambda_\kappa)\simeq T^\lambda_{\kappa'}$.
\end{prop}

\begin{proof}
Follows form the fact that the composition
$$(\hg_{\kappa',x_\infty}\mod^{G(\CO_{x_\infty})})^c\to (\hg_{\kappa',x_\infty}\mod^{G(\CO_{x_\infty})})^c
\overset{\BD^{\on{KM}}}\longrightarrow (\hg_{\kappa,x_\infty}\mod^{G(\CO_{x_\infty})})^c,$$
where the first arrow is the contragredient duality at the negative level, identifies with Arkhipov's
functor (the longest intertwining operator), see \cite[Theorem 9.2.4]{AG2}. 
\end{proof}

\ssec{Duality on $\Bun_G$}

Following \cite[Sect. 4.3.3]{DrGa1}, in addition to $\Dmod(\Bun_G)$, one introduces another version of the
category of D-modules on $\Bun_G$, denoted $\Dmod(\Bun_G)_{\on{co}}$.  

\sssec{}

We will not give a detailed review of the definition of $\Dmod(\Bun_G)_{\on{co}}$ here. Let us just say that the difference between
$\Dmod(\Bun_G)$ and $\Dmod(\Bun_G)_{\on{co}}$ has to do with the fact that the stack $\Bun_G$ is not quasi-compact (rather,
its connected components are not quasi-compact). So, when dealing with a fixed quasi-compact open $U\subset \Bun_G$, 
there will not be any difference between the two categories. 

\medskip

One shows $\Dmod(\Bun_G)$ is compactly generated by !-extensions 
of compact objects in $\Dmod(U)$ for $U$ as above, whereas $\Dmod(\Bun_G)_{\on{co}}$ is defined so that it is 
compactly generated by *-extensions of the same objects. 

\medskip

It follows from the construction of $\Dmod(\Bun_G)_{\on{co}}$ that the $\sotimes$ tensor product defines a functor
$$\Dmod(\Bun_G)\otimes \Dmod(\Bun_G)_{\on{co}}\to \Dmod(\Bun_G)_{\on{co}}.$$

\medskip

Again, by the construction of $\Dmod(\Bun_G)_{\on{co}}$, global
de Rham cohomology\footnote{Since $\Bun_G$ is a stack, when we talk about de Rham cohomology, we mean its
renormalized version, see \cite[Sect. 9.1]{DrGa1}; this technical point will not be relevant for the sequel.} 
is a \emph{continuous} functor
$$\Gamma_\dr(\Bun_G,-):\Dmod(\Bun_G)_{\on{co}}\to \Vect.$$

\sssec{}

The usual Verdier duality for quasi-compact algebraic stacks implies that the category $\Dmod(\Bun_G)_{\on{co}}$
identifies with the \emph{dual} of $\Dmod(\Bun_G)$:
\begin{equation} \label{e:duality on BunG}
(\Dmod(\Bun_G))^\vee\simeq \Dmod(\Bun_G)_{\on{co}}.
\end{equation}

\medskip

We can describe the corresponding pairing
$$\langle-,-\rangle_{\Bun_G}:\Dmod(\Bun_G)_{\on{co}}\otimes \Dmod(\Bun_G) \to \Vect$$
explicitly using the functor $\Gamma_\dr(\Bun_G,-)$. 

\medskip

Namely, $\langle-,-\rangle_{\Bun_G}$ equals the composition
$$\Dmod(\Bun_G)_{\on{co}}\otimes \Dmod(\Bun_G) \overset{\sotimes}\longrightarrow 
\Dmod(\Bun_G)_{\on{co}} \overset{\Gamma_\dr(\Bun_G,-)}\longrightarrow \Vect.$$

\medskip

Equivalently, the functor \emph{dual} to $\Gamma_\dr(\Bun_G,-)$ is the functor
$$\Vect\to \Dmod(\Bun_G),\quad \BC\mapsto \omega_{\Bun_G}.$$

\sssec{}  \label{sss:duality twisted}

A similar discussion applies in the twisted case, with the difference that the level gets reflected, i.e., we now have the 
canonical equivalence
\begin{equation} \label{e:duality on BunG kappa}
(\Dmod_\kappa(\Bun_G))^\vee\simeq \Dmod_{\kappa'}(\Bun_G)_{\on{co}}.
\end{equation}

\medskip

The corresponding pairing
$$\langle-,-\rangle_{\Bun_G}:\Dmod_{\kappa'}(\Bun_G)_{\on{co}}\otimes \Dmod_\kappa(\Bun_G)\to \Vect$$
is equal to 
\begin{multline} \label{e:duality on BunG pairing}
\Dmod_{\kappa'}(\Bun_G)_{\on{co}}\otimes \Dmod_\kappa(\Bun_G) \overset{\sotimes}\longrightarrow 
 \Dmod_{-\kappa_{\on{Kil}}}(\Bun_G)_{\on{co}}\simeq \\
\simeq \Dmod(\Bun_G)_{\on{co}} \overset{\Gamma_\dr(\Bun_G,-)}\longrightarrow \Vect,
\end{multline}
where the equivalence $\Dmod_{-\kappa_{\on{Kil}}}(\Bun_G)_{\on{co}}\simeq \Dmod(\Bun_G)_{\on{co}}$ is given by tensoring by the
determinant line bundle $\CL_{G,\kappa_{\on{Kil}}}$. 

\medskip

We denote the resulting contravariant equivalance 
$$(\Dmod_\kappa(\Bun_G))^c\simeq (\Dmod_{\kappa'}(\Bun_G)_{\on{co}})^c$$
by $\BD^{\on{Verdier}}$. 

\sssec{}

Note that when $\kappa$ is integral, the equivalence \eqref{e:duality on BunG kappa} goes over to the non-twisted equivalence
\eqref{e:duality on BunG} under the identifications
$$\Dmod(\Bun_G)\simeq \Dmod_{\kappa}(\Bun_G) \text{ and } \Dmod(\Bun_G)_{\on{co}}\simeq \Dmod_{\kappa'}(\Bun_G)_{\on{co}},$$
given by tensoring by the corresponding line bundles, i.e., $\CL_{G,\kappa}$ and $\CL_{G,\kappa'}$, respectively. 

\ssec{Duality and localization}

In this subsection we assume that the level $\kappa$ is positive (see \secref{ss:positive level} for what this means).

\medskip

We will review how the duality functor on the category of Kac-Moody representations interacts with Verdier duality on
$\Dmod(\Bun_G)$. 

\sssec{}

The basic property of the functor 
$$\Loc_{G,\kappa,x_\infty}: \hg_{\kappa,x_\infty}\mod^{G(\CO_{x_\infty})}\to \Dmod_\kappa(\Bun_G)$$
is that sends compacts to compacts (this is established in \cite[Theorem 6.1.8]{AG2}). 

\medskip

In particular, we obtain that there exists 
a canonically defined continuous functor
$$\Loc_{G,\kappa',x_\infty}:\hg_{\kappa',x_\infty}\mod^{G(\CO_{x_\infty})}\to \Dmod_{\kappa'}(\Bun_G)_{\on{co}},$$
so that
$$\BD^{\on{Verdier}}\circ \Loc_{G,\kappa,x_\infty}\circ \BD^{\on{KM}} \simeq \Loc_{G,\kappa',x_\infty},\quad
(\hg_{\kappa',x_\infty}\mod^{G(\CO_{x_\infty})})^c\to (\Dmod_{\kappa'}(\Bun_G)_{\on{co}})^c.$$

\medskip

The functor $\Loc_{G,\kappa',x_\infty}$ is \emph{localization at the negative level}, and it is explicitly described in
\cite[Corollary 6.1.10]{AG2}.

\begin{rem}

The functor $\Loc_{G,\kappa',x_\infty}$ is closely related to the naive localization functor
$$\Loc^{\on{naive}}_{G,\kappa',x_\infty}:\hg_{\kappa',x_\infty}\mod^{G(\CO_{x_\infty})}\to \Dmod_{\kappa'}(\Bun_G)$$
(the difference is that the target of the latter is the usual category $\Dmod_{\kappa'}(\Bun_G)$ rather than 
$\Dmod_{\kappa'}(\Bun_G)_{\on{co}}$).

\medskip

Namely, for every quasi-compact open substack $U\subset \Bun_G$, the following diagram commutes:
$$
\CD
\hg_{\kappa',x_\infty}\mod^{G(\CO_{x_\infty})}   @>{\on{Id}}>>  \hg_{\kappa',x_\infty}\mod^{G(\CO_{x_\infty})}   \\
@V{\Loc^{\on{naive}}_{G,\kappa',x_\infty}}VV    @VV{\Loc_{G,\kappa',x_\infty}}V    \\
\Dmod_{\kappa'}(\Bun_G)_{\on{co}}  & & \Dmod_{\kappa'}(\Bun_G)  \\
@VVV   @VVV   \\
\Dmod_{\kappa'}(U)  @>{\on{Id}}>> \Dmod_{\kappa'}(U).
\endCD
$$ 

\end{rem}

\sssec{}

Taking into account \propref{p:tilting and proj}, 
we obtain that \conjref{c:tilting conj} can be reformulated as the existence of a canonical isomorphism
\begin{equation} \label{e:reform 4}
\Gamma_\dr(\Bun_G, \Loc_{G,\kappa',x_\infty}(T^\lambda_{\kappa'})\sotimes\Eis_{\kappa,!*}(\IC_{\Bun_T}))
\simeq   \on{C}^\semiinf\left(\fu_q(G), \Res^{\on{big}\to \on{small}}(\fT^\lambda_q)\right),
\end{equation} 
where $\kappa$ is some positive integral level and $(X,x_\infty)=(\BP^1,\infty)$. 

\section{Duality and the Eisenstein functor}  \label{s:dual and Eis}

In this section we will perform a formal manipulation: we will rewrite the left-hand side in \eqref{e:reform 4}
so that instead of the functor $\Gamma_\dr(\Bun_G,-)$ we will consider the functor $\Gamma_\dr(\Bun_T,-)$. 

\ssec{The functor of constant term}

\sssec{}

We consider the stack $\ol\Bun_B$, and we note that it is \emph{truncatable} in the sense of \cite[Sect. 4]{DrGa2}. 
In particular, it makes sense to talk about the category
$$\Dmod(\ol\Bun_B)_{\on{co}}.$$

\medskip

Recall the canonical identifications given by Verdier duality
$$\Dmod(\Bun_G)_{\on{co}}\simeq (\Dmod(\Bun_G))^\vee \text{ and }
\Dmod(\Bun_T)\simeq (\Dmod(\Bun_T))^\vee.$$

We have a similar identification
$$\Dmod(\BunBb)_{\on{co}}\simeq (\Dmod(\BunBb))^\vee.$$

\sssec{}

Under the above identifications the dual of the functor
$$\ol\sfq^!:\Dmod(\Bun_T)\to \Dmod(\BunBb)$$
is the functor $\ol\sfq_*:\Dmod(\BunBb)_{\on{co}}\to \Dmod(\Bun_T)$, 
and the dual of the functor
$$\ol\sfp_*:\Dmod(\BunBb)\to \Dmod(\Bun_G),$$
is the functor $\ol\sfp^!:\Dmod(\Bun_G)_{\on{co}}\to \Dmod(\BunBb)_{\on{co}}$. 

\sssec{}

Consider the functor
$$\Dmod(\ol\Bun_B)_{\on{co}}\otimes \Dmod(\ol\Bun_B)
\to \Dmod(\ol\Bun_B)_{\on{co}}.$$

\medskip 

For a given $\CT\in \Dmod(\ol\Bun_B)$, the resulting functor
$$\CS\mapsto \CT\sotimes \CS:\Dmod(\ol\Bun_B)_{\on{co}}\to (\ol\Bun_B)_{\on{co}}$$
is the dual of the functor
$$\CS'\mapsto \CT\sotimes \CS':\Dmod(\ol\Bun_B)_{\on{co}}\to \Dmod(\ol\Bun_B)_{\on{co}}.$$

\medskip

Hence, we obtain that the dual of the (compactified) Eisenstein functor 
$$\Eis_{!*}:\Dmod(\Bun_T)\to \Dmod(\Bun_G)$$
is the functor $\on{CT}_{!*}:\Dmod(\Bun_G)_{\on{co}}\to \Dmod(\Bun_T)$, defined by 
$$\on{CT}_{!*}(\CF'):=\ol\sfq_*\left(\ol\sfp^!(\CF')\sotimes \IC_{\BunBb}[\dim(\Bun_T)]\right).$$

\ssec{Twistings on $\BunBb$}  \label{sss:twisting G/B}

\sssec{}

Recall the diagram

\begin{equation} \label{e:Eis diagram}
\xy
(15,0)*+{\Bun_G}="X";
(-15,0)*+{\Bun_T.}="Y";
(0,15)*+{\ol\Bun_B}="Z";
{\ar@{->}^{\ol\sfp} "Z";"X"};
{\ar@{->}_{\ol\sfq} "Z";"Y"};
\endxy
\end{equation}

Pulling back the $\kappa$-twisting on $\Bun_G$ by means of $\ol\sfp$, we obtain a twisting on $\ol\Bun_B$; 
we denote the corresponding category
of twisted D-modules $\Dmod_{\kappa,G}(\ol\Bun_B)$. 

\medskip

Pulling back the $\kappa$-twisting on $\Bun_T$ by means of $\ol\sfq$, we obtain 
\emph{another} twisting on $\ol\Bun_B$; we denote the corresponding category
of twisted D-modules $\Dmod_{\kappa,T}(\ol\Bun_B)$.

\medskip

We let $\Dmod_{\kappa,G/T}(\ol\Bun_B)$ the category of D-modules corresponding to the Baer difference 
of these two twistings.  

\medskip

In particular, tensor product gives rise to the functors
\begin{equation} \label{e:tensor twistings}
\Dmod_{\kappa,T}(\ol\Bun_B) \otimes \Dmod_{\kappa,G/T}(\ol\Bun_B)\to \Dmod_{\kappa,G}(\ol\Bun_B)
\end{equation} 
\begin{equation} \label{e:tensor twistings co}
\Dmod_{\kappa',G}(\ol\Bun_B)_{\on{co}} \otimes \Dmod_{\kappa',G/T}(\ol\Bun_B)\to \Dmod_{\kappa',T}(\ol\Bun_B)_{\on{co}}. 
\end{equation} 

\sssec{}

Recall that $\jmath$ denotes the open embedding $\Bun_B\to \ol\Bun_B$. 

\medskip

We note that the category 
$\Dmod_{\kappa,G/T}(\Bun_B)$ canonically identifies with the (untwisted) $\Dmod(\Bun_B)$: indeed the pullback
of the $\kappa$-twisting on $\Bun_T$ by means of $\sfq$ identifies canonically with the pullback of the 
$\kappa$-twisting on $\Bun_G$ by means of $\sfp$, both giving rise to the category $\Dmod_{\kappa}(\Bun_B)$. 

\medskip

Hence, it makes sense to speak of $\IC_{\Bun_B}$ as an object of $\Dmod_{\kappa,G/T}(\Bun_B)$, and of
$$\jmath_{\kappa,*}(\IC_{\Bun_B}),\,\, \jmath_{\kappa,!}(\IC_{\Bun_B}) \text{ and } 
\jmath_{\kappa,!*}(\IC_{\Bun_B})$$
as objects of $\Dmod_{\kappa,G/T}(\ol\Bun_B)$. 

\medskip

Note that when $\kappa$ is integral, under the equivalence 
$$\Dmod_{\kappa,G/T}(\ol\Bun_B)\simeq \Dmod(\ol\Bun_B),$$
given by tensoring by the corresponding line bundle, the
above objects correspond to the objects
$$\jmath_{*}(\IC_{\Bun_B}),\,\, \jmath_{!}(\IC_{\Bun_B}) \text{ and } 
\jmath_{!*}(\IC_{\Bun_B}) \in \Dmod(\ol\Bun_B),$$
respectively. 

\sssec{}

Let us denote by 
$$\on{CT}_{\kappa',!*}:\Dmod_{\kappa'}(\Bun_G)_{\on{co}}\to \Dmod_{\kappa'}(\Bun_T)$$
the functor
$$\CF'\mapsto \ol\sfq_*\left(\ol\sfp^!(\CF')\sotimes \jmath_{\kappa,!*}(\IC_{\Bun_B})[\dim(\Bun_T)]\right).$$

\medskip

Denote also by $\on{CT}_{\kappa',*}$ (resp., $\on{CT}_{\kappa',!}$) the similarly defined 
functors $$\Dmod_{\kappa'}(\Bun_G)_{\on{co}}\to \Dmod_{\kappa'}(\Bun_T),$$ where we replace 
$\jmath_{\kappa,!*}(\IC_{\Bun_B})$ by $\jmath_{\kappa,*}(\IC_{\Bun_B})$ (resp., $\jmath_{\kappa,!}(\IC_{\Bun_B})$).

\sssec{}

Note that the object $\Eis_{\kappa,!*}(\IC_{\Bun_T})\in \Dmod_{\kappa}(\Bun_G)$ that appears in 
\eqref{e:reform 4} is the result of application of the functor
$$\Dmod(\Bun_T) \overset{\Eis_{!*}}\longrightarrow \Dmod(\Bun_G) \overset{-\otimes \CL_{G,\kappa}}\longrightarrow \Dmod_\kappa(\Bun_G)$$
to $\IC_{\Bun_T}\simeq \omega_{\Bun_T}[-\dim(\Bun_T)]$.

\medskip

Hence, the functor
\begin{equation} \label{e:LHS functor}
\CF\mapsto \Gamma_\dr(\Bun_G,\CF\sotimes \Eis_{\kappa,!*}(\IC_{\Bun_T}))
\quad \Dmod_{\kappa'}(\Bun_G)_{\on{co}}\to \Vect
\end{equation} 
identifies with the composition 
\begin{multline} \label{e:CT kappa new}
\Dmod_{\kappa'}(\Bun_G)_{\on{co}}\overset{\on{CT}_{\kappa',!*}[-\dim(\Bun_T)]}\longrightarrow \Dmod_{\kappa'}(\Bun_T)
\overset{-\otimes \CL_{T,-\kappa'}}\longrightarrow \\
\to  \Dmod(\Bun_T) \overset{\Gamma_\dr(\Bun_T,-)}\longrightarrow \Vect.
\end{multline}  

\sssec{}

Summarizing, we obtain that \conjref{c:tilting conj} can be reformulated as the existence of a canonical isomorphism
\begin{multline} \label{e:reform 5}
\Gamma_\dr(\Bun_T, \on{CT}_{\kappa',!*}\circ \Loc_{G,\kappa',x_\infty}(T^\lambda_{\kappa'})\otimes \CL_{T,-\kappa'}) [-\dim(\Bun_T)])
\simeq  \\
\simeq \on{C}^\semiinf\left(\fu_q(G), \Res^{\on{big}\to \on{small}}(\fT^\lambda_q)\right),
\end{multline} 
where $\kappa$ is some positive integral level and $(X,x_\infty)=(\BP^1,\infty)$. 

\ssec{Anomalies}  \label{ss:anomalies}

In the second half of the paper we will study the interaction between the Kac-Moody Lie algebra $\hg_{\kappa'}$ and its counterpart
when $G$ is replaced by $T$. The point is that the passage from $\fg$ to $\ft$ introduces a \emph{critical twist} and a $\rho$-shift.  

\medskip

In this subsection we will specify what we mean by this. 

\sssec{}  \label{sss:shifted KM}

Let us denote by $\htt_{\kappa'}$ the extension of $\ft(\CK)$, given by $\kappa'|_\ft$. I.e., $\htt_{\kappa'}$
is characterized by the property that
$$\fb(\CK)\underset{\ft(\CK)}\times \htt_{\kappa'}\simeq 
\fb(\CK)\underset{\fg(\CK)}\times \hg_{\kappa'},$$
as extensions of $\fb(\CK)$.

\medskip

We let $\htt_{\kappa'+\on{shift}}$ be the Baer sum of $\htt_{\kappa'}$ with the 
\emph{Tate} extension $\htt_{\on{Tate}(\fn)}$, corresponding to $\ft$-representation $\fn$ (equipped with the adjoint action).  

\medskip

We refer the reader to \cite[Sect. 2.7-2.8]{CHA}, where the construction of the Tate extension is explained.

\medskip

There are two essential points of difference that adding $\htt_{\on{Tate}(\fn)}$ introduces:

\smallskip

\noindent(i) The \emph{level} of the extension $\htt_{\kappa'+\on{shift}}$ is no longer $\kappa'|_\ft$, but rather $\kappa'-\kappa_{\on{crit}}|_\ft$.

\smallskip

\noindent(ii) The extension $\htt_{\kappa'+\on{shift}}$ no longer splits over $\ft(\CO)$ in a canonical way. 

\sssec{}  \label{sss:Miura}

According to \cite[Theorem 2.8.17]{CHA}, we can alternatively describe $\htt_{\kappa'+\on{shift}}$ as follows. It is the Baer sum of
$\htt_{\kappa'-\kappa_{\on{crit}}}$ and the \emph{abelian} extension $\htt_{\check\rho(\omega)}$. The latter is by definition the
torsor over
$$(\ft(\CK))^\vee\simeq \ft^\vee\otimes \omega_\CK$$
equal to the push-out of $\check\rho(\omega_{\CK})$, thought of as a $\check{T}(\CK)$-torsor, under the map
$$d\on{log}:\check{T}(\CK)\to \ft^\vee\otimes \omega_\CK.$$

\medskip

In the above formula, $\check{T}$ denotes the torus dual to $T$. 

\begin{rem}
The reason we need to make the modification $\htt_{\kappa'}\rightsquigarrow \htt_{\kappa'+\on{shift}}$
will be explained in Remark \ref{r:anomaly}: it has to do with the properties of the BRST reduction functor. 

\medskip

However, the short answer is that we have no choice: 

\medskip

For example, the fact that the level needs to be shifted is used in the matching of $\kappa$ and $q$ parameters
(see \secref{sss:matching}). The fact that $\htt_{\check\rho(\omega)}$ appears is reflected by the presence of the linear term 
in the definition of the gerbe $\CG_{q,\on{loc}}$ (see \secref{sss:gerbe}), while the latter is forced by the ribbon structure 
on the category of modules over the quantum group.

\end{rem}

\sssec{}

Corresponding to $\htt_{\kappa'+\on{shift}}$ there is a canonically defined twisting on the stack $\Bun_T$; we denote the resulting
category of twisted D-modules by $\Dmod_{\kappa'+\on{shift}}(\Bun_T)$.

\medskip

According to the above description of $\htt_{\kappa'+\on{shift}}$, we can describe the twisting giving rise to $\Dmod_{\kappa'+\on{shift}}(\Bun_T)$
as follows: 

\medskip

It is the Baer sum of the twisting corresponding to $\Dmod_{\kappa'}(\Bun_T)$ (i.e., the twisting attached to the form
$\kappa'|_{\ft}$, see \secref{sss:level twisting}) and the twisting corresponding to the line $\CL_{T,\on{Tate}(\fn)}$ bundle on $\Bun_T$ that 
attaches to a $T$-bundle $\CF_T$ the line
$$\det\, R\Gamma(X,\fn_{\CF_T})^{\otimes -1},$$
where $\fn_{\CF_T}$ is the vector bundle over $X$ associated to the $T$-bundle $\CF_T$ and its representation $\fn$.

\sssec{}

Note, however, that since the difference between the two twistings is given by the line bundle $\CL_{T,\on{Tate}(\fn)}$, the corresponding
categories of twisted D-modules are canonically equivalent via the operation of tensor product with $\CL_{T,\on{Tate}(\fn)}$. 

\medskip

In what follows we will consider the equivalence
\begin{equation} \label{e:shifted equiv}
\Dmod_{\kappa'}(\Bun_T)\to \Dmod_{\kappa'+\on{shift}}(\Bun_T)
\end{equation}
obtained by composing one given by tensoring by the line bundle $\CL_{T,\on{Tate}(\fn)}$ with the cohomological shift 
$[\chi(R\Gamma(X,\fn_{\CF_T}))]$. 

\begin{rem}
Note that $\chi(R\Gamma(X,\fn_{\CF_T}))=-\on{dim.rel.}(\Bun_B/\Bun_T)$.  
As we shall see below (see Remark \ref{r:shift CT}), this is the source of the shift in the definition of the Eisenstein functors, see \secref{sss:Eis}.
\end{rem}

\sssec{}  \label{sss:twist as shift}

One can use \secref{sss:Miura}, to give the following 
alternative description of the twisting giving rise to $\Dmod_{\kappa'+\on{shift}}(\Bun_T)$.

\medskip

It is the Baer sum of the twisting giving rise to $\Dmod_{\kappa'-\kappa_{\on{crit}}}(\Bun_T)$ and one corresponding to
the line bundle on $\Bun_T$, given by
$$\CF_T\mapsto \on{Weil}(\check\rho(\CF_T),\omega_X)\simeq \on{Weil}(\CF_T,\check\rho(\omega_X)),$$
where in the left-hand side $\on{Weil}$ denotes the pairing
$$\on{Pic}\times \on{Pic}\to \BG_m$$ and
in the right-hand side $\on{Weil}$ denotes the induced pairing
$$\Bun_T\times \Bun_{\check{T}}\to \BG_m.$$

\ssec{The level-shifted constant term functor}  

\sssec{}  \label{sss:shift CT}

Let us denote by $\on{CT}_{\kappa'+\on{shift},!*}$ the functor
$$\Dmod_{\kappa'}(\Bun_G)_{\on{co}}\to \Dmod_{\kappa'+\on{shift}}(\Bun_T)$$
equal to the composition
\begin{multline*}
\Dmod_{\kappa'}(\Bun_G)_{\on{co}} \overset{\ol\sfp^!}\longrightarrow \Dmod_{\kappa',G}(\BunBb)_{\on{co}}
\overset{-\sotimes \jmath_{\kappa,!*}(\IC_{\Bun_B})[\dim(\Bun_B)]}\longrightarrow 
\Dmod_{\kappa',T}(\BunBb)_{\on{co}} \to \\ \overset{\ol\sfq_*}\longrightarrow
 \Dmod_{\kappa'}(\Bun_T)\to \Dmod_{\kappa'+\on{shift}}(\Bun_T),
\end{multline*} 
where the last arrow is the functor \eqref{e:shifted equiv}. 

\begin{rem}  \label{r:shift CT}
Note that $\jmath_{\kappa,!*}(\IC_{\Bun_B})[\dim(\Bun_B)]$ is really $\jmath_{\kappa,!*}(\omega_{\Bun_B})$; 
the problem with the latter notation is that it is illegal to apply $\jmath_{!*}$ to objects that are not in the
heart of the t-structure.

\medskip

However, this shows that the functor $\on{CT}_{\kappa'+\on{shift},!*}$ (unlike its counterparts $\on{CT}_{\kappa',!*}$ or
$\on{CT}_{!*}$) does not include any artificial cohomological shifts.

\end{rem}

\sssec{}

Denote by $\on{CT}_{\kappa'+\on{shift},*}$ (resp., $\on{CT}_{\kappa'+\on{shift},!}$) the similarly defined functor,
where we replace $\jmath_{\kappa,!*}(\IC_{\Bun_B}[\dim(\Bun_B)])$ by $\jmath_{\kappa,*}(\omega_{\Bun_B})$
(resp., $\jmath_{\kappa,!}(\omega_{\Bun_B})$). 

\sssec{}  \label{sss:two shifts}

Note that the composition
$$\Dmod_{\kappa'}(\Bun_G)_{\on{co}}  \overset{\on{CT}_{\kappa'+\on{shift},!*}}\longrightarrow 
\Dmod_{\kappa'+\on{shift}}(\Bun_T)\overset{-\otimes \CL_{T,-\kappa'-\on{shift}}}\longrightarrow \Dmod(\Bun_T)$$
identifies with the functor 
$$\Dmod_{\kappa'}(\Bun_G)_{\on{co}}  \overset{\on{CT}_{\kappa',!*}}\longrightarrow 
\Dmod_{\kappa'}(\Bun_T)\overset{-\otimes \CL_{T,-\kappa'}}\longrightarrow \Dmod(\Bun_T),$$
where
\begin{equation} \label{e:L shift}
\CL_{T,-\kappa'}\simeq \CL_{T,-\kappa'-\on{shift}}\otimes \CL_{T,\on{Tate}(\fn)}.
\end{equation}

\sssec{}

Summarizing, we obtain that \conjref{c:tilting conj} can be reformulated as the existence of a canonical isomorphism
\begin{multline} \label{e:reform 6}
\Gamma_\dr(\Bun_T, \on{CT}_{\kappa'+\on{shift},!*}\circ \Loc_{G,\kappa',x_\infty}(T^\lambda_{\kappa'})\otimes 
\CL_{T,-\kappa'-\on{shift}}) [-\dim(\Bun_T)])
\simeq  \\
\simeq \on{C}^\semiinf\left(\fu_q(G), \Res^{\on{big}\to \on{small}}(\fT^\lambda_q)\right),
\end{multline} 
where $\kappa$ is some positive integral level and $(X,x_\infty)=(\BP^1,\infty)$.

%Hence, we obtain that the functor \eqref{e:from big two} appearing in \conjref{c:main} can be rewritten as the composition
%\begin{multline} \label{e:from big two'} 
%\hg_{\kappa',x_1}\mod^{G(\CO_{x_1})}\otimes...\otimes \hg_{\kappa',x_n}\mod^{G(\CO_{x_n})}
%\overset{\Loc_{G,\kappa',x_1,...,x_n}}\longrightarrow \Dmod_{\kappa'}(\Bun_G)_{\on{co}} \to \\
%\overset{\on{CT}_{\kappa'+\on{shift},!*}}\longrightarrow \Dmod_{\kappa'}(\Bun_T) \overset{-\otimes \CL_{T,-\kappa'-\on{shift}}}
%\longrightarrow 
%\Dmod(\Bun_T) \overset{\Gamma_\dr(\Bun_T,-)}\longrightarrow \Vect.$$
%\end{multline}

\section{Digression: factorizable sheaves of [BFS]}  \label{s:BFS}

Our next step in bringing the two sides of \conjref{c:tilting conj} closer to one another is a geometric
interpretation of the category $\fu_q(G)\mod$, and, crucially, of the functor $\on{C}^\semiinf(\fu_q(G),-)$. 

\medskip

This interpretation is provided by the theory of \emph{factorizable sheaves} of \cite{BFS}. 

\ssec{Colored divisors}

Our treatment of factorizable sheaves will be slightly different from that in \cite{BFS}, with the following two main points of difference:

\smallskip

\noindent (i) Instead of considering the various partially symmetrized powers of our curve $X$, we will assemble them 
into an (infinite-dimensional) algebro-geometric object, the colored Ran space $\Ran(X,\cLambda)$ that parameterizes finite 
collections of points of $X$ with elements of $\cLambda$ assigned to them. The fact that we can consider the (DG) category 
of sheaves on such a space is a consequence of recent advances in higher category theory, see \secref{sss:sheaves}.

\smallskip

\noindent (ii) Instead of encoding the quantum parameter $q$ by a local system, we let it be encoded by a 
$\BC^*$-\emph{gerbe} on $\Ran(X,\cLambda)$. 

\sssec{}

Let us recall that the Ran space of $X$, denoted $\Ran(X)$ is a prestack that associates to a test-scheme
$S$ the set of finite non-empty subsets in $\Maps(S,X)$.

\medskip

We let $\Ran(X,\cLambda)$ be the prestack defined as follows. For a test-scheme $S$, we let 
$$\Maps(S,\Ran(X,\cLambda))=\{I\subset \Maps(S,X),\,\, \phi:I\to \cLambda\},$$
where $I$ is a non-empty finite set. 

\sssec{}  \label{sss:neg}

Let 
\begin{equation} \label{e:pos}
\Ran(X,\cLambda)^{\on{neg}}\subset \Ran(X,\cLambda)
\end{equation}
be the subfunctor, corresponding to the subset $\cLambda^{\on{neg}}-0\subset \cLambda$ (here $\cLambda^{\on{neg}}$
is the negative integral span of simple roots). 

\medskip

It is a crucial observation that $\Ran(X,\cLambda)^{\on{neg}}$ is \emph{essentially} a finite-dimensional algebraic variety.

\medskip

For a given $\clambda\in \cLambda^{\on{neg}}$, let $\Ran(X,\cLambda)^{\on{neg},\lambda}$ be the connected
component of $\Ran(X,\cLambda)^{\on{neg}}$ corresponding to those $S$-points
$$\{I\subset \Maps(S,X),\,\, \phi:I\to \cLambda\},$$
for which $\underset{i\in I}\Sigma\, \phi(i)=\clambda$. We have
$$\Ran(X,\cLambda)^{\on{neg}}= \underset{\check\lambda\in \cLambda^{\on{neg}}-0}\bigsqcup\, 
\Ran(X,\cLambda)^{\on{neg},\lambda}.$$

\medskip

Denote 
$$X^{\check\lambda}=\underset{s}\Pi\, X^{(n_s)} \text{ if }  \check\lambda=\Sigma\, n_s\cdot (-\check\alpha_s),$$
where the index $s$ runs through the set of vertices of the Dynkin diagram, and $\check\alpha_s$ denote the
corresponding simple roots.

\medskip

Note that we have a canonically defined map
\begin{equation}  \label{e:Ran graded}
\Ran(X,\cLambda)^{\on{neg},\lambda}\to X^{\check\lambda}.
\end{equation} 

We have:

\begin{lem}  \label{l:Ran graded}
The map \eqref{e:Ran graded} induces an isomorphism 
of sheafifications in the topology generated by finite surjective maps.
\end{lem}

\medskip 

Let $\overset{\circ}X{}^{\check\lambda}\subset X^{\check\lambda}$ be the open subscheme equal to the complement 
of the diagonal divisor. Let 
$$\overset{\circ}\Ran(X,\cLambda){}^{\on{neg},\clambda}%\overset{j}
\hookrightarrow \Ran(X,\cLambda)^{\on{neg}}$$ be the subfunctor
equal to the preimage of 
$$\overset{\circ}X{}^{\check\lambda} \subset  X^{\check\lambda}.$$
It is easy to see that the map
$$\overset{\circ}\Ran(X,\cLambda){}^{\on{neg},\clambda}\to  \overset{\circ}X{}^{\check\lambda}$$
is actually an isomorphism. 

\sssec{}

We have a tautological projection
$$\Ran(X,\cLambda)\to \Ran(X),$$
that remembers the data of $I$.

\medskip

The prestack $\Ran(X,\cLambda)$ has a natural factorization property with respect to the above projection. 
We refer the reader to \cite[Sect. 1]{Ras1} for what this means.

\sssec{}  \label{sss:gerbe}

The next to step is to associate to our choice of the quantum parameter $q$ a certain canonical \emph{factorizable} $\BC^*$-gerbe
on $\Ran(X,\cLambda)$, denoted $\CG_{q,\on{loc}}$. This construction depends on an additional choice:
we need to choose a $W$-invariant symmetric bilinear form $b^{\frac{1}{2}}_q$ on $\cLambda$ with coefficients in $\BC^*$ 
such that
$$q(\lambda)=b^{\frac{1}{2}}_q(\lambda,\lambda).$$ 

Note that by definition
$$(b^{\frac{1}{2}}_q)^2=b_q,$$ where $b_q$ is as in \secref{sss:q}.

\sssec{}

Recall that given a line bundle $\CL$ over a space $\CY$ and an element $a\in \BC^*$, to this data we can attach a 
canonically defined $\BC^*$-gerbe over $\CY$, denoted
$$\CL^{\on{log}(a)}.$$

Namely, the objects of $\CL^{\on{log}(a)}$ are $\BC^*$-local systems on the total space of $\CL-\{0\}$, such that
their monodromy along the fiber is given by $a$. 

\medskip 

The gerbe $\CG_{q,\on{loc}}$ is uniquely characterized by the following requirement. For an $n$-tuple 
$\check\lambda_1,...,\check\lambda_n$ of elements of $\check\Lambda$, and the resulting map
$X^n\to \Ran(X,\cLambda)$, the pullback of $\CG_{q,\on{loc}}$ to $X^n$ is the gerbe
$$\left(\underset{i}\boxtimes\, \omega_X^{\on{log}(b^{\frac{1}{2}}_q(\check\lambda_i,\check\lambda_i+2\check\rho))}\right)^{\otimes -1}\bigotimes
\left(\underset{i\neq j}\otimes\, \CO(\Delta_{i,j})^{\on{log}(b^{\frac{1}{2}}_q(\lambda_i,\lambda_j))}\right).$$

\medskip

In other words, for a point of $\Ran(X,\cLambda)$ given by a collection of pairwise distinct points $x_1,...,x_n$ with assigned weights
$\check\lambda_1,...,\check\lambda_n$,  the fiber of $\CG_{q,\on{loc}}$ over this point is the tensor product
$$\underset{i=1,...,n}\bigotimes\,  (\omega^{\otimes -1}_{x_i})^{\on{log}(q(\check\lambda_i+\check\rho))-\on{log}(q(\check\rho))}.$$ 

\begin{rem}
As can be seen from the above formula, the individual fibers of the above gerbe do not depend on the additional
datum of $b^{\frac{1}{2}}_q$; the latter is needed in order to make the gerbe $\CG_{q,\on{loc}}$ well-defined on
all of $\Ran(X,\cLambda)$.
\end{rem} 

\begin{rem}
As can be observed from either of the above descriptions of $\CG_{q,\on{loc}}$, it naturally arises as a tensor product of two
gerbes: one comes from just the quadratic part $\lambda\mapsto q(\check\lambda)$, and the other from the linear part
$\lambda\mapsto b_q(\check\lambda,\check\rho)$. The quadratic part encodes the ``true" quantum parameter for $\cLambda$,
whereas the linear part is the ``$\rho$-shift" that we will comment on in Remark \ref{r:rho for q}. 

\medskip

Note also that for each (local) trivialization of the canonical line bundle $\omega_X$, we obtain a trivialization of the 
linear part of the gerbe over the corresponding open sub-prestack of $\Ran(X,\cLambda)$. 

\end{rem}

\sssec{}   \label{sss:twisted shv Ran}

Recall now that if $\CY$ is a topological space (resp., prestack) equipped with a $\BC^*$-gerbe $\CG$, we can consider the category
$$\on{Shv}_\CG(\CY)$$
of sheaves on $\CY$ twisted by $\CG$. 

\medskip

We let
$$\on{Shv}_{\CG_{q,\on{loc}}}(\Ran(X,\cLambda))$$
the category of $\CG_{q,\on{loc}}$-twisted sheaves on $\Ran(X,\cLambda)$. 

\medskip

The factorization property of $\CG_{q,\on{loc}}$ over $\Ran(X)$ implies that it makes sense to talk about factorization
algebras in $\on{Shv}_{\CG_{q,\on{loc}}}(\Ran(X,\cLambda))$, and for a given factorization algebra, about
factorization modules over it, see \cite[Sect. 6]{Ras3}.

\ssec{The factorization algebra of [BFS]}

\sssec{}

The basic property of $\CG_{q,\on{loc}}$ is that its restriction to $\overset{\circ}\Ran(X,\cLambda){}^{\on{neg}}$ is canonically
trivialized. This follows from the fact that for a simple root $\check\alpha_i$, we have
$$b^{\frac{1}{2}}_q(-\check\alpha_i,-\check\alpha_i+2\check\rho)=\frac{q(\check\rho-\check\alpha_i)}{q(\check\rho)}=
1\in \BC^*,$$
since $s_i(\check\rho)=\check\rho-\check\alpha_i$ and $q$ is $W$-invariant.  

\medskip

Therefore the category $\on{Shv}_{\CG_{q,\on{loc}}}(\overset{\circ}\Ran(X,\cLambda){}^{\on{neg}})$ identifies canonically with 
the non-twisted
category $\on{Shv}(\overset{\circ}\Ran(X,\cLambda){}^{\on{neg}})$. In particular, we can consider the sign local system 
$sign$ on $\overset{\circ}\Ran(X,\cLambda){}^{\on{neg}}$
as an object of $\on{Shv}_{\CG_{q,\on{loc}}}(\overset{\circ}\Ran(X,\cLambda){}^{\on{neg}})$. 

\sssec{}

We define the object
$$\Omega^{\on{small}}_q\in \on{Shv}_{\CG_{q,\on{loc}}}(\Ran(X,\cLambda){}^{\on{neg}})$$ as follows. 

First, we note that \lemref{l:Ran graded} implies that for the purposes of considering (twisted) sheaves,
we can think that $\Ran(X,\cLambda){}^{\on{neg}}$ is an an algebraic variety. Now, we let $\Omega^{\on{small}}_q$
be the Goresky-MacPherson extension of 
$sign\in \on{Shv}_{\CG_{q,\on{loc}}}(\overset{\circ}\Ran(X,\cLambda){}^{\on{neg}})$, 
cohomologically shifted so that it lies in the heart of the perverse t-structure.

\medskip

We shall regard $\Omega^{\on{small}}_q$ as an object of $\on{Shv}_{\CG_{q,\on{loc}}}(\Ran(X,\cLambda))$ via the 
embedding \eqref{e:pos}.  It follows from the construction that $\Omega^{\on{small}}_q$ has a natural structure of
\emph{factorization algebra} in $\on{Shv}_{\CG_{q,\on{loc}}}(\Ran(X,\cLambda))$. 

\sssec{}

Given points $x_1,...,x_n\in X$, we let
$$\Omega^{\on{small}}_q\mod_{x_1,...,x_n}$$
denote the category of \emph{factorization $\Omega^{\on{small}}_q$-modules} at the above points.

\begin{rem}
In the terminology of \cite{BFS}, the category $\Omega^{\on{small}}_q\mod_{x_1,...,x_n}$ is referred to as
the category of \emph{factorable sheaves}.
\end{rem}

\sssec{}

The main construction of \cite{BFS} says that there is an equivalence
\begin{equation} \label{e:BFS !*}
(\fu_q(G)\mod)_{T_{x_1}(X)}\otimes...\otimes (\fu_q(G)\mod)_{T_{x_n}(X)}\to  
\Omega^{\on{small}}_q\mod_{x_1,...,x_n},
\end{equation} 
where for a one-dimensional $\BC$-vector space $\ell$, we denote by
$$(\fu_q(G)\mod)_\ell$$
the twist of $\fu_q(G)\mod$ by $\ell$ using the auto-equivalence, given by the ribbon structure. Here $T_x(X)$ denotes the tangent
line to $X$ at $x\in X$.  

\medskip

We will reinterpret the construction of the functor \eqref{e:BFS !*} in \secref{ss:BFS via Koszul}.

\sssec{}   \label{sss:BFS functor}

We denote the resulting functor 
$$(\fu_q(G)\mod)_{T_{x_1}(X)}\otimes...\otimes (\fu_q(G)\mod)_{T_{x_n}(X)}\to \Omega^{\on{small}}_q\mod_{x_1,...,x_n}\to 
\on{Shv}_{\CG_{q,\on{loc}}}(\Ran(X,\cLambda)),$$
where the last arrow is a forgetful functor, by $\BFS^{\on{top}}_{\fu_q}$. 

\ssec{Conformal blocks}

In this section we will generalize a procedure from \cite{BFS} that starts with $n$ modules over $\fu_q$, thought of as placed
at points $x_1,...,x_n$ on $X$, and produces an object of $\Vect$. 

\medskip

Unlike the functor \eqref{e:BFS !*}, this construction will be of a \emph{global} nature, in that it
will involve taking cohomology over the stack $\on{Pic}(X)\underset{\BZ}\otimes \cLambda$. 

\sssec{}  \label{sss:AJ}

Let
$$\on{AJ}:\Ran(X,\cLambda)\to \on{Pic}(X)\underset{\BZ}\otimes \cLambda$$
denote the Abel-Jacobi map
$$\{x_i,\clambda_i\}\mapsto \underset{i}\Sigma\, \CO(-x_i)\otimes \clambda_i\in 
\in \on{Pic}(X)\underset{\BZ}\otimes \cLambda.$$

\sssec{}

A basic property of $\CG_{q,\on{loc}}$ is that it canonically descends to a $\BC^*$-gerbe on 
$\on{Pic}(X)\underset{\BZ}\otimes \cLambda$.  We shall denote the latter by $\CG_{q,\on{glob}}$. 

\medskip

Specifically, this gerbe attaches to a point $\underset{i}\Sigma\, \CL_i\otimes \clambda_i\in \on{Pic}(X)\underset{\BZ}\otimes \cLambda$
(where $\CL_i$ are line bundles on $X$) the $\BC^*$-gerbe
\begin{multline*}
\left(\underset{i}\otimes\, \on{Weil}(\CL_i,\CL_i)^{\on{log}(b_q^{\frac{1}{2}}(\clambda_i,\clambda_i))}\right) \bigotimes 
\left(\underset{i\neq j}\otimes\, \on{Weil}(\CL_i,\CL_j)^{\on{log}(b_q^{\frac{1}{2}}(\clambda_i,\clambda_j))}\right) \bigotimes  \\
\bigotimes \left(\underset{i}\otimes\, \on{Weil}(\CL_i,\omega_X)^{\on{log}(b_q^{\frac{1}{2}}(\clambda_i,2\check\rho))}\right),
\end{multline*} 
where
$$\on{Weil}:\on{Pic}\times \on{Pic}\to BG_m$$
is the Weil pairing. 

\sssec{}  \label{sss:E}

There exists a canonically defined ($\CG^{-1}_{q,\on{glob}}$-twisted) local system
$$\CE_{q^{-1}}\in \on{Shv}_{\CG^{-1}_{q,\on{glob}}}(\on{Pic}(X)\underset{\BZ}\otimes \cLambda),$$
which is supported on the union of the connected components corresponding to
$$-(2g-2)\check\rho+\on{Im}(\on{Frob}_{\Lambda,q})\subset \cLambda.$$

\medskip

We will specify what $\CE_{q^{-1}}$ is in \secref{sss:Heis} in terms of the Fourier-Mukai transform. 

\begin{rem}
One can show that the restriction of $\CE_{q^{-1}}$ to the connected component $-(2g-2)\check\rho$ identifies with the
Heisenberg local system of \cite{BFS}. 
\end{rem} 

\sssec{}

Let us specialize for a moment to the case when $X=\BP^1$.  Choosing $x_\infty$ as our base point, we obtain an isomorphism 
$$\on{Pic}(X)\underset{\BZ}\otimes \cLambda\simeq BG_m\times \cLambda.$$

In this case, the gerbe $\CG_{q,\on{glob}}$ is trivial (and hence supports non-zero objects of the category
$\on{Shv}_{\CG_{q,\on{glob}}}(\on{Pic}(X)\underset{\BZ}\otimes \cLambda)$) only on the connected components corresponding to 
$$\check\rho+\on{Im}(\on{Frob}_{\Lambda,q})\subset \cLambda.$$
Moreover, the choice of the base point $x_\infty$ defines a preferred trivialization of $\CG_{q,\on{glob}}$ on the above connected
components. 

\medskip

The following be a corollary of the construction:

\begin{lem} \label{l:Heis genus 0}
With respect to the trivialization of the gerbe $\CG_{q,\on{glob}}$ on each of the above connected components of 
$\on{Pic}(X)\underset{\BZ}\otimes \cLambda$, the twisted local system $\CE_{q^{-1}}$ identifies with
$\omega_{BG_m}$.
\end{lem} 

\sssec{}   \label{sss:conf}

For an $n$-tuple of points $x_1,...,x_n\in X$ we consider the functor
$$(\fu_q(G)\mod)_{T_{x_1}(X)}\otimes...\otimes (\fu_q(G)\mod)_{T_{x_n}(X)}\to \Vect$$
equal to the composition 
\begin{multline*}
(\fu_q(G)\mod)_{T_{x_1}(X)}\otimes...\otimes (\fu_q(G)\mod)_{T_{x_n}(X)}\overset{\BFS^{\on{top}}_{\fu_q}}\longrightarrow
\on{Shv}_{\CG_{q,\on{loc}}}(\Ran(X,\cLambda)) \overset{\on{AJ}_!}\longrightarrow \\
\to \on{Shv}_{\CG_{q,\on{glob}}} (\on{Pic}(X)\underset{\BZ}\otimes \cLambda)
\overset{-\sotimes \CE_{q^{-1}}[-\dim(\Bun_T)]}\to \on{Shv}(\on{Pic}(X)\underset{\BZ}\otimes \cLambda) \to \Vect,
\end{multline*}
where the last arrow is the functor of sheaf cohomology. 

\medskip

We denote this functor by $\on{Conf}^{\fu_q}_{X;x_1,...,x_n}$. 

\begin{rem}
A version of this functor, when instead of all of
$\on{Pic}(X)\underset{\BZ}\otimes \cLambda$ we use its connected component corresponding to 
$-(2g-2)\check\rho$ is the functor of conformal blocks of \cite{BFS}. 
\end{rem}

\sssec{}

Assume now that $X=\BP^1$, $n=1$ and $x_1=\infty$. We obtain a functor
$$\on{Conf}^{\fu_q}_{\BP^1;\infty}: \fu_q(G)\mod\to \Vect.$$

According to \cite[Theorem IV.8.11]{BFS}, we have the following:

\begin{thm}  \label{t:semi-inf}
There exists a canonical isomorphism of functors $\fu_q(G)\mod\to \Vect$
$$\on{Conf}^{\fu_q}_{\BP^1;\infty}\simeq \on{C}^\semiinf.$$
\end{thm}

\section{Digression: quantum groups and configuration spaces} \label{s:other versions}

\ssec{The construction of \cite{BFS} via Koszul duality}    \label{ss:BFS via Koszul}

In this subsection we will show how the functor \eqref{e:BFS !*} can be interpreted as the Koszul
duality functor for the Hopf algebra $\fu_q(N^+)$ in the braided monoidal category $\Rep_q(T)$. 
Such an interpretation is crucial for strategy of the proof of the isomorphism \eqref{e:reform 6}.

\sssec{}  \label{sss:E2}

For the material of this subsection we will need to recall the following constructions, 
essentially contained in \cite[Sect. 5.5]{Lu}:

\medskip

\noindent{(i)} Given a braided monoidal category $\CC$, we can canonically attach to it a category $\CC_{\Ran(X)}$
over the Ran space of the curve $X=\BA^1$.

\medskip

\noindent{(i')} If $\CC$ is endowed with a \emph{ribbon structure}, we can extend this construction
and replace $\BA^1$ by an arbitrary algebraic curve $X$.

\medskip

\noindent{(ii)} For a given monoidal category $\CC$ it make sense to talk about associative (a.k.a. $\BE_1$)
algebras in $\CC$.  If $\CC$ is \emph{braided monoidal}, we can talk about $\BE_2$-algebras in $\CC$.

\smallskip

One can take the following as a definition of the notion of $\BE_2$-algebra in $\CC$:  the braided structure
on $\CC$ makes the tensor product functor $\CC\otimes \CC\to \CC$ into a monoidal functor. In particular,
it induces a monoidal structure in the category $\BE_1\on{-alg}(\CC)$ of $\BE_1$-algebras in $\CC$.  Now,
the category $\BE_2\on{-alg}(\CC)$ is defined to be
$$\BE_1\on{-alg}(\BE_1\on{-alg}(\CC)).$$

\medskip

Equivalently, for $\CA\in \BE_1\on{-alg}(\CC)$, to endow it with a structure of $\BE_2$-algebra amounts to
endowing the category $\CA\mod$ with a monoidal structure such that the forgetful functor $$\CA\mod\to \CC$$
is monoidal.

\medskip

\noindent{(iii)} Given an $\BE_2$-algebra $\CA$ in $\CC$ we can canonically attach to it an object $\CA_{\Ran(\BA^1)}\in \CC_{\Ran(\BA^1)}$
that is equipped with a structure of \emph{factorization algebra}. Moreover, this construction is an equivalence between the category 
$\BE_2\on{-alg}(\CC)$ and the category of factorization algebras in $\CC_{\Ran(\BA^1)}$.

\medskip

\noindent{(iii')} 
If $\CC$ is as in (i') and if $\CA$ is equivariant with respect to the ribbon structure, we can attach to $\CA$ an object 
$\CA_{\Ran(X)}\in \CC_{\Ran(X)}$ for any $X$. 

\medskip

\noindent{(iv)} For an $\BE_2$-algebra $\CA$ in $\CC$ we can talk about the category of $\BE_2$-modules over $\CA$,
denoted $\CA\mod_{\BE_2}$. The category $\CA\mod_{\BE_2}$ is itself braided monoidal and we have a canonical identification
\begin{equation} \label{e:Dr center}
\CA\mod_{\BE_2}\simeq Z_{\on{Dr},\CC}(\CA\mod),
\end{equation}
where $Z_{\on{Dr},\CC}(-)$ denotes the (relative to $\CC$) Drinfeld center of a given $\CC$-linear monoidal category.  

\medskip

\noindent{(v)} We have a canonical equivalence between $\CA\mod_{\BE_2}$ and the category of
\emph{factorization modules at $0\in \BA^1$} over $\CA_{\Ran(\BA^1)}$ in $\CC_{\Ran(\BA^1)}$.

\medskip

\noindent{(v')} In the situation of (iii'), given a point $x\in X$, let $(\CA\mod_{\BE_2})_{T_x(X)}$ be the twist of the
category $\CA\mod_{\BE_2}$ by the tangent line of $X$ at $x$ (the ribbon structure allows to twist the category 
$\CA\mod_{\BE_2}$ by a complex line). Then we have a canonical equivalence between $(\CA\mod_{\BE_2})_{T_x(X)}$ 
and the category of \emph{factorization modules at $x\in X$} over $\CA_{\Ran(X)}$ in $\CC_{\Ran(X)}$.

\sssec{}  \label{sss:BFS1}

Let $\Rep_q(T)$ denote the ribbon braided monoidal category, corresponding to $(\check\Lambda,q)$. Specifically,
the braiding is defined by setting
$$R^{\clambda,\cmu}:\BC^\clambda\otimes \BC^\cmu \to \BC^\cmu\otimes \BC^\lambda$$
to be the tautological map multiplied by $b^{\frac{1}{2}}_q(\clambda,\cmu)$. 

\medskip

We set the ribbon automorphism of $\BC^\clambda$ to be given by $b^{\frac{1}{2}}_q(\clambda,\clambda+2\check\rho)$.

\medskip

The category
$\on{Shv}_{\CG_{q,\on{loc}}}(\Ran(X,\cLambda))$, considered in \secref{sss:twisted shv Ran}, is the category over
the Ran space of $X$ corresponding to $\Rep_q(T)$ in the sense of \secref{sss:E2}(i').

\begin{rem}
The extra linear term in the formula for the ribbon structure corresponds to the linear term in the definition of
the gerbe $\CG_{q,\on{loc}}$. 
\end{rem} 

\sssec{}

Consider $\fu_q(N^+)$ as a Hopf algebra in $\Rep_q(T)$. In particular, we can consider the monoidal category
$\fu_q(N^+)\mod$ of modules over $\fu_q(N^+)$ \emph{in} $\Rep_q(T)$. We use a renormalized version of
$\fu_q(N^+)\mod$, which is compactly generated by finite-dimensional modules. 
We consider the \emph{lax} monoidal functor
of $\fu_q(N^+)$-invariants\footnote{Of course, the functor of invariants is understood in the derived sense.}:
\begin{equation}  \label{e:Inv}
\on{Inv}_{\fu_q(N^+)}:\fu_q(N^+)\mod\to \Rep_q(T).
\end{equation}

The Hopf algebra structure on $\fu_q(N^+)$ defines on $\on{Inv}_{\fu_q(N^+)}(\BC)$ a natural structure of 
$\BE_2$-algebra in $\Rep_q(T)$, see \secref{sss:E2}(ii). 

\medskip

Note also that $\fu_q(N^+)$ is naturally equivariant with respect to the ribbon twist on $\Rep_q(T)$, thus inducing an 
equivariant structure on the functor $\on{Inv}_{\fu_q(N^+)}(\BC)$. In particular, the braided monoidal
category $\on{Inv}_{\fu_q(N^+)}(\BC)\mod_{\BE_2}$ carries a canonical ribbon structure. 

\medskip

By \secref{sss:E2}(iii'), we can attach to the $\BE_2$-algebra $\on{Inv}_{\fu_q(N^+)}(\BC)$ in the ribbon
braided monoidal category $\Rep_q(T)$ a factorization algebra in the category over the Ran space of $X$ 
corresponding to $\Rep_q(T)$, i.e., $\on{Shv}_{\CG_{q,\on{loc}}}(\Ran(X,\cLambda))$.  

\medskip

We have (see \cite[Corollary 6.8]{FS}):

\begin{prop} \label{p:identify Omega !*}
The factorization algebra in $\on{Shv}_{\CG_{q,\on{loc}}}(\Ran(X,\cLambda))$ corresponding to 
the $\BE_2$-algebra $\on{Inv}_{\fu_q(N^+)}(\BC)\in \Rep_q(T)$ identifies canonically with 
$\Omega^{\on{small}}_q$.
\end{prop}

Hence, by \secref{sss:E2}(v'), we obtain a canonical equivalence
\begin{multline} \label{e:E2}
\left(\on{Inv}_{\fu_q(N^+)}(\BC)\mod_{\BE_2}\right)_{T_{x_1}(X)}\otimes...\otimes
\left(\on{Inv}_{\fu_q(N^+)}(\BC)\mod_{\BE_2}\right)_{T_{x_n}(X)} 
\simeq \Omega^{\on{small}}_q\mod_{x_1,...,x_n}.
\end{multline}

\sssec{}   \label{sss:BFS2}

Note now that the \emph{lax} monoidal functor 
\begin{equation} \label{e:N^+ inv}
\on{Inv}_{\fu_q(N^+)}:\fu_q(N^+)\mod\to \Rep_q(T)
\end{equation}
upgrades to a monoidal equivalence
\begin{equation} \label{e:N^+ inv enh}
\on{Inv}_{\fu_q(N^+)}^{\on{enh}}:\fu_q(N^+)\mod\to \on{Inv}_{\fu_q(N^+)}(\BC)\mod,
\end{equation}
and the latter induces a braided monoidal equivalence
\begin{equation} \label{e:from Z to Z}
Z_{\on{Dr},\Rep_q(T)}(\fu_q(N^+)\mod)\to Z_{\on{Dr},\Rep_q(T)}(\on{Inv}_{\fu_q(N^+)}(\BC)\mod).
\end{equation}
Applying \eqref{e:Dr center}, we obtain a braided monoidal equivalence
\begin{equation} \label{e:from Z small}
Z_{\on{Dr},\Rep_q(T)}(\fu_q(N^+)\mod)\to \on{Inv}_{\fu_q(N^+)}(\BC)\mod_{\BE_2}.
\end{equation}

\sssec{}  \label{sss:BFS3}

Finally, we recall that we have a canonical equivalence of ribbon braided monoidal categories 
\begin{equation} \label{e:double}
\fu_q(G)\mod\simeq Z_{\on{Dr},\Rep_q(T)}(\fu_q(N^+)\mod).
\end{equation} 

\medskip

Combining, we obtain an equivalence
\begin{multline} \label{e:BFS !* again}
(\fu_q(G)\mod)_{T_{x_1}(X)}\otimes...\otimes (\fu_q(G)\mod)_{T_{x_n}(X)}\overset{\text{\eqref{e:double}}}\simeq  \\
\simeq (Z_{\on{Dr},\Rep_q(T)}(\fu_q(N^+)\mod))_{T_{x_1}(X)}\otimes...\otimes (Z_{\on{Dr},\Rep_q(T)}(\fu_q(N^+)\mod))_{T_{x_n}(X)} 
\overset{\text{\eqref{e:from Z to Z}}}
\simeq \\
\simeq (Z_{\on{Dr},\Rep_q(T)}(\on{Inv}_{\fu_q(N^+)}(\BC)\mod)_{T_{x_1}(X)}\otimes...
\otimes (Z_{\on{Dr},\Rep_q(T)}(\on{Inv}_{\fu_q(N^+)}(\BC)\mod))_{T_{x_n}(X)} 
\overset{\text{\eqref{e:Dr center}}}
\simeq \\
\simeq \left(\on{Inv}_{\fu_q(N^+)}(\BC)\mod_{\BE_2}\right)_{T_{x_1}(X)}\otimes...\otimes
\left(\on{Inv}_{\fu_q(N^+)}(\BC)\mod_{\BE_2}\right)_{T_{x_n}(X)} \overset{\text{\eqref{e:E2}}}\simeq \\
\simeq \Omega^{\on{small}}_q\mod_{x_1,...,x_n}.
\end{multline}

The latter is the functor \eqref{e:BFS !*} from \cite{BFS}. 

\ssec{The Lusztig and Kac-De Concini versions of the quantum group}

The contents of rest of this section are not needed for the proof of \conjref{c:tilting conj}.

\sssec{}

In addition to $\Omega^{\on{small}}_q$ one can consider (at least) two more factorization algebras associated to $G$
in $\on{Shv}_{\CG_{q,\on{loc}}}(\Ran(X,\cLambda))$, 
denoted
\begin{equation} \label{e:other Omega}
\Omega^{\on{KD}}_q \text{ and } \Omega^{\on{Lus}}_q,
\end{equation}
respectively.

\medskip

These functors are defined as follows. We consider the Hopf algebras
$$\fU_q(N^+)^{\on{KD}} \text{ and } \fU_q(N^+)^{\on{Lus}}$$
in the braided monoidal category $\Rep_q(T)$, corresponding to the Kac-De Concini and Lusztig versions of the
quantum group, respectively. 

\medskip

Proceeding as in \secref{sss:BFS1} we obtain $\BE_2$-algebras in $\Rep_q(T)$, denoted 
$$\on{Inv}_{\fU_q(N^+)^{\on{KD}}}(\BC)\mod_{\BE_2} \text{ and } \on{Inv}_{\fU_q(N^+)^{\on{Lus}}}(\BC)\mod_{\BE_2},$$
respectively, equivariant with respect to the ribbon structure on $\Rep_q(T)$. We let 
$$\Omega^{\on{KD}}_q \text{ and } \Omega^{\on{Lus}}_q$$
be the corresponding factorization algebras in $\on{Shv}_{\CG_{q,\on{loc}}}(\Ran(X,\cLambda))$.

\medskip 

By construction we have canonical equivalences
\begin{multline*} %\label{e:E2 KD}
\left(\on{Inv}_{\fU_q(N^+)^{\on{KD}}}(\BC)\mod_{\BE_2}\right)_{T_{x_1}(X)}\otimes...\otimes
\left(\on{Inv}_{\fU_q(N^+)^{\on{KD}}}(\BC)\mod_{\BE_2}\right)_{T_{x_n}(X)} 
\simeq \Omega^{\on{KD}}_q\mod_{x_1,...,x_n}
\end{multline*}
and 
\begin{multline*} %\label{e:E2 KD}
\left(\on{Inv}_{\fU_q(N^+)^{\on{Lus}}}(\BC)\mod_{\BE_2}\right)_{T_{x_1}(X)}\otimes...\otimes
\left(\on{Inv}_{\fU_q(N^+)^{\on{Lus}}}(\BC)\mod_{\BE_2}\right)_{T_{x_n}(X)} 
\simeq \Omega^{\on{Lus}}_q\mod_{x_1,...,x_n}.
\end{multline*}

\sssec{}

As in \secref{sss:BFS2} we have canonically defined braided monoidal equivalences
%\begin{equation} \label{e:from Z KD}
$$\on{Inv}_{\fU_q(N^+)^{\on{KD}}}^{\on{enh}}: Z_{\on{Dr},\Rep_q(T)}(\fu_q(N^+)^{\on{KD}}\mod)\to 
\on{Inv}_{\fU_q(N^+)^{\on{KD}}}(\BC)\mod_{\BE_2}$$
%\end{equation}
and
%\begin{equation} \label{e:from Z KD}
$$\on{Inv}_{\fU_q(N^+)^{\on{Lus}}}^{\on{enh}}: Z_{\on{Dr},\Rep_q(T)}(\fu_q(N^+)^{\on{Lus}}\mod)\to 
\on{Inv}_{\fU_q(N^+)^{\on{Lus}}}(\BC)\mod_{\BE_2},$$
%\end{equation}
respectively. 

\sssec{}

Consider the braided monoidal categories
$$Z_{\on{Dr},\Rep_q(T)}(\fU_q(N^+)^{\on{KD}}\mod) \text{ and } Z_{\on{Dr},\Rep_q(T)}(\fU_q(N^+)^{\on{Lus}}\mod).$$

They identify, respectively, with the categories of modules over the corresponding ``lopsided" versions of the quantum group
$$\fU_q(G)^{+_{\on{KD}},-_{\on{Lus}}}\mod  \text{ and } \fU_q(G)^{+_{\on{Lus}},-_{\on{KD}}}\mod.$$

\medskip

This follows from the fact that the (graded and relative to $\Rep_q(T)$) duals of the Hopf algebras $\fU_q(N^+)^{\on{KD}}$ 
and $\fU_q(N^+)^{\on{Lus}}$ are the Hopf algebras $\fU_q(N^-)^{\on{Lus}}$ and $\fU_q(N^-)^{\on{KD}}$, respectively.

\sssec{}

Composing, we obtain the functors
\begin{multline} \label{e:BFS *}
(\fU_q(G)^{+_{\on{KD}},-_{\on{Lus}}}\mod)_{T_{x_1}(X)}\otimes...\otimes (\fU_q(G)^{+_{\on{KD}},-_{\on{Lus}}}\mod)_{T_{x_n}(X)}\to \\
\to \Omega^{\on{KD}}_q\mod_{x_1,...,x_n}
\end{multline}
and 
\begin{multline} \label{e:BFS !}
(\fU_q(G)^{+_{\on{Lus}},-_{\on{KD}}}\mod)_{T_{x_1}(X)}\otimes...\otimes (\fU_q(G)^{+_{\on{Lus}},-_{\on{KD}}}\mod)_{T_{x_n}(X)}\to \\
\to \Omega^{\on{Lus}}_q\mod_{x_1,...,x_n}.
\end{multline}

The functors \eqref{e:BFS *} and \eqref{e:BFS !} are the respective counterparts for 
$\fU_q(N^+)^{\on{KD}}$ and $\fU_q(N^+)^{\on{Lus}}$ of the functor \eqref{e:BFS !*}.

\sssec{}

Composing the functors \eqref{e:BFS *} and \eqref{e:BFS !} with the forgetful functors
$$\Omega^{\on{KD}}_q\mod_{x_1,...,x_n}\to \on{Shv}_{\CG_{q,\on{loc}}}(\Ran(X,\cLambda)) \text{ and }
\Omega^{\on{Lus}}_q\mod_{x_1,...,x_n}\to \on{Shv}_{\CG_{q,\on{loc}}}(\Ran(X,\cLambda)),$$
we obtain the functors
$$(\fU_q(G)^{+_{\on{KD}},-_{\on{Lus}}}\mod)_{T_{x_1}(X)}\otimes...\otimes (\fU_q(G)^{+_{\on{KD}},-_{\on{Lus}}}\mod)_{T_{x_n}(X)}\to
\on{Shv}_{\CG_{q,\on{loc}}}(\Ran(X,\cLambda))$$
and 
$$(\fU_q(G)^{+_{\on{Lus}},-_{\on{KD}}}\mod)_{T_{x_1}(X)}\otimes...\otimes (\fU_q(G)^{+_{\on{Lus}},-_{\on{KD}}}\mod)_{T_{x_n}(X)}\to
\on{Shv}_{\CG_{q,\on{loc}}}(\Ran(X,\cLambda))$$
that we denote by $\BFS^{\on{top}}_{\fU^{\on{KD}}_q}$ and $\BFS^{\on{top}}_{\fU^{\on{Lus}}_q}$, respectively. 

\sssec{}

The functors 
$$\BFS^{\on{top}}_{\fu_q},\,\, \BFS^{\on{top}}_{\fU^{\on{KD}}_q} \text{ and }\BFS^{\on{top}}_{\fU^{\on{Lus}}_q}$$
can each be viewed as coming from the corresponding \emph{lax} braided monoidal functors
$$\on{Inv}_{\fu_q(N^+)}:\fu_q(G)\mod \to \Rep_q(T),$$
$$\on{Inv}_{\fU_q(N^+)^{\on{KD}}}:\fU_q(G)^{+_{\on{KD}},-_{\on{Lus}}}\mod\to \Rep_q(T)$$
and
$$\on{Inv}_{\fU_q(N^+)^{\on{Lus}}}:\fU_q(G)^{+_{\on{Lus}},-_{\on{KD}}}\mod \to \Rep_q(T),$$
respectively.

\medskip

The functors $\BFS^{\on{top}}_{\fU^{\on{KD}}_q}$ and $\BFS^{\on{top}}_{\fU^{\on{Lus}}_q}$ are the respective counterparts for 
$\fU_q(N^+)^{\on{KD}}$ and $\fU_q(N^+)^{\on{Lus}}$ of the functor $\BFS^{\on{top}}_{\fu_q}$ from \secref{sss:BFS functor}.

\ssec{Restriction functors and natural transformations}

\sssec{}

Note now that we have the homomorphisms of Hopf algebras in $\Rep_q(T)$:
\begin{equation} \label{e:N+ res}
\fU_q(N^+)^{\on{KD}}\to \fu_q(N^+)\to \fU_q(N^+)^{\on{Lus}}.
\end{equation} 

\medskip

In addition, the braided monoidal categories
$$\fU_q(G)^{+_{\on{KD}},-_{\on{small}}}\mod  \text{ and } \fU_q(G)^{+_{\on{small}},-_{\on{KD}}}\mod$$
and the following commutative diagrams of braided monoidal functors
$$
\CD
\fU_q(G)\mod    @>{\Res^{\on{big}\to \on{KD}}}>>    \fU_q(G)^{+_{\on{KD}},-_{\on{Lus}}}\mod   \\
@V{\Res^{\on{big}\to \on{small}}}VV    @VVV  \\
 \fu_q(G)\mod  @>>>  \fU_q(G)^{+_{\on{KD}},-_{\on{small}}}\mod
\endCD
$$
and
$$
\CD
\fU_q(G)\mod    @>{\Res^{\on{big}\to \on{Lus}}}>>    \fU_q(G)^{+_{\on{Lus}},-_{\on{KD}}}\mod   \\
@V{\Res^{\on{big}\to \on{small}}}VV    @VVV  \\
 \fu_q(G)\mod  @>>>  \fU_q(G)^{+_{\on{small}},-_{\on{KD}}}\mod.
\endCD
$$
From here we obtain the natural transformations
$$
\xy
(40,0)*+{\fU_q(G)^{+_{\on{KD}},-_{\on{Lus}}}\mod}="D";
(40,-20)*+{\fU_q(G)^{+_{\on{KD}},-_{\on{small}}}\mod}="E";
(40,-40)*+{\fu_q(G)\mod}="F";
(-20,0)*+{\fU_q(G)\mod}="A";
(-20,-20)*+{\fU_q(G)\mod}="B";
(-20,-40)*+{\fU_q(G)\mod}="C";
{\ar@{<-}_{=} "A";"B"};
{\ar@{->}^{\Res^{\on{big}\to \on{KD}}} "A";"D"};
{\ar@{->} "B";"E"};
{\ar@{->}_{\Res^{\on{big}\to \on{small}}} "C";"F"};
{\ar@{<-}_{=}  "B";"C"};
{\ar@{<-}^{\on{induction}}  "D";"E"};
{\ar@{<-} "E";"F"};
{\ar@{=>} "E";"A"};
{\ar@{=>}_{\sim} "F";"B"};
\endxy
$$
and
$$
\xy
(40,0)*+{\fU_q(G)^{+_{\on{Lus}},-_{\on{KD}}}\mod}="D";
(40,-20)*+{\fU_q(G)^{+_{\on{small}},-_{\on{KD}}}\mod}="E";
(40,-40)*+{\fu_q(G)\mod,}="F";
(-20,0)*+{\fU_q(G)\mod}="A";
(-20,-20)*+{\fU_q(G)\mod}="B";
(-20,-40)*+{\fU_q(G)\mod}="C";
{\ar@{->}_{=} "A";"B"};
{\ar@{->}^{\Res^{\on{big}\to \on{Lus}}} "A";"D"};
{\ar@{->} "B";"E"};
{\ar@{->}_{\Res^{\on{big}\to \on{small}}} "C";"F"};
{\ar@{->}_{=}  "B";"C"};
{\ar@{->} "D";"E"};
{\ar@{->}^{\on{induction}} "E";"F"};
{\ar@{=>} _{\sim}  "D";"B"};
{\ar@{=>} "E";"C"};
\endxy
$$
where the induction functors $$\fU_q(G)^{+_{\on{KD}},-_{\on{small}}}\mod\to \fU_q(G)^{+_{\on{KD}},-_{\on{Lus}}}\mod$$ and 
$$\fU_q(G)^{+_{\on{small}},-_{\on{KD}}}\mod\to \fu_q(G)\mod$$
appearing in the above two diagrams
are left adjoint to the restriction functors
$$ \fU_q(G)^{+_{\on{KD}},-_{\on{Lus}}}\mod\to \fU_q(G)^{+_{\on{KD}},-_{\on{small}}}\mod$$
and
$$\fu_q(G)\mod\to \fU_q(G)^{+_{\on{small}},-_{\on{KD}}}\mod,$$
respectively.

\sssec{}

In addition, by adjunction we obtain the natural transformations
$$
\xy
(-20,0)*+{\fU_q(G)^{+_{\on{KD}},-_{\on{Lus}}}\mod}="A";
(-20,-20)*+{\fU_q(G)^{+_{\on{KD}},-_{\on{small}}}\mod}="B";
(-20,-40)*+{\fu_q(G)\mod}="C";
(40,0)*+{\Rep_q(T)}="D";
(40,-20)*+{\Rep_q(T)}="E";
(40,-40)*+{\Rep_q(T)}="F";
{\ar@{->}^{\on{induction}} "B";"A"};
{\ar@{->}^{\on{Inv}_{\fU_q(N^+)^{\on{KD}}}} "A";"D"};
{\ar@{->}^{\on{Inv}_{\fU_q(N^+)^{\on{KD}}}} "B";"E"};
{\ar@{->}_{\on{Inv}_{\fu_q(N^+)}} "C";"F"};
{\ar@{->} "C";"B"};
{\ar@{<-}_{=} "D";"E"};
{\ar@{<-}_{=} "E";"F"};
{\ar@{=>} "F";"B"};
{\ar@{=>} "E";"A"};
\endxy
$$
and
$$
\xy
(-20,0)*+{\fU_q(G)^{+_{\on{Lus}},-_{\on{KD}}}\mod}="A";
(-20,-20)*+{\fU_q(G)^{+_{\on{Lus}},-_{\on{small}}}\mod}="B";
(-20,-40)*+{\fu_q(G)\mod}="C";
(40,0)*+{\Rep_q(T)}="D";
(40,-20)*+{\Rep_q(T)}="E";
(40,-40)*+{\Rep_q(T).}="F";
{\ar@{->} "A";"B"};
{\ar@{->}^{\on{Inv}_{\fU_q(N^+)^{\on{Lus}}}} "A";"D"};
{\ar@{->}^{\on{Inv}_{\fU_q(N^+)^{\on{Lus}}}} "B";"E"};
{\ar@{->}_{\on{Inv}_{\fu_q(N^+)}} "C";"F"};
{\ar@{->}_{\on{induction}} "B";"C"};
{\ar@{->}^{=} "D";"E"};
{\ar@{->}^{=} "E";"F"};
{\ar@{=>}  "D";"B"};
{\ar@{=>} "E";"C"};
\endxy
$$

\sssec{}

Composing, we obtain the natural transformations 
$$\on{Inv}_{\fU_q(N^+)^{\on{Lus}}} \circ \Res^{\on{big}\to \on{Lus}} \to
\on{Inv}_{\fu_q(N^+)} \circ \Res^{\on{big}\to \on{small}} \to \on{Inv}_{\fU_q(N^+)^{\on{KD}}} \circ \Res^{\on{big}\to \on{KD}}$$
as braided monoidal functors 
$$\fU_q(G)\mod\to \Rep_q(T).$$

\medskip
 
Hence we obtain the natural transformations
\begin{equation} \label{e:BFS nat trans}
\BFS^{\on{top}}_{\fU^{\on{Lus}}_q}\circ \Res^{\on{big}\to \on{Lus}} \to
\BFS^{\on{top}}_{\fu_q}  \circ \Res^{\on{big}\to \on{small}} \to \BFS^{\on{top}}_{\fU^{\on{KD}}_q}\circ \Res^{\on{big}\to \on{KD}}
\end{equation}
as functors
$$(\fU_q(G))_{T_{x_1}(X)}\otimes...\otimes (\fU_q(G))_{T_{x_n}(X)}\to
\on{Shv}_{\CG_{q,\on{loc}}}(\Ran(X,\cLambda)).$$

\section{Passing from modules over quantum groups to Kac-Moody representations}  \label{s:KL}

In this section we let $\kappa$ be a positive integral level and $\kappa'=-\kappa-\kappa_{\on{Kil}}$ the reflected level. 

\medskip

Recall that we have reduced the statement of \conjref{c:tilting conj} to the existence of the isomorphism \eqref{e:reform 6}.
The bridge between between the two sides in \eqref{e:reform 6} will be provided by the Kazhdan-Lusztig equivalence.

\ssec{The Kazhdan-Lusztig equivalence}

In this subsection we will finally explain what the tilting conjecture is ``really about"\footnote{From the point of view taken in this paper.}.  
Namely, we will replace it by a more general statement, in which the curve $X$ will be arbitrary (rather than $\BP^1$), and 
instead of the tilting module we will consider an arbitrary collection of representations of the Kac-Moody algebra.  

\sssec{} \label{sss:matching}

We take the data of $q$ and $\kappa'$ to match in the following sense. 

\medskip

Starting from the bilinear form $\kappa'$, consider the form $\kappa'-\kappa_{\on{crit}}|_{\ft}$. Since the latter was assumed
non-degenerate, we can consider its inverse, which is a symmetric bilinear form $(\kappa'-\kappa_{\on{crit}}|_{\ft})^{-1}$ on $\ft^\vee$,
and can thus be regarded as a symmetric bilinear form on $\cLambda$ with coefficients in $\BC$. 

\medskip

Finally, we set
$$b_q=\on{exp}(2\pi \cdot i\cdot \frac{(\kappa'-\kappa_{\on{crit}}|_{\ft})^{-1}}{2}),$$
regarded as a symmetric bilinear form on $\cLambda$ with coefficients in $\BC^*$. 

\sssec{}

According to \cite{KL}, we have a canonical equivalence
$$\KL_G:\hg_{\kappa'}\mod^{G(\CO)}\to \fU_q(G)\mod.$$

\medskip

Note that we have:
\begin{equation} \label{e:tilting to tilting}
\KL_G(T^\lambda_{\kappa'})\simeq \fT^\lambda_q.
\end{equation}

\sssec{}

For an $n$-tuple of points $x_1,...,x_n\in X$ we consider the following two functors
$$\hg_{\kappa',x_1}\mod^{G(\CO_{x_1})}\otimes...\otimes \hg_{\kappa',x_n}\mod^{G(\CO_{x_n})}\to \Vect.$$

One functor is 
\begin{multline} \label{e:from big two}
\hg_{\kappa',x_1}\mod^{G(\CO_{x_1})}\otimes...\otimes \hg_{\kappa',x_n}\mod^{G(\CO_{x_n})}
\overset{\Loc_{G,\kappa',x_1,...,x_n}}\longrightarrow \Dmod_{\kappa'}(\Bun_G)_{\on{co}} \to \\
\overset{\on{CT}_{\kappa'+\on{shift},!*}}\longrightarrow \Dmod_{\kappa'+\on{shift}}(\Bun_T) \overset{-\otimes \CL_{T,-\kappa'-\on{shift}}}
\longrightarrow 
\Dmod(\Bun_T) \overset{\Gamma_\dr(\Bun_T,-)}\longrightarrow \Vect.$$
\end{multline}

\medskip

The other functor is
\begin{multline} \label{e:from big one}
\hg_{\kappa',x_1}\mod^{G(\CO_{x_1})}\otimes...\otimes \hg_{\kappa',x_n}\mod^{G(\CO_{x_n})}\overset{\KL_G}\longrightarrow \\
\to 
(\fU_q(G)\mod)_{T_{x_1}(X)}\otimes...\otimes (\fU_q(G)\mod)_{T_{x_n}(X)} 
\overset{\BFS^{\on{top}}_{\fu_q}\circ \Res^{\on{big}\to \on{small}}}\longrightarrow  \on{Shv}_{\CG_{q,\on{loc}}}(\Ran(X,\cLambda)) \to \\
 \overset{\on{AJ}_!}\longrightarrow  \on{Shv}_{\CG_{q,\on{glob}}} (\on{Pic}(X)\underset{\BZ}\otimes \cLambda)
\overset{-\sotimes \CE_{q^{-1}}}\to \on{Shv}(\on{Pic}(X)\underset{\BZ}\otimes \cLambda) \longrightarrow \Vect.
\end{multline} 

\sssec{}

Taking into account the reformulation of \conjref{c:tilting conj} as the existence of an isomorphism \eqref{e:reform 5},
combining with \thmref{t:semi-inf} and the isomorphism \eqref{e:tilting to tilting}, we obtain that we can reformulate \conjref{c:tilting conj} 
as the existence of an isomorphism
\begin{multline} \label{e:reform 7}
\Gamma_\dr(\Bun_T, \on{CT}_{\kappa'+\on{shift},!*}\circ
\Loc_{G,\kappa',x_\infty}(T^\lambda_{\kappa'})\otimes \CL_{T,-\kappa'-\on{shift}}) [-\dim(\Bun_T)]) \simeq \\
\simeq 
\on{Conf}^{\fu_q}_{\BP^1;\infty}(\Res^{\on{big}\to \on{small}}\circ \KL_G(T^\lambda_{\kappa'}))=\\
=\Gamma\left(\on{Pic}(X)\underset{\BZ}\otimes \cLambda,
\on{AJ}_!\left(\BFS^{\on{top}}_{\fu_q}\circ \Res^{\on{big}\to \on{small}}\circ 
\KL_G(T^\lambda_{\kappa'})\right)\sotimes \CE_{q^{-1}}[-\dim(\Bun_T)]\right).
\end{multline} 

\medskip

Hence, \conjref{c:tilting conj} follows from the following more general statement: 

\begin{conj}  \label{c:main}
The functors \eqref{e:from big two} and \eqref{e:from big one} are canonically isomorphic. I.e., the diagram of functors
$$
\CD
\hg_{\kappa',x_1}\mod^{G(\CO_{x_1})}\otimes...\otimes \hg_{\kappa',x_n}\mod^{G(\CO_{x_n})}   @>{\Loc_{G,\kappa',x_1,...,x_n}}>> 
 \Dmod_{\kappa'}(\Bun_G)_{\on{co}}   \\
@V{\KL_G}VV       @VV{\on{CT}_{\kappa'+\on{shift},!*}}V     \\
(\fU_q(G)\mod)_{T_{x_1}(X)}\otimes...\otimes (\fU_q(G)\mod)_{T_{x_n}(X)}  & &    \Dmod_{\kappa'+\on{shift}}(\Bun_T)  \\
@V{\BFS^{\on{top}}_{\fu_q}\circ \Res^{\on{big}\to \on{small}}}VV     @VV{\otimes \CL_{T,-\kappa'-\on{shift}}}V      \\
\on{Shv}_{\CG_{q,\on{loc}}}(\Ran(X,\cLambda))     & &   \Dmod(\Bun_T) \\
@V{\on{AJ}_!}VV    \\
\on{Shv}_{\CG_{q,\on{glob}}}(\on{Pic}(X)\underset{\BZ}\otimes \cLambda) & & @VV{\Gamma_\dr(\Bun_T,-)}V   \\
@V{-\sotimes \CE_{q^{-1}}}VV      \\
\on{Shv}(\Ran(X,\cLambda))     @>{\Gamma(\on{Pic}(X)\underset{\BZ}\otimes \cLambda,-)}>>  \Vect
\endCD
$$
commutes. 
\end{conj}

The rest of the paper is devoted to the sketch of a proof of \conjref{c:main}. 

\sssec{}

In addition to \conjref{c:main} we will sketch the proof of the following two of its versions:

\begin{conj} \label{c:main ! and *}  \hfill

\smallskip

\noindent{\em(a)}
The following functors  $\hg_{\kappa',x_1}\mod^{G(\CO_{x_1})}\otimes...\otimes \hg_{\kappa',x_n}\mod^{G(\CO_{x_n})}\to \Vect$
are canonically isomorphic:
\begin{multline*} 
\hg_{\kappa',x_1}\mod^{G(\CO_{x_1})}\otimes...\otimes \hg_{\kappa',x_n}\mod^{G(\CO_{x_n})}
\overset{\Loc_{G,\kappa',x_1,...,x_n}}\longrightarrow \Dmod_{\kappa'}(\Bun_G)_{\on{co}} \to \\
\overset{\on{CT}_{\kappa'+\on{shift},*}}\longrightarrow \Dmod_{\kappa'+\on{shift}}(\Bun_T) \overset{-\otimes \CL_{T,-\kappa'-\on{shift}}}
\longrightarrow 
\Dmod(\Bun_T) \overset{\Gamma_\dr(\Bun_T,-)}\longrightarrow \Vect
\end{multline*}
and 
\begin{multline*}
\hg_{\kappa',x_1}\mod^{G(\CO_{x_1})}\otimes...\otimes \hg_{\kappa',x_n}\mod^{G(\CO_{x_n})}\overset{\KL_G}\longrightarrow \\
\to 
(\fU_q(G)\mod)_{T_{x_1}(X)}\otimes...\otimes (\fU_q(G)\mod)_{T_{x_n}(X)} 
\overset{\BFS^{\on{top}}_{\fU_q^{\on{KD}}}\circ \Res^{\on{big}\to \on{KD}}}\longrightarrow  \on{Shv}_{\CG_{q,\on{loc}}}(\Ran(X,\cLambda)) \to \\
 \overset{\on{AJ}_!}\longrightarrow  \on{Shv}_{\CG_{q,\on{glob}}} (\on{Pic}(X)\underset{\BZ}\otimes \cLambda)
\overset{\sotimes \CE_{q^{-1}}}\to \on{Shv}(\on{Pic}(X)\underset{\BZ}\otimes \cLambda) \longrightarrow \Vect.
\end{multline*}
\smallskip

\noindent{\em(b)}
The following functors  $\hg_{\kappa',x_1}\mod^{G(\CO_{x_1})}\otimes...\otimes \hg_{\kappa',x_n}\mod^{G(\CO_{x_n})}\to \Vect$
are canonically isomorphic:
\begin{multline*}
\hg_{\kappa',x_1}\mod^{G(\CO_{x_1})}\otimes...\otimes \hg_{\kappa',x_n}\mod^{G(\CO_{x_n})}
\overset{\Loc_{G,\kappa',x_1,...,x_n}}\longrightarrow \Dmod_{\kappa'}(\Bun_G)_{\on{co}} \to \\
\overset{\on{CT}_{\kappa'+\on{shift},!}}\longrightarrow \Dmod_{\kappa'+\on{shift}}(\Bun_T) \overset{-\otimes \CL_{T,-\kappa'-\on{shift}}}
\longrightarrow 
\Dmod(\Bun_T) \overset{\Gamma_\dr(\Bun_T,-)}\longrightarrow \Vect
\end{multline*}
and
\begin{multline*}
\hg_{\kappa',x_1}\mod^{G(\CO_{x_1})}\otimes...\otimes \hg_{\kappa',x_n}\mod^{G(\CO_{x_n})}\overset{\KL_G}\longrightarrow \\
\to 
(\fU_q(G)\mod)_{T_{x_1}(X)}\otimes...\otimes (\fU_q(G)\mod)_{T_{x_n}(X)} 
\overset{\BFS^{\on{top}}_{\fU_q^{\on{Lus}}}\circ \Res^{\on{big}\to \on{Lus}}}\longrightarrow  
\on{Shv}_{\CG_{q,\on{loc}}}(\Ran(X,\cLambda)) \to \\
 \overset{\on{AJ}_!}\longrightarrow  \on{Shv}_{\CG_{q,\on{glob}}} (\on{Pic}(X)\underset{\BZ}\otimes \cLambda)
\overset{\sotimes \CE_{q^{-1}}}\to \on{Shv}(\on{Pic}(X)\underset{\BZ}\otimes \cLambda)
\longrightarrow \Vect.
\end{multline*}
\end{conj}

\begin{rem}  \label{r:nat trans main}
As we shall see, the natural transformations between the left-hand sides in Conjectures \ref{c:main}
and \ref{c:main ! and *}, induced by the natural transformations
$$\Eis_!\to \Eis_{!*}\to \Eis_*$$
correspond to the natural transformations between the right-hand sides in Conjectures \ref{c:main}
and \ref{c:main ! and *}, induced by the natural transformations
$$\BFS^{\on{top}}_{\fU^{\on{Lus}}_q}\circ \Res^{\on{big}\to \on{Lus}} \to
\BFS^{\on{top}}_{\fu_q}  \circ \Res^{\on{big}\to \on{small}} \to \BFS^{\on{top}}_{\fU^{\on{KD}}_q}\circ \Res^{\on{big}\to \on{KD}}$$
of \eqref{e:BFS nat trans}.
\end{rem}

\ssec{Riemann-Hilbert correspondence}

Let us observe that in \conjref{c:main} the left-hand side, i.e., \eqref{e:from big two}, 
is completely algebraic (i.e., is formulated in the language of twisted D-modules), while right-hand side
deals with sheaves in the analytic topology. 

\medskip

In this subsection we will start the process of replacing sheaves by D-modules, by applying Riemann-Hilbert correspondence. 
In particular, we will introduce the D-module counterparts of the objects discussed in \secref{s:BFS}. 

\sssec{}

Let us return to the construction of the gerbe $\CG_{q,\on{loc}}$ in \secref{sss:gerbe}. If instead of the symmetric bilinear form
$$b_q:\cLambda\otimes\cLambda\to \BC^*$$
we use 
$$\frac{(\kappa'-\kappa_{\on{crit}}|_{\ft})^{-1}}{2}:\cLambda\otimes\cLambda\to \BC,$$
the same constructing yields a twisting on the prestack $\Ran(X,\cLambda)$.

\medskip

We denote the resulting category of twisted D-modules by
$$\Dmod_{\kappa^{-1}+\on{trans}}(\Ran(X,\cLambda))$$
(the reason for the choice of the notation ``$\kappa^{-1}+\on{trans}$" in the subscript will become clear in \secref{sss:trans twist}). 

\medskip

By construction, Riemann-Hilbert correspondence defines a fully-faithful embedding
$$\on{Shv}_{\CG_{q,\on{loc}}}(\Ran(X,\cLambda))\overset{\on{RH}}\longrightarrow \Dmod_{\kappa^{-1}+\on{trans}}(\Ran(X,\cLambda)).$$

\sssec{}

Thus, starting from the factorization algebra 
$$\Omega^{\on{small}}_q\in \on{Shv}_{\CG_{q,\on{loc}}}(\Ran(X,\cLambda)),$$
we obtain the factorization algebra
$$\Omega^{\on{small}}_{\kappa^{-1}+\on{trans}} \in\Dmod_{\kappa^{-1}+\on{trans}}(\Ran(X,\cLambda))$$
and the functor
$$(\fu_q(G)\mod)_{T_{x_1}(X)}\otimes...\otimes (\fu_q(G)\mod)_{T_{x_n}(X)}\to
\Omega^{\on{small}}_{\kappa^{-1}+\on{trans}}\mod_{x_1,...,x_n},$$
the latter being the D-module counterpart of the functor \eqref{e:BFS !*}. 

\medskip

We denote the composition of this functor with the forgetful functor
$$\Omega^{\on{small}}_{\kappa^{-1}+\on{trans}}\mod_{x_1,...,x_n}\to\Dmod_{\kappa^{-1}+\on{trans}}(\Ran(X,\cLambda))$$
by $\BFS^{\on{Dmod}}_{\fu_q}$.

\sssec{}

Consider now the stack $$\on{Pic}(X)\underset{\BZ}\otimes \cLambda\simeq \Bun_{\check{T}}.$$

\medskip

On it we will consider the twisting given by the bilinear form $(\kappa-\kappa_{\on{crit}})^{-1}$ on $\ft^\vee$; denote the resulting 
category of twisted D-modules by
$$\Dmod_{(\kappa-\kappa_{\on{crit}})^{-1}}(\on{Pic}(X)\underset{\BZ}\otimes \cLambda).$$

\medskip

We will now consider another twisting on $\on{Pic}(X)\underset{\BZ}\otimes \cLambda$, 
obtained from one above by \emph{translation} by the point 
$\omega_X\otimes \check\rho\in \on{Pic}(X)\underset{\BZ}\otimes \cLambda$. We denote the resulting category of D-modules
by 
$$\Dmod_{\kappa^{-1}+\on{trans}}(\on{Pic}(X)\underset{\BZ}\otimes \cLambda).$$
By definition, we have an equivalence of categories
$$\Dmod_{(\kappa-\kappa_{\on{crit}})^{-1}}(\on{Pic}(X)\underset{\BZ}\otimes \cLambda) \to
\Dmod_{\kappa^{-1}+\on{trans}}(\on{Pic}(X)\underset{\BZ}\otimes \cLambda),$$
given by translation by $\omega_X\otimes -\check\rho$. 

\sssec{}  \label{sss:trans twist}

It follows from the construction that the above twistings on $\Ran(X,\cLambda)$ and $\on{Pic}(X)\underset{\BZ}\otimes \cLambda$
are compatible under the Abel-Jacobi map 
$$\on{AJ}:\Ran(X,\cLambda)\to \on{Pic}(X)\underset{\BZ}\otimes \cLambda$$
(see \secref{sss:AJ}). 

\medskip

In particular, we have a pair of mutually adjoint functors 
$$\on{AJ}_!:\Dmod_{\kappa^{-1}+\on{trans}}(\Ran(X,\cLambda))\rightleftarrows \Dmod_{\kappa^{-1}+\on{trans}}
(\on{Pic}(X)\underset{\BZ}\otimes \cLambda):\on{AJ}^!.$$

\sssec{}

In addition, by the construction of $\Dmod_{\kappa^{-1}+\on{trans}}(\on{Pic}(X)\underset{\BZ}\otimes \cLambda)$, the 
Riemann-Hilbert correspondence defines a fully-faithful embedding
$$\on{Shv}_{\CG_{q,\on{glob}}}(\on{Pic}(X)\underset{\BZ}\otimes \cLambda)\overset{\on{RH}}\longrightarrow
\Dmod_{\kappa^{-1}+\on{trans}}(\on{Pic}(X)\underset{\BZ}\otimes \cLambda)$$
so that the diagram 
$$
\CD
\on{Shv}_{\CG_{q,\on{loc}}}(\Ran(X,\cLambda))  @>{\on{RH}}>>  \Dmod_{\kappa^{-1}+\on{trans}}(\Ran(X,\cLambda))  \\
@V{\on{AJ}_!}VV   @VV{\on{AJ}_!}V    \\
\on{Shv}_{\CG_{q,\on{glob}}}(\on{Pic}(X)\underset{\BZ}\otimes \cLambda)  @>>{\on{RH}}>
\Dmod_{\kappa^{-1}+\on{trans}}(\on{Pic}(X)\underset{\BZ}\otimes \cLambda)
\endCD
$$
commutes.

\sssec{}

Let $\CE_{-\kappa^{-1}-\on{trans}}$ be the image under the Riemann-Hilbert correspondence (for the opposite twisting) of
the local system 
$$\CE_{q^{-1}}\in \on{Shv}_{\CG_{q,\on{glob}}}(\on{Pic}(X)\underset{\BZ}\otimes \cLambda),$$
see \secref{sss:E}. 

\medskip

Thus, we can rewrite the functor appearing in \eqref{e:from big one}, i.e., the right-hand side of \conjref{c:main}, as
\begin{multline} \label{e:rewrite RH}
\hg_{\kappa',x_1}\mod^{G(\CO_{x_1})}\otimes...\otimes \hg_{\kappa',x_n}\mod^{G(\CO_{x_n})}\overset{\KL_G}\longrightarrow  \\
\to (\fU_q(G)\mod)_{T_{x_1}(X)}\otimes...\otimes (\fU_q(G)\mod)_{T_{x_n}(X)} 
\overset{\BFS^{\on{Dmod}}_{\fu_q}\circ \Res^{\on{big}\to \on{small}}}\longrightarrow \\
\to \Dmod_{\kappa^{-1}+\on{trans}}(\Ran(X,\cLambda)) 
\overset{\on{AJ}_!}\longrightarrow  \Dmod_{\kappa^{-1}+\on{trans}}(\on{Pic}(X)\underset{\BZ}\otimes \cLambda)
\overset{-\sotimes \CE_{-\kappa^{-1}-\on{trans}}}\longrightarrow  \\
\to \Dmod(\on{Pic}(X)\underset{\BZ}\otimes \cLambda)
\overset{\Gamma_\dr(\on{Pic}(X)\underset{\BZ}\otimes \cLambda,-)} \longrightarrow  \Vect.
\end{multline}

\medskip

Thus, we obtain that \conjref{c:main} is equivalent to the existence of an isomorphism between the functors 
\eqref{e:from big two} and \eqref{e:rewrite RH}. 

\begin{rem}
Note that part of the functor in \eqref{e:rewrite RH} is the composition
\begin{multline*} 
\hg_{\kappa',x_1}\mod^{G(\CO_{x_1})}\otimes...\otimes \hg_{\kappa',x_n}\mod^{G(\CO_{x_n})}\overset{\KL_G}\longrightarrow \\
\to (\fU_q(G)\mod)_{T_{x_1}(X)}\otimes...\otimes (\fU_q(G)\mod)_{T_{x_n}(X)} 
\overset{\BFS^{\on{Dmod}}_{\fu_q}\circ \Res^{\on{big}\to \on{small}}}\longrightarrow \\
\to \Dmod_{\kappa^{-1}+\on{trans}}(\Ran(X,\cLambda)).
\end{multline*}

\medskip

This functor combines the Kazhdan-Lusztig functor with the Riemann-Hilbert functor. 
As we shall see, the transcendental aspects of both of these functors cancel each other out, so the 
above composition is actually of algebraic nature (in particular, it can be defined of an arbitrary ground field of
characteristic zero). 

\end{rem} 

\ssec{Global Fourier-Mukai transform}

Note that the assertion that the functors \eqref{e:from big two} and \eqref{e:rewrite RH} are isomorphic is
non-tautological even for $G=T$ since in the case of the former we are dealing with the stack $\Bun_T$,
and in the case of the latter with $\on{Pic}(X)\underset{\BZ}\otimes \cLambda\simeq \Bun_{\check{T}}$. 

\medskip

In this subsection we will apply the Fourier-Mukai transform in order to 
replace $\check{T}$ by $T$ throughout. 

\sssec{}

Note that the Fourier-Mukai transform defines an equivalence
$$\Dmod_{(\kappa-\kappa_{\on{crit}})^{-1}}(\on{Pic}(X)\underset{\BZ}\otimes \cLambda) \simeq \Dmod_{\kappa'-\kappa_{\on{crit}}}(\Bun_T)$$

From here, using \secref{sss:twist as shift}, we obtain the equivalence
\begin{equation} 
\on{FM}_{\on{glob}}:\Dmod_{\kappa^{-1}+\on{trans}}(\on{Pic}(X)\underset{\BZ}\otimes \cLambda)\simeq 
\Dmod_{\kappa'+\on{shift}}(\Bun_T).
\end{equation} 

\sssec{}

Consider also the categories endowed with opposite twistings, denoted
$$\Dmod_{-\kappa^{-1}-\on{trans}}(\on{Pic}(X)\underset{\BZ}\otimes \cLambda)  \text{ and } \Dmod_{-\kappa'-\on{shift}}(\Bun_T),$$
respectively. 

\medskip

As in \secref{sss:duality twisted}, we have the canonical identifications
$$\Dmod_{-\kappa^{-1}-\on{trans}}(\on{Pic}(X)\underset{\BZ}\otimes \cLambda) \simeq
(\Dmod_{\kappa^{-1}+\on{trans}}(\on{Pic}(X)\underset{\BZ}\otimes \cLambda))^\vee$$
and 
$$\Dmod_{-\kappa'-\on{shift}}(\Bun_T)\simeq (\Dmod_{\kappa'+\on{shift}}(\Bun_T))^\vee.$$ 

\medskip

Let $\on{FM}'_{\on{glob}}$ denote the resulting dual equivalence
$$\Dmod_{-\kappa'-\on{shift}}(\Bun_T)\to \Dmod_{-\kappa^{-1}-\on{trans}}(\on{Pic}(X)\underset{\BZ}\otimes \cLambda).$$

\sssec{}   \label{sss:Heis}

Since $\kappa$ was assumed integral, the twisting $\Dmod_{-\kappa'-\on{shift}}(\Bun_T)$ is also integral, i.e.,
we have a canonical equivalence
\begin{equation}  \label{e:integral untwist T}
\Dmod(\Bun_T)\to \Dmod_{-\kappa'-\on{shift}}(\Bun_T),
\end{equation}
given by tensoring with the line bundle $\CL_{T,-\kappa'-\on{shift}}$, the latter being defined by \eqref{e:L shift}. 

\medskip

We are finally able to give the definition of the object 
$$\CE_{q^{-1}}\in \on{Shv}_{\CG_{q,\on{glob}}}(\on{Pic}(X)\underset{\BZ}\otimes \cLambda).$$

Namely, it is defined so that its Riemann-Hilbert image
$$\CE_{-\kappa^{-1}-\on{trans}}\in  \Dmod_{-\kappa'-\on{shift}}(\Bun_T)$$
equals
$$\on{FM}'_{\on{glob}}(\omega_{\Bun_T}\otimes \CL_{T,-\kappa'-\on{shift}}).$$

\sssec{}

Thus, we can rewrite the functor \eqref{e:rewrite RH} as the composition 
\begin{multline} \label{e:rewrite FM}
\hg_{\kappa',x_1}\mod^{G(\CO_{x_1})}\otimes...\otimes \hg_{\kappa',x_n}\mod^{G(\CO_{x_n})}\overset{\KL_G}\longrightarrow  \\
\to (\fU_q(G)\mod)_{T_{x_1}(X)}\otimes...\otimes (\fU_q(G)\mod)_{T_{x_n}(X)} 
\overset{\BFS^{\on{Dmod}}_{\fu_q}\circ \Res^{\on{big}\to \on{small}}}\longrightarrow \\
\to \Dmod_{\kappa^{-1}+\on{trans}}(\Ran(X,\cLambda)) 
\overset{\on{AJ}_!}\longrightarrow  \Dmod_{\kappa^{-1}+\on{trans}}(\on{Pic}(X)\underset{\BZ}\otimes \cLambda) 
\overset{\on{FM}_{\on{glob}}}\longrightarrow \\
\to \Dmod_{\kappa'+\on{shift}}(\Bun_T) \overset{-\otimes \CL_{T,-\kappa'-\on{shift}}}\longrightarrow 
\Dmod(\Bun_T) \overset{\Gamma_\dr(\Bun_T,-)}\longrightarrow \Vect.
\end{multline}

Thus, we obtain that we can rewrite the statement of \conjref{c:main} as saying that there exists a canonical isomorphism of functors
between \eqref{e:from big two} and \eqref{e:rewrite FM}. 

\sssec{} \label{sss:local reform}

In particular, we obtain \conjref{c:main} follows from the following stronger statement, namely, that the following two functors
$$\hg_{\kappa',x_1}\mod^{G(\CO_{x_1})}\otimes...\otimes \hg_{\kappa',x_n}\mod^{G(\CO_{x_n})}\to \Dmod_{\kappa'+\on{shift}}(\Bun_T)$$
are canonically isomorphic:
\begin{multline} \label{e:from big two local}
\hg_{\kappa',x_1}\mod^{G(\CO_{x_1})}\otimes...\otimes \hg_{\kappa',x_n}\mod^{G(\CO_{x_n})}
\overset{\Loc_{G,\kappa',x_1,...,x_n}}\longrightarrow \Dmod_{\kappa'}(\Bun_G)_{\on{co}} \to \\
\overset{\on{CT}_{\kappa'+\on{shift},!*}}\longrightarrow \Dmod_{\kappa'+\on{shift}}(\Bun_T)
\end{multline} 
and
\begin{multline} \label{e:rewrite FM new}
\hg_{\kappa',x_1}\mod^{G(\CO_{x_1})}\otimes...\otimes \hg_{\kappa',x_n}\mod^{G(\CO_{x_n})}\overset{\KL_G}\longrightarrow  \\
\to (\fU_q(G)\mod)_{T_{x_1}(X)}\otimes...\otimes (\fU_q(G)\mod)_{T_{x_n}(X)} 
\overset{\BFS^{\on{Dmod}}_{\fu_q}\circ \Res^{\on{big}\to \on{small}}}\longrightarrow \\
\to \Dmod_{\kappa^{-1}+\on{trans}}(\Ran(X,\cLambda)) 
\overset{\on{AJ}_!}\longrightarrow  \Dmod_{\kappa^{-1}+\on{trans}}(\on{Pic}(X)\underset{\BZ}\otimes \cLambda) 
\overset{\on{FM}_{\on{glob}}}\longrightarrow \\
\to \Dmod_{\kappa'+\on{shift}}(\Bun_T). 
\end{multline}

I.e., that the diagram
$$
\CD
\hg_{\kappa',x_1}\mod^{G(\CO_{x_1})}\otimes...\otimes \hg_{\kappa',x_n}\mod^{G(\CO_{x_n})}   @>{\Loc_{G,\kappa',x_1,...,x_n}}>> 
 \Dmod_{\kappa'}(\Bun_G)_{\on{co}}   \\
@V{\KL_G}VV       @VV{\on{CT}_{\kappa'+\on{shift},!*}}V     \\
(\fU_q(G)\mod)_{T_{x_1}(X)}\otimes...\otimes (\fU_q(G)\mod)_{T_{x_n}(X)}  & &    \Dmod_{\kappa'+\on{shift}}(\Bun_T)  \\
@V{\BFS^{\on{Dmod}}_{\fu_q}\circ \Res^{\on{big}\to \on{small}}}VV       @AA{\on{FM}_{\on{glob}}}A   \\
\Dmod_{\kappa^{-1}+\on{trans}}(\Ran(X,\cLambda)) @>{\on{AJ}_!}>>   \Dmod_{\kappa^{-1}+\on{trans}}(\on{Pic}(X)\underset{\BZ}\otimes \cLambda)
\endCD
$$
commutes.

\ssec{Local Fourier-Mukai transform} 

Our next step is to interpret the composition
\begin{multline*} 
\hg_{\kappa',x_1}\mod^{G(\CO_{x_1})}\otimes...\otimes \hg_{\kappa',x_n}\mod^{G(\CO_{x_n})}\overset{\KL_G}\longrightarrow  \\
\to (\fU_q(G)\mod)_{T_{x_1}(X)}\otimes...\otimes (\fU_q(G)\mod)_{T_{x_n}(X)} 
\overset{\BFS^{\on{Dmod}}_{\fu_q}\circ \Res^{\on{big}\to \on{small}}}\longrightarrow \\
\to \Dmod_{\kappa^{-1}+\on{trans}}(\Ran(X,\cLambda)) 
\overset{\on{AJ}_!}\longrightarrow  \Dmod_{\kappa^{-1}+\on{trans}}(\on{Pic}(X)\underset{\BZ}\otimes \cLambda) 
\overset{\on{FM}_{\on{glob}}}\longrightarrow \\
\to \Dmod_{\kappa'+\on{shift}}(\Bun_T)
\end{multline*}
appearing as part of the functor \eqref{e:rewrite FM}
in terms of the localization functor for the Kac-Moody Lie algebra associated to the group $T$, and the corresponding functor $\on{KL}_T$. 

\medskip

The version of the Kac-Moody algebra for $T$ that we will consider is $\htt_{\kappa'+\on{shift}}$, introduced in \secref{sss:shifted KM}. 

\sssec{}

We consider the category $\htt_{\kappa'+\on{shift}}\mod^{T(\CO)}$ as a \emph{factorization category}. In particular, we can consider
the corresponding category over the Ran space
$$(\htt_{\kappa'+\on{shift}}\mod^{T(\CO)})_{\Ran(X)},$$
and the localization functor
$$\Loc_{T,\kappa'+\on{shift},\Ran(X)}: (\htt_{\kappa'+\on{shift}}\mod^{T(\CO)})_{\Ran(X)} \to \Dmod_{\kappa'+\on{shift}}(\Bun_T).$$

\sssec{}

The key observation is that we have the following (nearly tautological) equivalence of categories
$$\on{FM}_{\on{loc}}:\Dmod_{\kappa^{-1}+\on{trans}}(\Ran(X,\cLambda))\simeq (\htt_{\kappa'+\on{shift}}\mod^{T(\CO)})_{\Ran(X)},$$
which is compatible with the factorization structure, and makes the following diagram commute:

$$
\CD 
\Dmod_{\kappa^{-1}+\on{trans}}(\Ran(X,\cLambda))   @>{\on{FM}_{\on{loc}}}>> (\htt_{\kappa'+\on{shift}}\mod^{T(\CO)})_{\Ran(X)} \\
@V{\on{AJ}_!}VV       @VV{\Loc_{T,\kappa'+\on{shift},\Ran(X)}}V         \\
\Dmod_{\kappa^{-1}+\on{trans}}(\on{Pic}(X)\underset{\BZ}\otimes \cLambda)   @>{\on{FM}_{\on{glob}}}>> \Dmod_{\kappa'+\on{shift}}(\Bun_T).
\endCD
$$

\sssec{}  \label{sss:inv to small quantum}

We now consider the Kazhdan-Lusztig equivalence for the group $T$:
$$\KL_T:  \htt_{\kappa'+\on{shift}}\mod^{T(\CO)}\simeq  \Rep_q(T).$$

\medskip

Let $\on{Inv}_{\fn(\CK),!*}$ denote the \emph{factorizable} functor
$$\hg_{\kappa'}\mod^{G(\CO)}\to \htt_{\kappa'+\on{shift}}\mod^{T(\CO)}$$  that makes the diagram
$$
\CD
\hg_{\kappa'}\mod^{G(\CO)}  @>{\on{Inv}_{\fn(\CK),!*}}>>   \htt_{\kappa'+\on{shift}}\mod^{T(\CO)} \\
@V{\KL_G}VV    @VV{\KL_T}V   \\
\fU_q(G)\mod     @>{\on{Inv}_{\fu_q(N^+)} \circ \Res^{\on{big}\to \on{small}}}>>  \Rep_q(T) 
\endCD
$$
commute, where $\on{Inv}_{\fu_q(N^+)}$ is as in \eqref{e:Inv}.

\sssec{}

Let
$$((\on{Inv}_{\fn(\CK),!*})_{\Ran(X)}:(\hg_{\kappa'}\mod^{G(\CO)})_{\Ran(X)}\to  (\htt_{\kappa'+\on{shift}}\mod^{T(\CO)})_{\Ran(X)}$$
denote the resulting functor between the corresponding categories over the Ran space. 

\medskip

Using the interpretation of the functor $\BFS^{\on{top}}_{\fu_q}$, given in \secref{ss:BFS via Koszul}, we obtain the following commutative diagram
$$
\CD
\hg_{\kappa',x_1}\mod^{G(\CO_{x_1})}\otimes...\otimes \hg_{\kappa',x_n}\mod^{G(\CO_{x_n})}    @>>> (\hg_{\kappa'}\mod^{G(\CO)})_{\Ran(X)}   \\
@V{\KL_G}VV  \\
(\fU_q(G)\mod)_{T_{x_1}(X)}\otimes...\otimes (\fU_q(G)\mod)_{T_{x_n}(X)} & & @VV{((\on{Inv}_{\fn(\CK),!*})_{\Ran(X)}}V   \\
@V{\BFS^{\on{Dmod}}_{\fu_q}\circ \Res^{\on{big}\to \on{small}}}VV   \\
\Dmod_{\kappa^{-1}+\on{trans}}(\Ran(X,\cLambda))  @>{\on{FM}_{\on{loc}}}>>  (\htt_{\kappa'+\on{shift}}\mod^{T(\CO)})_{\Ran(X)}.
\endCD
$$

\sssec{}

Taking into account the fact that the functor
$$\hg_{\kappa',x_1}\mod^{G(\CO_{x_1})}\otimes...\otimes \hg_{\kappa',x_n}\mod^{G(\CO_{x_n})}
\overset{\Loc_{G,\kappa',x_1,...,x_n}}\longrightarrow \Dmod_{\kappa'}(\Bun_G)_{\on{co}}$$ is isomorphic to the composition
\begin{multline*}
\hg_{\kappa',x_1}\mod^{G(\CO_{x_1})}\otimes...\otimes \hg_{\kappa',x_n}\mod^{G(\CO_{x_n})}\to (\hg_{\kappa'}\mod^{G(\CO)})_{\Ran(X)} 
\overset{\Loc_{G,\kappa',\Ran(X)}}\longrightarrow \\ 
\to \Dmod_{\kappa'}(\Bun_G)_{\on{co}},
\end{multline*} 
and using \secref{sss:local reform}, we obtain that \conjref{c:main} follows from the following one:

\begin{conj} \label{c:deduction}
The following diagram of functors commutes
$$
\CD
(\hg_{\kappa'}\mod^{G(\CO)})_{\Ran(X)} @>{((\on{Inv}_{\fn(\CK),!*})_{\Ran(X)}}>>   (\htt_{\kappa'+\on{shift}}\mod^{T(\CO)})_{\Ran(X)}  \\
@V{\Loc_{G,\kappa',\Ran(X)}}VV   @VV{\Loc_{T,\kappa'+\on{shift},\Ran(X)}}V    \\
\Dmod_{\kappa'}(\Bun_G)_{\on{co}}   @>{\on{CT}_{\kappa'+\on{shift},!*}}>>    \Dmod_{\kappa'+\on{shift}}(\Bun_T).
\endCD
$$
\end{conj} 

\begin{rem}  \label{r:algebraic}

All we have done so far was push the content of \conjref{c:main} 
into the understanding of the functor $\on{Inv}_{\fn(\CK),!*}$: on the one hand, it was defined via 
the Kazhdan-Lusztig functors $\on{KL}_G$ and $\on{KL}_T$, and on the other hand we must relate
it to the functor of Eisenstein series.

\medskip

As we shall see (see \secref{ss:BRST vs KL}), even though the definition of $\on{Inv}_{\fn(\CK),!*}$ involves a transcendental 
procedure (the functors $\on{KL}_G$ and $\on{KL}_T$), it is actually
algebraic in nature.

\end{rem}

\ssec{Other versions of the functor of invariants}

\sssec{}

Parallel to $\Omega^{\on{small}}_{\kappa^{-1}+\on{trans}} $, we also have the factorization algebras 
$$\Omega^{\on{KD}}_{\kappa^{-1}+\on{trans}} \text{ and } \Omega^{\on{Lus}}_{\kappa^{-1}+\on{trans}}\in\Dmod_{\kappa^{-1}+\on{trans}}(\Ran(X,\cLambda)),$$
and the functors
$$\BFS^{\on{Dmod}}_{\fU^{\on{KD}}_q}\circ \Res^{\on{big}\to \on{KD}} \text{ and }
\BFS^{\on{Dmod}}_{\fU^{\on{Lus}}_q}\circ \Res^{\on{big}\to \on{Lus}}$$
both mapping 
$$(\fU_q(G)\mod)_{T_{x_1}(X)}\otimes...\otimes (\fU_q(G)\mod)_{T_{x_n}(X)}\to\Dmod_{\kappa^{-1}+\on{trans}}(\Ran(X,\cLambda)).$$

\sssec{}  \label{sss:inv ! and *}

Let
$\on{Inv}_{\fn(\CK),*}$ denote the factorizable functor
$$\hg_{\kappa'}\mod^{G(\CO)}\to \htt_{\kappa'+\on{shift}}\mod^{T(\CO)}$$  that makes the diagram
$$
\CD
\hg_{\kappa'}\mod^{G(\CO)}  @>{\on{Inv}_{\fn(\CK),*}}>>   \htt_{\kappa'+\on{shift}}\mod^{T(\CO)} \\
@V{\KL_G}VV    @VV{\KL_T}V   \\
\fU_q(G)\mod     @>{\on{Inv}_{\fU_q(N^+)^{\on{KD}}} \circ \Res^{\on{big}\to \on{KD}}}>>  \Rep_q(T) 
\endCD
$$
commute. 

\medskip

Let
$\on{Inv}_{\fn(\CK),!}$ denote the factorizable functor
$$\hg_{\kappa'}\mod^{G(\CO)}\to \htt_{\kappa'+\on{shift}}\mod^{T(\CO)}$$  that makes the diagram
$$
\CD
\hg_{\kappa'}\mod^{G(\CO)}  @>{\on{Inv}_{\fn(\CK),!}}>>   \htt_{\kappa'+\on{shift}}\mod^{T(\CO)} \\
@V{\KL_G}VV    @VV{\KL_T}V   \\
\fU_q(G)\mod     @>{\on{Inv}_{\fU_q(N^+)^{\on{Lus}}} \circ \Res^{\on{big}\to \on{Lus}}}>>  \Rep_q(T) 
\endCD
$$
commute. 

\sssec{}  \label{sss:inv to quantum}

We obtain the corresponding functors
$$(\on{Inv}_{\fn(\CK),*})_{\Ran(X)} \text{ and }(\on{Inv}_{\fn(\CK),!})_{\Ran(X)},$$
both of which map 
$$(\hg_{\kappa'}\mod^{G(\CO)})_{\Ran(X)}\to  (\htt_{\kappa'+\on{shift}}\mod^{T(\CO)})_{\Ran(X)}.$$

\medskip

We obtain that \conjref{c:main ! and *} follows from the next one:

\begin{conj}   \label{c:deduction ! and *}  \hfill

\smallskip

\noindent{\em(a)}
The following diagram of functors commutes:
$$
\CD
(\hg_{\kappa'}\mod^{G(\CO)})_{\Ran(X)} @>{((\on{Inv}_{\fn(\CK),*})_{\Ran(X)}}>>   (\htt_{\kappa'+\on{shift}}\mod^{T(\CO)})_{\Ran(X)}  \\
@V{\Loc_{G,\kappa',\Ran(X)}}VV   @VV{\Loc_{T,\kappa'+\on{shift},\Ran(X)}}V    \\
\Dmod_{\kappa'}(\Bun_G)_{\on{co}}   @>{\on{CT}_{\kappa'+\on{shift},*}}>>    \Dmod_{\kappa'+\on{shift}}(\Bun_T).
\endCD
$$

\smallskip

\noindent{\em(b)}
The following diagram of functors commutes:
$$
\CD
(\hg_{\kappa'}\mod^{G(\CO)})_{\Ran(X)} @>{((\on{Inv}_{\fn(\CK),!})_{\Ran(X)}}>>   (\htt_{\kappa'+\on{shift}}\mod^{T(\CO)})_{\Ran(X)}  \\
@V{\Loc_{G,\kappa',\Ran(X)}}VV   @VV{\Loc_{T,\kappa'+\on{shift},\Ran(X)}}V   \\
\Dmod_{\kappa'}(\Bun_G)_{\on{co}}   @>{\on{CT}_{\kappa'+\on{shift},!}}>>    \Dmod_{\kappa'+\on{shift}}(\Bun_T).
\endCD
$$

\end{conj}

\begin{rem}   \label{r:nat trans deduction}
The natural transformations \eqref{e:BFS nat trans} induce 
natural transformations
$$\BFS^{\on{Dmod}}_{\fU^{\on{Lus}}_q}\circ \Res^{\on{big}\to \on{Lus}} \to
\BFS^{\on{Dmod}}_{\fu_q}  \circ \Res^{\on{big}\to \on{small}} \to \BFS^{\on{Dmod}}_{\fU^{\on{KD}}_q}\circ \Res^{\on{big}\to \on{KD}}$$
and also natural transformations
$$\on{Inv}_{\fn(\CK),!}\to \on{Inv}_{\fn(\CK),!*}\to \on{Inv}_{\fn(\CK),*},$$
the latter as factorizable functors
$$\hg_{\kappa'}\mod^{G(\CO)}\to \htt_{\kappa'+\on{shift}}\mod^{T(\CO)}.$$

We will see that the above natural transformations are compatible with the natural transformations
$$\Eis_{!}\to \Eis_{!*}\to \Eis_{*}$$
via the isomorphisms of functors of
Conjectures \ref{c:deduction} and \ref{c:deduction ! and *}.

\end{rem} 

\section{The semi-infinite flag space}    \label{s:semi-inf}
In the previous section we replaced the Tilting Conjecture (\conjref{c:tilting conj}) by a more general statement,
\conjref{c:main}, and subsequently reduced the latter to \conjref{c:deduction}. 

\medskip

In order to tackle \conjref{c:deduction}, we need to understand the functor 
$$\on{Inv}_{\fn(\CK),!*}:\hg_{\kappa'}\mod^{G(\CO)}\to \htt_{\kappa'+\on{shift}}\mod^{T(\CO)}.$$ 

\medskip

In this section we will show how to produce functors
$\hg_{\kappa'}\mod^{G(\CO)}\to \htt_{\kappa'+\on{shift}}\mod^{T(\CO)}$ starting from objects of the category of D-modules on
\emph{the semi-infinite flag space} of $G$. Our functor $\on{Inv}_{\fn(\CK),!*}$ will correspond to some particular object of this category
(as will do the functors $\on{Inv}_{\fn(\CK),*}$ and $\on{Inv}_{\fn(\CK),!}$). 

\ssec{The category of D-modules on the semi-infinite flag space}

Morally, the semi-infinite flag space of $G$ is the quotient $N(\CK)\backslash G(\CK)$, and $G(\CO)$-equivariant
D-modules on this space should be D-modules on the double quotient $N(\CK)\backslash G(\CK)/G(\CO)$. 

\medskip

Unfortunately, we still do not know how to make sense of $N(\CK)\backslash G(\CK)$ as an algebro-geometric object
so that the category of D-modules on it is defined a priori. Instead, we will define spherical D-modules on it by first considering
D-modules on the affine Grassmannian
$$\Gr_G:=G(\CK)/G(\CO),$$
and then imposing an equivariance condition with respect to $N(\CK)$. 

\sssec{}

Consider the affine Grassmannian $\Gr_G$, the category $\Dmod(\Gr_G)$ and its twisted version
$$\Dmod_{\kappa'}(\Gr_G).$$

The latter category is equipped with an action of the group $G(\CK)$ (at level $\kappa'$), and in particular, 
of the group $N(\CK)$.

\medskip

For the purposes of this paper we will be interested in the category
$$\CC_{\kappa'}:=\Dmod_{\kappa'}(\Gr_G)_{N(\CK)}$$
of $N(\CK)$-\emph{coinvariants} on $\Dmod(\Gr_G)_{\kappa'}$.

\medskip

Note that $\CC_{\kappa'}$ carries an action of $T(\CO)$ and we will also consider the category of $T(\CO)$-equivariant objects
$$\CC_{\kappa'}^{T(\CO)}.$$

\sssec{}

By definition,
\begin{equation} \label{e:coinvariants}
\Dmod_{\kappa'}(\Gr_G)_{N(\CK)}:=\underset{i}{colim}\, \Dmod_{\kappa'}(\Gr_G)^{N_i},
\end{equation}
where $N_i$ is a family of \emph{group-schemes} such that $N(\CK)=\underset{i}{colim}\, N_i$,
and where in \eqref{e:coinvariants} the transition functors
$$\Dmod_{\kappa'}(\Gr_G)^{N_i}\to  \Dmod_{\kappa'}(\Gr_G)^{N_j}, \quad N_i\subset N_j$$
are given by *-averaging over $N_j/N_i$.

\medskip

In other words, since each $N_i$ is a group-scheme (rather than a group ind-scheme), 
we can think of $\Dmod_{\kappa'}(\Gr_G)^{N_i}$ as $\Dmod_{\kappa'}(\Gr_G)_{N_i}$ and in this interpretation
the transition functors are the tautological projections
$$\Dmod_{\kappa'}(\Gr_G)_{N_i}\to  \Dmod_{\kappa'}(\Gr_G)_{N_j}.$$

\begin{rem}
Another version of the category of ($G(\CO)$-equivariant) D-modules on the semi-infinite flag space is
$$\Dmod_{\kappa'}(\Gr_G)^{N(\CK)}:=\underset{i}{lim}\, \Dmod_{\kappa'}(\Gr_G)^{N_i},$$
where the transition functors
$$\Dmod_{\kappa'}(\Gr_G)^{N_i}\leftarrow \Dmod_{\kappa'}(\Gr_G)^{N_j}$$
are the forgetful functors. 

\medskip

It is not difficult to see that $\Dmod_{\kappa'}(\Gr_G)_{N(\CK)}$ and $\Dmod_{\kappa'}(\Gr_G)^{N(\CK)}$ are
duals of each other. 

\end{rem}

\sssec{}  \label{sss:with x}

For any $x\in X$ we can consider the version of $\CC_{\kappa'}$ 
with $\CO$ replaced by $\CO_x$; we denote the resulting category
by $\CC_{\kappa',x}$.  We can view $\CC_{\kappa'}$ as 
a \emph{unital} factorization category (see \cite[Sect. 6]{Ras3} for what this means); let
$$\CC_{\kappa',\Ran(X)}$$
denote the corresponding category over the Ran space. 

\medskip

We let 
$${\bf 1}_{\CC_{\kappa'}}\in \CC_{\kappa'} \text{ and } {\bf 1}_{\CC_{\kappa',\Ran(X)}}\in \CC_{\kappa',\Ran(X)}$$
denote the corresponding unit objects.

\medskip

By definition, ${\bf 1}_{\CC_{\kappa'}}$ (resp., ${\bf 1}_{\CC_{\kappa',\Ran(X)}}$) is the image of the $\delta$-function
under the canonical projections
$$\Dmod_{\kappa'}(\Gr_G)\to \CC_{\kappa'} \text{ and } \Dmod_{\kappa'}((\Gr_G)_{\Ran(X)})\to \CC_{\kappa',\Ran(X)},$$
respectively. 

\medskip

A similar discussion applies to the $T(\CO)$-equivariant version.

\sssec{}  \label{sss:indep level}

Note that when $\kappa'$ is integral, the category $\CC_{\kappa'}$ along with all its variants, identifies with 
the corresponding non-twisted version, denoted simply by $\CC$.

\ssec{The completion}

\sssec{}

Recall that $\Gr_G$ is stratified by $N(\CK)$-orbits, and the latter are parameterized by
elements of $\Lambda$. We let $(\Gr_G)^{\leq \lambda}$ denote the (closed) union of orbits
with parameters $\leq \lambda$. 

\medskip

For $\lambda\in \Lambda$ we let $(\CC_{\kappa'})^{\leq \lambda}$
denote the full subcategory that consists of objects supported on $(\Gr_G)^{\leq \lambda}$. 

\medskip

We have ${\bf 1}_{\CC_{\kappa'}}\in (\CC_{\kappa'})^{\leq 0}$. 

\sssec{}

The inclusion
$$(\CC_{\kappa'})^{\leq \lambda}\hookrightarrow \CC_{\kappa'}$$
admits a continuous right adjoint. Thus, we obtain a localization sequence
$$(\CC_{\kappa'})^{\leq \lambda}\rightleftarrows \CC_{\kappa'}\rightleftarrows \CC_{\kappa'}/(\CC_{\kappa'})^{\leq \lambda}.$$

\medskip

We let $\ol\CC_{\kappa'}$ denote the category
$$\underset{\Lambda}{lim}\, \CC_{\kappa'}/(\CC_{\kappa'})^{\leq -\lambda},$$
where we regard $\Lambda$ as a poset with respect to the usual order relation.

\medskip

In other words, an object of $\ol\CC_{\kappa'}$ is a system of objects $\CF^\lambda\in \CC_{\kappa'}/(\CC_{\kappa'})^{\leq -\lambda}$
that are compatible in the sense that for $\lambda_1\leq \lambda_2$, the image of $\CF^{\lambda_2}$
under the projection
$$\CC_{\kappa'}/(\CC_{\kappa'})^{\leq -\lambda_2}\to  \CC_{\kappa'}/(\CC_{\kappa'})^{\leq -\lambda_1}$$
identifies with $\CF^{\lambda_1}$.

\sssec{}

For every $\lambda$ we have a full subcategory
$$(\ol\CC_{\kappa'})^{\leq \lambda}\subset \ol\CC_{\kappa'}$$
and a localization sequence
$$(\ol\CC_{\kappa'})^{\leq \lambda}\rightleftarrows \ol\CC_{\kappa'}\rightleftarrows 
\ol\CC_{\kappa'}/(\ol\CC_{\kappa'})^{\leq \lambda},$$
where the tautological functor
$$\CC_{\kappa'}/(\CC_{\kappa'})^{\leq \lambda}\to \ol\CC_{\kappa'}/(\ol\CC_{\kappa'})^{\leq \lambda}$$
is an equivalence. 

\sssec{}

As in \secref{sss:with x}, for $x\in X$ we have the corresponding categories 
$$\ol\CC_{\kappa',x} \text{ and } x\mapsto (\ol\CC_{\kappa',x})^{\leq 0}.$$
Moreover $\ol\CC_{\kappa'}$ and $(\ol\CC_{\kappa'})^{\leq 0}$ are unital factorization categories, and
we let 
$$\ol\CC_{\kappa',\Ran(X)} \text{ and } (\ol\CC_{\kappa',\Ran(X)})^{\leq 0}$$ denote the corresponding
categories over the Ran space. 

\medskip

Furthermore, the above discussion extends to the $T(\CO)$-equivariant case. I.e., we have the categories
$$(\CC^{T(\CO)}_{\kappa'})^{\leq \lambda}\rightleftarrows \CC^{T(\CO)}_{\kappa'}\rightleftarrows 
\CC^{T(\CO)}_{\kappa'}/(\CC^{T(\CO)}_{\kappa'})^{\leq \lambda}$$
and 
$$(\ol\CC^{T(\CO)}_{\kappa'})^{\leq \lambda}\rightleftarrows \ol\CC^{T(\CO)}_{\kappa'}\rightleftarrows 
\ol\CC^{T(\CO)}_{\kappa'}/(\ol\CC^{T(\CO)}_{\kappa'})^{\leq \lambda},$$
and the categories over the Ran space
$$(\CC^{T(\CO)}_{\kappa',\Ran(X)})^{\leq 0}, \,\, \ol\CC^{T(\CO)}_{\kappa',\Ran(X)} \text{ and } (\ol\CC^{T(\CO)}_{\kappa',\Ran(X)})^{\leq 0}.$$

\ssec{The functor of BRST reduction}

A key ingredient in understanding the functor  
$$\on{Inv}_{\fn(\CK),!*}:\hg_{\kappa'}\mod^{G(\CO)}\to \htt_{\kappa'+\on{shift}}\mod^{T(\CO)}$$ 
is a canonically defined functor
$$\on{BRST}^{\on{conv}}_\fn:\CC^{T(\CO)}_{\kappa'}\otimes \hg_{\kappa'}\mod^{G(\CO)}\to 
\htt_{\kappa'+\on{shift}}\mod^{T(\CO)}.$$

\medskip

In this subsection we will describe the construction of this functor.

\sssec{}

The action of $G(\CK)$ on $\hg_{\kappa'}\mod$
defines a functor
$$\Dmod_{\kappa'}(\Gr_G)\otimes \hg_{\kappa'}\mod^{G(\CO)}\to  \hg_{\kappa'}\mod.$$

This functor respects the actions of $G(\CK)$ (at level $\kappa'$), where the $G(\CK)$-action
on the source is via the first factor, i.e.,  $\Dmod_{\kappa'}(\Gr_G)$. In particular, it gives rise to a functor
\begin{equation} \label{e:conv semiinf 1}
(\Dmod_{\kappa'}(\Gr_G)_{N(\CK)})^{T(\CO)}\otimes  \hg_{\kappa'}\mod^{G(\CO)}\to (\hg_{\kappa'}\mod_{N(\CK)})^{T(\CO)}.
\end{equation}

\sssec{}

Consider now the functor of BRST reduction
\begin{equation} \label{e:initial BRST}
\on{BRST}_\fn: \hg_{\kappa'}\mod \to \htt_{\kappa'+\on{shift}}\mod,
\end{equation} 
see \cite[Sect. 3.8]{CHA}. 

\begin{rem}  \label{r:anomaly}
The fact that the target of the functor is the Kac-Moody extension $\htt_{\kappa'+\on{shift}}$
rather than $\htt_{\kappa'-\kappa_{\on{crit}}}$ (or even more naively $\htt_{\kappa'}$) is
the reason we needed to add the Tate extension $\htt_{\on{Tate}(\fn)}$ in \secref{ss:anomalies}
\end{rem} 

\sssec{}

The functor \eqref{e:initial BRST}
is invariant with respect to the $N(\CK)$-action on $\hg_{\kappa'}$, and hence, gives rise to a functor
$$\hg_{\kappa'}\mod_{N(\CK)} \to \htt_{\kappa'+\on{shift}}\mod.$$

The latter functor respects the action of $T(\CO)$, and thus gives rise to a functor
\begin{equation} \label{e:conv semiinf 2}
(\hg_{\kappa'}\mod_{N(\CK)})^{T(\CO)}\to \htt_{\kappa'+\on{shift}}\mod^{T(\CO)}.
\end{equation}

\medskip

Finally, composing \eqref{e:conv semiinf 1} and \eqref{e:conv semiinf 2}, we obtain the desired functor $\on{BRST}^{\on{conv}}_\fn$
$$\CC^{T(\CO)}_{\kappa'}\otimes \hg_{\kappa'}\mod^{G(\CO)}=
(\Dmod_{\kappa'}(\Gr_G)_{N(\CK)})^{T(\CO)}\otimes  \hg_{\kappa'}\mod^{G(\CO)}\to  \htt_{\kappa'+\on{shift}}\mod^{T(\CO)}.$$

\sssec{}

We have the following basic assertion:

\begin{lem}
Assume that $\kappa'$ is negative. Then the functor $\on{BRST}^{\on{conv}}_\fn$ canonically extends to
a functor
$$\on{BRST}^{\on{conv}}_\fn:
\ol\CC^{T(\CO)}_{\kappa'}\otimes \hg_{\kappa'}\mod^{G(\CO)}\to \htt_{\kappa'+\on{shift}}\mod^{T(\CO)}.$$
\end{lem} 

\begin{rem}
This lemma amounts to the following observation. For $\kappa'$ negative, the restriction of the functor $\on{BRST}^{\on{conv}}_\fn$ to 
$$(\CC^{T(\CO)}_{\kappa'})^{\leq -\lambda}\otimes M,$$
where $M\in \hg_{\kappa'}\mod^{G(\CO)}$ is a given \emph{compact} object
has the property that it maps to a subcategory of $\htt_{\kappa'+\on{shift}}\mod^{T(\CO)}$, consisting of objects
whose $\ft$-weights are of the form $$\cmu_i(M)+\on{Frob}_{\Lambda,\kappa}(\lambda-\Lambda^{\on{pos}}),$$ where $\cmu_i(M)\in \ft^\vee$
in a finite collection of weights that only depends on $M$. 
\end{rem}

\sssec{}

The functor $\on{BRST}^{\on{conv}}_\fn$ is factorizable. We shall denote by $(\on{BRST}^{\on{conv}}_\fn)_{\Ran(X)}$
the corresponding functor
$$\CC^{T(\CO)}_{\kappa',\Ran(X)}\underset{\Dmod(\Ran(X))}\otimes 
(\hg_{\kappa'}\mod^{G(\CO)})_{\Ran(X)}\to (\htt_{\kappa'+\on{shift}}\mod^{T(\CO)})_{\Ran(X)}.$$

Similarly, if $\kappa'$ is negative, we will denote by the same symbol 
the resulting functor
$$\ol\CC^{T(\CO)}_{\kappa',\Ran(X)}\underset{\Dmod(\Ran(X))}\otimes 
(\hg_{\kappa'}\mod^{G(\CO)})_{\Ran(X)}\to (\htt_{\kappa'+\on{shift}}\mod^{T(\CO)})_{\Ran(X)}.$$

\begin{rem}  \label{r:rho for q}
We can now explain the presence of the linear term in the gerbe $\CG_{q,\on{loc}}$ from \secref{sss:gerbe}. 
The actual source is the fact that the target of the functor \eqref{e:initial BRST} is the category of
modules over $\htt_{\kappa'+\on{shift}}$, rather than $\htt_{\kappa'-\kappa_{\on{crit}}}$, the difference being
the abelian extension of $\ft(\CK)$ described in \secref{sss:Miura}. 

\medskip

Since we are dealing with $\htt_{\kappa'+\on{shift}}$, the local Fourier-Mukai transforms implies that
over $\Ran(X,\cLambda)$ we need to consider the category $\Dmod_{\kappa^{-1}+\on{trans}}(\Ran(X,\cLambda))$,
rather than $\Dmod_{(\kappa-\kappa_{\on{crit}})^{-1}}(\Ran(X,\cLambda))$. 

\medskip

The category 
$\Dmod_{\kappa^{-1}+\on{trans}}(\Ran(X,\cLambda))$ is related via 
Riemann-Hilbert to the category of sheaves on $\Ran(X,\cLambda)$, twisted by the gerbe 
$\CG_{q,\on{loc}}$. This is while $\Dmod_{(\kappa-\kappa_{\on{crit}})^{-1}}(\Ran(X,\cLambda))$
corresponds to the category of sheaves twisted by the gerbe that only has the quadratic part.

\end{rem} 

\ssec{Relation to the Kac-Moody equivalence}  \label{ss:BRST vs KL}

In this subsection we will formulate a crucial statement, Quasi-Theorem \ref{t:char KL !*} that will express the functor $\on{Inv}_{\fn(\CK),!*}$
in terms of the functor $\on{BRST}^{\on{conv}}_\fn$. 

\medskip

In particular, this will show that the functor $\on{Inv}_{\fn(\CK),!*}$ is of algebraic nature, as was promised in Remark \ref{r:algebraic}.

\sssec{}

For any level $\kappa'$ we consider the object 
$$\jmath_{\kappa',0,*}:={\bf 1}_{\ol\CC^{G(\CO)}_{\kappa'}}\in \ol\CC^{T(\CO)}_{\kappa'},$$
which is equal to the image of ${\bf 1}_{\CC^{G(\CO)}_{\kappa'}}$ under the tautological projection
$$\CC^{T(\CO)}_{\kappa',\Ran(X)}\to \ol\CC^{T(\CO)}_{\kappa'}.$$

\medskip

Being the unit of $\ol\CC^{T(\CO)}_{\kappa'}$, the object $\jmath_{\kappa',0,*}$ has a natural structure of factorization 
algebra in $\ol\CC^{T(\CO)}_{\kappa'}$, and hence gives rise to an object 
$$(\jmath_{\kappa',0,*})_{\Ran(X)}\in \ol\CC^{T(\CO)}_{\kappa',\Ran(X)},$$
which identifies tautologically with ${\bf 1}_{\ol\CC^{G(\CO)}_{\kappa',\Ran(X)}}$.

\sssec{}

The following result (along with Quasi-Theorem \ref{t:char KL !}) can be viewed as a characterization of the
Kazhdan-Lusztig equivalence:

\begin{qthm} \label{t:char KL *}
Let $\kappa'$ be negative. Then the (factorizable) functor
$$\on{Inv}_{\fn(\CK),*}:\hg_{\kappa'}\mod^{G(\CO)}\to \htt_{\kappa'+\on{shift}}\mod^{T(\CO)}$$
of \secref{sss:inv ! and *} identifies canonically with the (factorizable) functor
$$\on{BRST}^{\on{conv}}_\fn(\jmath_{\kappa',0,*},-):\hg_{\kappa'}\mod^{G(\CO)}\to \htt_{\kappa'+\on{shift}}\mod^{T(\CO)}.$$
\end{qthm}

\sssec{}

In what follows we will denote the above functor $\on{BRST}^{\on{conv}}_\fn(\jmath_{\kappa',0,*},-)$ by $\on{BRST}_{\fn,*}$. Note
that this is essentially the functor $\on{BRST}_\fn$ of \eqref{e:initial BRST} in the sense that we have a commutative diagram
$$
\CD
\hg_{\kappa'}\mod^{G(\CO)}   @>>>  \hg_{\kappa'}\mod \\
@V{\on{BRST}_{\fn,*}}VV     @VV{\on{BRST}_\fn}V   \\
\htt_{\kappa'+\on{shift}}\mod^{T(\CO)} @>>>  \htt_{\kappa'+\on{shift}}\mod.
\endCD
$$

We shall denote by $(\on{BRST}_{\fn,*})_{\Ran(X)}$ the corresponding functor
$$(\hg_{\kappa'}\mod^{G(\CO)})_{\Ran(X)}\to (\htt_{\kappa'+\on{shift}}\mod^{T(\CO)})_{\Ran(X)}.$$

\sssec{}

From now on, until the end of this subsection we will assume that $\kappa'$ is integral. In \secref{s:IC on semi-inf}
we will describe two more factorization algebras in $\ol\CC^{T(\CO)}_{\kappa'}$, denoted 
$$\jmath_{\kappa',0,!} \text{ and } \jmath_{\kappa',0,!*},$$
respectively.  We let
$$(\jmath_{\kappa',0,!})_{\Ran(X)} \text{ and } (\jmath_{\kappa',0,!*})_{\Ran(X)}$$
denote the resulting objects of $\ol\CC^{T(\CO)}_{\kappa',\Ran(X)}$. 

\begin{rem}
Unlike $\jmath_{\kappa',0,*}$, the objects $\jmath_{\kappa',0,!}$ and $\jmath_{\kappa',0,!*}$ do \emph{not} belong 
to the image of the functor $\CC^{T(\CO)}_{\kappa'}\to \ol\CC^{T(\CO)}_{\kappa'}$.
\end{rem}

\sssec{}

We shall denote the resulting (factorizable) functors $\hg_{\kappa'}\mod^{G(\CO)}\to \htt_{\kappa'+\on{shift}}\mod^{T(\CO)}$
$$\on{BRST}^{\on{conv}}_\fn(\jmath_{\kappa',0,!},-)  \text{ and } \on{BRST}^{\on{conv}}_\fn(\jmath_{\kappa',0,!*},-)$$
by $\on{BRST}_{\fn,!}$ and $\on{BRST}_{\fn,!*}$, respectively.

\medskip

We have the following counterparts of Quasi-Theorem \ref{t:char KL *}:

\begin{qthm} \label{t:char KL !}
Let $\kappa'$ be negative. Then the (factorizable) functor
$$\on{Inv}_{\fn(\CK),!}:\hg_{\kappa'}\mod^{G(\CO)}\to \htt_{\kappa'+\on{shift}}\mod^{T(\CO)}$$
of \secref{sss:inv ! and *} identifies canonically with the (factorizable) functor
$$\on{BRST}_{\fn,!}:\hg_{\kappa'}\mod^{G(\CO)}\to \htt_{\kappa'+\on{shift}}\mod^{T(\CO)}.$$
\end{qthm}

\begin{qthm} \label{t:char KL !*}
Let $\kappa'$ be negative. Then the (factorizable) functor
$$\on{Inv}_{\fn(\CK),!*}:\hg_{\kappa'}\mod^{G(\CO)}\to \htt_{\kappa'+\on{shift}}\mod^{T(\CO)}$$
of \secref{sss:inv to small quantum} identifies canonically with the (factorizable) functor
$$\on{BRST}_{\fn,!*}:\hg_{\kappa'}\mod^{G(\CO)}\to \htt_{\kappa'+\on{shift}}\mod^{T(\CO)}.$$
\end{qthm}

\sssec{}

Let $(\on{BRST}_{\fn,!})_{\Ran(X)}$ and $(\on{BRST}_{\fn,!*})_{\Ran(X)}$ denote the resulting functors
$$(\hg_{\kappa'}\mod^{G(\CO)})_{\Ran(X)}\to (\htt_{\kappa'+\on{shift}}\mod^{T(\CO)})_{\Ran(X)}.$$

\medskip

In view of Quasi-Theorem \ref{t:char KL !*}, we can reformulate \conjref{c:deduction} as follows:

\begin{conj} \label{c:deduction semiinf !*}  Let $\kappa'$ be a negative integral level. 
Then the following diagram of functors commutes
$$
\CD
(\hg_{\kappa'}\mod^{G(\CO)})_{\Ran(X)} @>{(\on{BRST}_{\fn,!*})_{\Ran(X)}}>>   (\htt_{\kappa'+\on{shift}}\mod^{T(\CO)})_{\Ran(X)}  \\
@V{\Loc_{G,\kappa',\Ran(X)}}VV   @VV{\Loc_{T,\kappa'+\on{shift},\Ran(X)}}V    \\
\Dmod_{\kappa'}(\Bun_G)_{\on{co}}   @>{\on{CT}_{\kappa'+\on{shift},!*}}>>    \Dmod_{\kappa'+\on{shift}}(\Bun_T).
\endCD
$$
\end{conj} 

In a similar way, we can use Quasi-Theorems \ref{t:char KL *} and Quasi-Theorem \ref{t:char KL !} to reformulate 
\conjref{c:deduction ! and *} by substituting 
$$((\on{Inv}_{\fn(\CK),*})_{\Ran(X)} \text{ and } ((\on{Inv}_{\fn(\CK),!})_{\Ran(X)}$$
by
$$(\on{BRST}_{\fn,*})_{\Ran(X)} \text{ and } (\on{BRST}_{\fn,!})_{\Ran(X)},$$
respectively. 

\begin{rem}  \label{r:j maps}
It will follow from the construction of the objects $\jmath_{\kappa',0,!}$ and $\jmath_{\kappa',0,!*}$ that we have the following
canonical maps of factorization algebras
\begin{equation} \label{e:maps on semi-inf}
\jmath_{\kappa',0,!}  \to \jmath_{\kappa',0,!*}\to \jmath_{\kappa',0,*}.
\end{equation}

In terms of the isomorphisms of Quasi-Theorems \ref{t:char KL *}, \ref{t:char KL !} and \ref{t:char KL !*}, these maps
correspond to the natural transformations 
$$\on{Inv}_{\fn(\CK),!}\to \on{Inv}_{\fn(\CK),!*}\to \on{Inv}_{\fn(\CK),*}.$$

\end{rem}

\ssec{The !-extension}

The contents of this subsection will not be used in the sequel.

\medskip

The general construction of the object $\jmath_{\kappa',0,!}$ will be explained in \secref{s:IC on semi-inf}. Here we indicate
an alternative construction (that works for any $\kappa$, i.e., one that is not necessarily integral). Specifically, we will
describe the object $(\jmath_{\kappa',0,!})_{\Ran(X)}$. 

\sssec{}

Consider the functor 
\begin{equation} \label{e:! to 0}
(\ol\CC^{T(\CO)}_{\kappa',\Ran(X)})^{\leq 0}/(\ol\CC^{T(\CO)}_{\kappa',\Ran(X)})^{<0}
\simeq (\CC^{T(\CO)}_{\kappa',\Ran(X)})^{\leq 0}/(\CC^{T(\CO)}_{\kappa',\Ran(X)})^{<0}\simeq \Dmod(\Ran(X)).
\end{equation}

\medskip

We have:

\begin{lem}  \label{l:existence of !-extension}
The functor in \eqref{e:! to 0} admits a left adjoint.
\end{lem}

\begin{rem}
The existence of the left adjoint in \lemref{l:existence of !-extension} would be \emph{false} if we worked
with the uncompleted category $\CC^{T(\CO)}_{\kappa',\Ran(X)}$ instead of $\ol\CC^{T(\CO)}_{\kappa',\Ran(X)}$.
\end{rem}

\sssec{}

Now, we claim that the object $(\jmath_{\kappa',0,!})_{\Ran(X)}$ is the value of the above left adjoint on
$$\omega_{\Ran(X)}\simeq {\bf 1}_{\Ran(X)}\in \Dmod(\Ran(X)).$$

\section{The IC object on the semi-infinite flag space}  \label{s:IC on semi-inf}

In this section we will give the construction of the objects $\jmath_{\kappa',0,!}$ and $\jmath_{\kappa',0,!*}$
in $\ol\CC^{T(\CO)}_{\kappa'}$ for an integral level $\kappa'$. 

\medskip

Since $\kappa'$ is assumed integral, the category $\ol\CC^{T(\CO)}_{\kappa'}$ is equivalent to
$\ol\CC^{T(\CO)}$ (see \secref{sss:indep level}), so we will consider the latter.

\ssec{The spherical Hecke category for $T$}

\sssec{}

We consider the affine Grassmannian $\Gr_T$ of the group $T$, and the category 
$$\Sph_T:=\Dmod(\Gr_T)^{T(\CO)}.$$ This category acquires a monoidal structure 
given by convolution, and a compatible structure of factorization category.

\medskip

We let $(\Sph_T)_{\Ran(X)}$ denote the corresponding category over the Ran space.

\begin{rem}
The \emph{derived} geometric Satake equivalence gives a description of this category in terms
of the Langlands dual torus, see \cite[Theorem 12.5.3]{AriGa}.
\end{rem}

\sssec{}

Consider now the category $\Dmod(\Gr_T)$. Note that it identifies canonically with 
$\Vect^{\Lambda}$; this identification is the \emph{naive} (i.e., non-derived) 
geometric Satake for the group $T$.

\medskip

The corresponding category over the Ran space $(\Dmod(\Gr_T))_{\Ran(X)}$ identifies canonically with 
$\Dmod(\Ran(X,\Lambda))$. 

\medskip

We have the natural forgetful functors 
\begin{equation} \label{e:forget derived}
\sff:\Sph_T\to \Dmod(\Gr_T) \text{ and } \sff_{\Ran(X)}:(\Sph_T)_{\Ran(X)}\to (\Dmod(\Gr_T))_{\Ran(X)}.
\end{equation}

In particular, it makes sense to talk about objects of $\Sph_T$ (resp., $(\Sph_T)_{\Ran(X)}$) supported
over $\Lambda^{\on{neg}}$ (resp., $\Ran(X,\Lambda)^{\on{neg}}$), see \secref{sss:neg} for the notation.
We denote the corresponding full subcategory by $\Sph_T^{\on{neg}}$ (resp., $(\Sph_T)^{\on{neg}}_{\Ran(X)}$).

\sssec{}

Since the group $T$ is commutative, the category $\Dmod(\Gr_T)$ itself has a natural (symmetric)
monoidal structure, and we have a naturally defined monoidal functor
\begin{equation} \label{e:naive Satake}
\sg:\Dmod(\Gr_T)\to \Sph_T,
\end{equation}
compatible with the factorization structures.

\medskip

The functor $\sg$ is a right inverse of the functor $\sff$ (but note that the 
latter does \emph{not} have a natural monoidal structure). 

\ssec{The Hecke action on the semi-infinite flag space}

\sssec{}

Since the category $\ol\CC^{T(\CO)}$ is obtained by taking $T(\CO)$-invariants in the category $\ol\CC$ acted
on by $T(\CK)$, the category $\Sph_T$ naturally acts on $\ol\CC^{T(\CO)}$ by convolution. 

\medskip

We denote this action by
$$\CS,\CT\mapsto \CS\ast\CT.$$
 
\sssec{}

Recall the object
$$\jmath_{0,*} \in \ol\CC^{T(\CO)}.$$

Let $\CA\in \Sph_T$ be the \emph{universal} algebra object that acts on $\jmath_{0,*}$. In particular,
we have a canonical action map
$$\CA\ast \jmath_{0,*} \to \CA.$$

\medskip

The object $\CA$ has a natural structure of factorization algebra; we denote by $\CA_{\Ran(X)}$ the corresponding
object in $(\Sph_T)_{\Ran(X)}$.

\sssec{}

It is easy to see that $\CA$ is naturally augmented and its augmentation ideal $\CA^+$ 
(resp., $\CA^+_{\Ran(X)}$) belongs in fact to $\Sph_T^{\on{neg}}$ (resp., $(\Sph_T)^{\on{neg}}_{\Ran(X)}$).

\begin{rem}
The object $\CA^+_{\Ran(X)}$ was introduced in \cite[Sect. 6.1.2]{What Acts} under the name $\wt\Omega(\check\fn)$, and 
a description of this object is given in {\it loc.cit.}, Conjecture 10.3.4 in terms of the geometric Satake equivalence. 
Proving this description is work-in-progress by S.~Raskin.
\end{rem} 

\sssec{Construction of the !-extension}

We are finally able to define the sought-for object $$\jmath_{0,!}\in \ol\CC^{T(\CO)}.$$ Namely,
it is defined to be
$$\on{coBar}(\CA^+,\jmath_{0,*}),$$
where $\on{coBar}$ stands for the co-Bar construction for $\CA^+$ (i.e., the co-Bar construction for $\CA$
relative to its augmentation).

\medskip

Note that $\on{coBar}(\CA^+,-)$ involves the procedure of taking the (inverse) limit. Now, one shows
that this inverse limit is equivalent to one over a finite index category when projected to each $\ol\CC^{T(\CO)}/(\ol\CC^{T(\CO)})^{\leq -\lambda}$, 
and hence gives rise to a well-defined object of $\ol\CC^{T(\CO)}$.

\begin{rem}
If we worked with $\CC$ instead of $\ol\CC$, the inverse limit involved in the definition of $\on{coBar}(\CA^+,-)$
would be something unmanageable.
\end{rem} 

\ssec{Definition of the IC object}

\sssec{}

Consider $\sff_{\Ran(X)}(\CA^+_{\Ran(X)})$ as an object of $\Dmod(\Ran(X,\Lambda)^{\on{neg}})$. 

\medskip

It follows from \cite[Sect. 6.1]{What Acts} 
that it belongs to $\Dmod(\Ran(X,\Lambda)^{\on{neg}})^{\geq 0}$, with respect to the natural t-structure; moreover
$$\CA_{0,\Ran(X)}^+:=\tau^{\leq 0}(\sff(\CA^+_{\Ran(X)}))$$
is the object in $\Dmod(\Ran(X,\Lambda)^{\on{neg}})$ associated to a canonically
defined factorization algebra 
$$\CA^+_0\in \Dmod(\Gr_T).$$

\medskip

Furthermore, $\CA^+_0$ is the augementation ideal of a canonically defined (commutative) algebra object $\CA_0$ in
the (symmetric) monoidal category $\Dmod(\Gr_T)$. 

\medskip

Finally, the map $\CA^+_0\to \sff(\CA^+)$ canonically comes from a homomorphism of algebras
$$\sg(\CA^+_0)\to \CA^+,$$
compatible with the factorization structures.

\begin{rem}  \label{r:cl Omega}
The object 
$$(\CA^+_0)_{\Ran(X)}\in \Dmod(\Ran(X,\Lambda)^{\on{neg}})\subset \Dmod(\Ran(X,\Lambda))$$
in fact identifies canonically with the object $\Omega^{\on{Lus}}$ \emph{for the Langlands dual group $\cG$}
and the critical level for $\cG$ (so that the corresponding twisting on $\Ran(X,\Lambda)^{\on{neg}}$
is trivial); see \cite[Sects. 3 and 4]{BG2}.
\end{rem}

\sssec{}

We are finally able to define the object $\jmath_{0,!*}\in \ol\CC^{T(\CO)}$. Namely,
it is defined to be 
$$\on{coBar}(\CA^+_0,\jmath_{0,*}).$$

\begin{rem}
Note that it follows from the construction that we have the canonical maps
$$\jmath_{0,!}\to\jmath_{0,!*}\to \jmath_{0,*},$$
and hence the maps
\begin{equation} \label{e:j maps}
(\jmath_{0,!})_{\Ran(X)}\to (\jmath_{0,!*})_{\Ran(X)}\to (\jmath_{0,*})_{\Ran(X)},
\end{equation}
as promised in Remark \ref{r:j maps}.
\end{rem}

\begin{rem}
The object $\jmath_{0,!*}$ plays the following role: the category of factorization modules over
$(\jmath_{0,!*})_{\Ran(X)}$ in $\ol\CC^{T(\CO)}$ is closely related to the version of
\emph{the category of D-modules on the semi-infinite flag manifold}, expected by Feigin--Frenkel, and whose global 
incarnation was the subject of \cite{FM}. 

\medskip

When we consider this situation at the critical level, the above category 
is related by a localization functor to the category of Kac-Moody representations at the critical level.  
\end{rem} 

\begin{rem}
Note that Quasi-Theorems \ref{t:char KL !} and \ref{t:char KL !*} imply that the functor $\on{Inv}_{\fn(\CK),!}$
can be expressed through the functor $\on{Inv}_{\fn(\CK),!*}$ via the factorization algebra $\CA_0$.

\medskip

Let us observe that this is natural from the point of view of quantum groups. Indeed, according to
Remark \ref{r:cl Omega}, the factorization algebra $\CA_0$ encodes the Chevalley complex of 
$\check\fn$. Now, the precise statement at the level of quantum groups is that for $\CM\in \fU_q(G)\mod$,
the object 
$$\on{Inv}_{\fu_q(N^+)}(\CM)$$
carries an action of $U(\check\fn)$ via the quantum Frobenius, and 
$$\on{Inv}_{\fU^{\on{Lus}}_q(N^+)}(\CM)\simeq \on{Inv}_{U(\check\fn)}(\on{Inv}_{\fu_q(N^+)}(\CM)).$$
\end{rem} 

\section{The semi-infinite flag space vs. Drinfeld's compactification}   \label{s:semi-inf and loc}

Our goal in this section is to deduce \conjref{c:deduction semiinf !*} from another statement,
Quasi-Theorem \ref{t:BRST and global}.

\ssec{The local-to-global map (case of $G/N$)}  \label{ss:loc G/N}

\sssec{}

Consider the stack $\ol{\Bun}_N$. We let $\Dmod_{\kappa}(\ol{\Bun}_N)$ the category of
twisted D-modules on it, where the twisting is the pullback from one on $\Bun_G$ under the
natural projection $\ol{\Bun}_N\to \Bun_G$. 

\medskip

For a point $x\in X$ we have a naturally defined map
$$\phi_x:((\Gr_G)_x)^{\leq 0}\to \ol{\Bun}_N$$
that remembers the reduction of our $G$-bundle to $N$ on $X-x$. 

\medskip

Consider the functor
$$(\phi_x)_*:\Dmod_{\kappa'}(((\Gr_G)_x)^{\leq 0})\to \Dmod_{\kappa}(\ol{\Bun}_N).$$

\begin{rem}
The exchange of levels $\kappa\mapsto \kappa'$ is due to the fact that we are thinking about $(\Gr_G)_x$ as $G(\CK_x)/G(\CO_x)$
(the quotient by $G(\CO_x)$ in the right), while $\Bun_G$ is the quotient of $\Bun_G^{\on{level}_x}$
by $G(\CO_x)$ with the left action.

\medskip

Indeed that according to \cite{AG1}, the $\kappa$-level on $G(\CK)$ with respect to the left action corresponds to
the level $\kappa':=-\kappa-\kappa_{\on{Kil}}$ with respect to the right action. 
\end{rem}

\sssec{}

For a group-scheme $N_i\subset N(\CK_x)$ consider the composed functor
$$(\phi_x)_*\circ \on{Av}^{N_i}_*:\Dmod_{\kappa'}(((\Gr_G)_x)^{\leq 0})\to \Dmod_{\kappa}(\ol{\Bun}_N).$$

These functors form an inverse family:
$$N_i\subset N_j \quad \rightsquigarrow \quad (\phi_x)_*\circ \on{Av}^{N_j}_*\to (\phi_x)_*\circ \on{Av}^{N_i}_*.$$

\begin{lem}
For every compact object $\CF\in \Dmod_{\kappa'}(((\Gr_G)_x)^{\leq 0})$ the family
$$i\mapsto (\phi_x)_*\circ \on{Av}^{N_i}_*(\CF)\in \Dmod_{\kappa}(\ol{\Bun}_N)$$
stabilizes. 
\end{lem}

\sssec{}

Hence, we obtain that the assignment
$$\CF \in \Dmod_{\kappa'}(((\Gr_G)_x)^{\leq 0}) \text{ compact }\rightsquigarrow \text{ eventual value of }  
(\phi_x)_*\circ \on{Av}^{N_i}_*(\CF)\in \Dmod_{\kappa}(\ol{\Bun}_N)$$
gives rise to a contunous $N(\CK_x)$-invariant functor
$$\Dmod_{\kappa'}(((\Gr_G)_x)^{\leq 0})\to \Dmod_{\kappa}(\ol{\Bun}_N),$$
i.e., a functor
\begin{equation} \label{e:almost localization}
(\CC_{\kappa',x})^{\leq 0}\to \Dmod_{\kappa}(\ol{\Bun}_N).
\end{equation}

\medskip

The following results from the definitions:

\begin{lem}
The functor \eqref{e:almost localization} canonically factors through a functor
$$(\ol\CC_{\kappa',x})^{\leq 0}\to \Dmod_{\kappa}(\ol{\Bun}_N).$$
\end{lem}

\sssec{}

We denote the resulting functor 
$$(\ol\CC_{\kappa',x})^{\leq 0}\to \Dmod_{\kappa}(\ol{\Bun}_N)$$
by $\Phi_x$. 

\medskip

By the same token we obtain a functor
$$\Phi_{\Ran(X)}: (\ol\CC_{\kappa',\Ran(X)})^{\leq 0}\to \Dmod_{\kappa}(\ol{\Bun}_N).$$

\ssec{The local-to-global map (case of $G/B$)}

We shall now discuss a variant of the functors $\Phi_x$ and $\Phi_{\Ran(X)}$ above for $\ol\Bun_B$ instead
of $\ol\Bun_N$.

\sssec{}

Note that for $x\in X$ we have the following version of the map $\phi_x$:
$$\phi^{T(\CO_x)}_x: ((\Gr_G)_x)^{\leq 0}/T(\CO_x)\underset{\on{pt}/T(\CO_x)}\times \Bun_T\to \ol\Bun_B,$$
where the map $\Bun_T\to \on{pt}/T(\CO_x)$ is given by restricting a $T$-bundle to the
formal disc around the point $x$.

\medskip

Recall the category $\Dmod_{\kappa,G/T}(\ol{\Bun}_B)$, see \secref{sss:twisting G/B}.
Repeating the construction of \secref{ss:loc G/N} we now obtain a functor
$$\Phi^{T(\CO_x)}_x: (\ol\CC^{T(\CO_x)}_{\kappa',x})^{\leq 0}\to \Dmod_{\kappa,G/T}(\ol{\Bun}_B),$$
and its Ran version
$$\Phi^{T(\CO)}_{\Ran(X)}:(\ol\CC^{T(\CO)}_{\kappa',\Ran(X)})^{\leq 0}\to \Dmod_{\kappa,G/T}(\ol{\Bun}_B).$$

\sssec{}

The following is tautological:

\begin{lem}   \label{l:loc *}
$$\Phi^{T(\CO)}_{\Ran(X)}((\jmath_{\kappa',0,*})_{\Ran(X)})\simeq  \jmath_{\kappa,*}(\omega_{\Bun_B})\in \Dmod_{\kappa,G/T}(\ol\Bun_B).$$
\end{lem}

In addition, we have the following statement that essentially follows from \cite[Sect. 6.1]{What Acts}: 

\begin{prop}   \label{p:Loc on BunBb}  Assume that $\kappa$ is integral. 

\smallskip

\noindent{\em(a)} 
There exists a canonical isomorphism
$$\Phi^{T(\CO)}_{\Ran(X)}((\jmath_{\kappa',0,!})_{\Ran(X)})\simeq  \jmath_{\kappa,!}(\omega_{\Bun_B})\in \Dmod_{\kappa,G/T}(\ol\Bun_B).$$

\smallskip

\noindent{\em(b)} 
There exists a canonical isomorphism
$$\Phi^{T(\CO)}_{\Ran(X)}((\jmath_{\kappa',0,!*})_{\Ran(X)})
\simeq  \jmath_{\kappa,!*}(\IC_{\Bun_B})[\dim(\Bun_B)]\in \Dmod_{\kappa,G/T}(\ol\Bun_B).$$

\end{prop}

\begin{rem}  \label{r:global j maps}
One can show that the maps 
$$(\jmath_{0,!})_{\Ran(X)}\to (\jmath_{0,!*})_{\Ran(X)}\to (\jmath_{0,*})_{\Ran(X)}$$
of \eqref{e:j maps} induce the natural maps
$$\jmath_{\kappa,!}(\omega_{\Bun_B})\to \jmath_{\kappa,!*}(\IC_{\Bun_B})[\dim(\Bun_B)]\to \jmath_{\kappa,*}(\omega_{\Bun_B}).$$
\end{rem}

\ssec{Interaction of the BRST functor with localization}

In the previous sections we have reduced \conjref{c:main} (and hence \conjref{c:tilting conj}) 
to \conjref{c:deduction semiinf !*} (and similarly for \conjref{c:main ! and *}). 

\medskip

In this subsection we will show how  \conjref{c:deduction semiinf !*} follows from a certain general statement, Quasi-Theorem
\ref{t:BRST and global}, that describes the interaction of the functor $\on{BRST}_{\fn}^{\on{conv}}$ 
with the localization functors $\on{Loc}_G$ and $\on{Loc}_T$, respectively.

\sssec{}

Namely, we claim:

\begin{qthm}  \label{t:BRST and global}
Let $\kappa'$ be a negative level. Then the following diagram of functors commutes:
$$
\CD
(\ol\CC^{T(\CO)}_{\kappa',\Ran(X)})^{\leq 0}\underset{\Dmod(\Ran(X))}\otimes 
(\hg_{\kappa'}\mod^{G(\CO)})_{\Ran(X)}  @>{\on{BRST}^{\on{conv}}_\fn}>>   (\hg_{\kappa',x_\infty}\mod^{T(\CO)})_{\Ran(X)}   \\
@V{\Phi^{T(\CO)}_{\Ran(X)}}\otimes \Loc_{G,\kappa',\Ran(X)}VV     @VV{\Loc_{T,\kappa',\Ran(X)}}V     \\
\Dmod_{\kappa,G/T}(\ol\Bun_B) \otimes \Dmod_{\kappa'}(\Bun_G)_{\on{co}} & &  \Dmod_{\kappa'}(\Bun_T) \\
@V{\on{Id}\otimes \ol\sfp^!}VV     @AA{\ol\sfq_*}A   \\   
\Dmod_{\kappa,G/T}(\ol\Bun_B) \otimes \Dmod_{\kappa',G}(\ol\Bun_B)_{\on{co}} 
@>{\overset{!}\otimes}>>  \Dmod_{\kappa',T}(\ol\Bun_B)_{\on{co}}.
\endCD
$$
\end{qthm}

\sssec{}

Let us show how Quasi-Theorem \ref{t:BRST and global} implies \conjref{c:deduction semiinf !*} (the situation with \conjref{c:main ! and *}
will be similar):

\medskip

Let us evaluate the two circuits in the commutative diagram in Quasi-Theorem \ref{t:BRST and global} on
$$(\jmath_{\kappa',0,!*})_{\Ran(X)}\otimes M\in (\ol\CC^{T(\CO)}_{\kappa',\Ran(X)})^{\leq 0}\underset{\Dmod(\Ran(X))}\otimes 
(\hg_{\kappa'}\mod^{G(\CO)})_{\Ran(X)}$$
for $M\in (\hg_{\kappa'}\mod^{G(\CO)})_{\Ran(X)}$.

\medskip

On the one hand, the clockwise circuit gives $\Loc_{G,\kappa',\Ran(X)}(\on{BRST}_{\fn,!*}(M))$, 
by the definition of $\on{BRST}_{\fn,!*}$.

\medskip

On the other hand, applying \propref{p:Loc on BunBb}(b), we obtain that the anti-clockwise circuit gives
$\on{CT}_{\kappa,!*}(\Loc_{G,\kappa',\Ran(X)}(M))$,
as required.

\sssec{}

Note also that Quasi-Theorem \ref{t:BRST and global}, coupled with Remark \ref{r:global j maps}, implies that the natural transformations
$$\Loc_{G,\kappa',\Ran(X)}\circ \on{BRST}_{\fn,!}\to \Loc_{G,\kappa',\Ran(X)}\circ \on{BRST}_{\fn,!*}\to 
\Loc_{G,\kappa',\Ran(X)}\circ \on{BRST}_{\fn,*}$$
that come from the maps \eqref{e:j maps} correspond to the natural transformations
$$\on{CT}_{\kappa,!}\circ \Loc_{G,\kappa',\Ran(X)}\to \on{CT}_{\kappa,!*}\circ \Loc_{G,\kappa',\Ran(X)}\to
\on{CT}_{\kappa,*}\circ \Loc_{G,\kappa',\Ran(X)}$$
that come from the maps
$$\jmath_{\kappa,!}(\omega_{\Bun_B})\to \jmath_{\kappa,!*}(\IC_{\Bun_B})[\dim(\Bun_B)]\to \jmath_{\kappa,*}(\omega_{\Bun_B}),$$ 
as expected (see Remarks \ref{r:nat trans main} and \ref{r:nat trans deduction}).


\begin{thebibliography}{99}

\bibitem[Arkh]{Arkh} S.~Arkhipov, {\it Semin-infinite cohomology of quantum groups}, Comm. Math. Phys. {\bf 188} (1997), 379--405. 

\bibitem[AG1]{AG1} S.~Arkhipov and D.~Gaitsgory, {\em Differential operators on the loop group via chiral algebras}, 
Internat. Math. Res. Notices 2002-4, 165--210.

\bibitem[AG2]{AG2} S.~Arkhipov and D.~Gaitsgory,
{\em Localization and the long intertwining operator for representations of affine Kac-Moody algebras}, 
available at http://www.math.harvard.edu/~gaitsgde/GL/. 

\bibitem[AriGa]{AriGa} D.~Arinkin and D.~Gaitsgory, {\it Singular support of coherent sheaves, and the geometric Langlands conjecture}, 
Selecta Math. New Ser. {\bf 21} (2015), 1--199.

\bibitem[BGS]{BGS} A. ~Beilinson, V. ~Ginzburg, W. ~Soergel, {\em Koszul duality patterns in representation
theory}, J. Amer. Math. Soc. {\bf 9} (1996), 473--527.

\bibitem[BD]{CHA} A. ~Beilinson and V.~Drinfeld, {\em Chiral algebras}, AMS Colloquium Publications {\bf 51}, (2004).

\bibitem[BG1]{BG1} A.~Braverman and D.~Gaitsgory, {\em Geometric Eisenstein series}, Invent. Math. {\bf 150} (2002), 287--84.

\bibitem[BG2]{BG2} A.~Braverman and D.~Gaitsgory, {\em Deformations of local systems and Eisenstein series}, 
GAFA {\bf 17} (2008), 1788--1850.

\bibitem[DrGa1]{DrGa1} V.~Drinfeld and D.~Gaitsgory, {\em On some finiteness questions for algebraic stacks}, GAFA {\bf 23} (2013), 149--294.

\bibitem[DrGa2]{DrGa2} V.~Drinfeld and D.~Gaitsgory, {\em Compact generation of the category of D-modules on the stack of G-bundles on a curve}, 
arXiv:1112.2402

\bibitem[DrGa3]{DrGa3} V.~Drinfeld and D.~Gaitsgory, {\em Geometric constant term functor(s)}, arXiv:1311.2071.

\bibitem[BFS]{BFS} R.~Bezrukavnikov, M.~Finkelberg, V.~Schechtman, {\em Factorizable sheaves and quantum groups},
Lecture Notes in Mathematics {\bf 1691} (1998). 

\bibitem[FFKM]{FFKM}  B.~Feigin, M.~Finkelberg, A.~Kuznetsov, I.~Mirkovic, {\em Semi-infinite FlagsÐII}, The AMS Translations
{\bf 194} (1999), 81--148.

\bibitem[FM]{FM} M.~Finkelberg and I.~Mirkovic, {\em Semi-infinite Flags--I. Case of global curve $\BP^1$},  Differential
topology, infinite-dimensional Lie algebras, and applications, Amer. Math. Soc. Transl. Ser. 2 {\bf 194} (1999), 81--112.

\bibitem[Ga1]{Ga1} D.~Gaitsgory, {\em Twisted Whittaker model and factorizable sheaves},  
Selecta Math. (N.S.) {\bf 13} (2008), 617--659.

\bibitem[Ga2]{Ga2} D.~Gaitsgory,{\em Outline of the proof of the geometric Langlands conjecture for GL(2)}, 
arXiv:1302.2506.

\bibitem[Ga3]{Tam} D.~Gaitsgory, {\em The Atiyah-Bott formula for the cohomology of the moduli space of bundles on a curve}, \newline
available at http://www.math.harvard.edu/~gaitsgde/GL/

\bibitem[Ga4]{FS} D.~Gaitsgory, {\em Notes on factorizable sheaves},
available at http://www.math.harvard.edu/~gaitsgde/GL/

\bibitem[Ga5]{What Acts} D.~Gaitsgory, {\em What acts on geometric Eisenstein series}, \newline  
available at http://www.math.harvard.edu/~gaitsgde/GL/

\bibitem[Ga6]{KM lecture}  D.~Gaitsgory, {\em Kac-Moody representations (Notes for Day 4, Talk 3)}, \newline 
available at https://sites.google.com/site/geometriclanglands2014/notes

\bibitem[GRo]{GRo} D.~Gaitsgory and N.~Rozenblyum, 
{\em Crystals and D-modules}, PAMQ {\bf 10} (2014), 57--155.

\bibitem[KL]{KL} D.~Kazhdan and G.~Lusztig, {\em Tensor structures arising from affine Lie algebras}, JAMS {\bf 6}
(1993), 905--1011 and {\bf 7} (1994), 335--453.

\bibitem[KT]{KT} M.~Kashiwara and T.~Tanisaki, {\em Kazhdan-Lusztig conjecture for symmetrizable Kac-Moody Lie algebras III: positive 
rational case}, Asian J. of Math. {\bf 2} (1998), 779--832. 

\bibitem[Lu]{Lu} J.~Lurie, {\em Higher Algebra}, 
available at http://www.math.harvard.edu/~lurie/

\bibitem[Ras1]{Ras1} S.~Raskin, {\em Factorization algebras (Notes for Day 2, Talk 2)}, \newline 
available at https://sites.google.com/site/geometriclanglands2014/notes

\bibitem[Ras2]{Ras2} S.~Raskin, {\em Factorization categories (Notes for Day 3, Talk 1)}, \newline 
available at https://sites.google.com/site/geometriclanglands2014/notes

\bibitem[Ras3]{Ras3} S.~Raskin, {\em Chiral categories},
available at http://math.mit.edu/~sraskin/

\bibitem[SV1]{SV1} S.~Schechtman and A.~Varchenko, {\em Arrangements of hyperplanes and Lie algebra homology}, Invent. Math. {\bf 106}
(1991), 134--194.

\bibitem[SV2]{SV2} S.~Schechtman and A.~Varchenko, {\em Quantum groups and homology of local systems}, 
ICM satellite conference proceedings: Algebraic geometry and analytic geometry (Tokyo 1990), Springer (1991),  182--191. 

\end{thebibliography}
\end{document}